\documentclass{amsart}

\usepackage{amssymb}
\usepackage[all]{xy}
\usepackage{mathrsfs}
\usepackage{upgreek}
\usepackage{stmaryrd}
\usepackage[colorlinks = true]{hyperref}
\usepackage{rotating}
\usepackage{tikz}

\newcommand{\bk}{\Bbbk}
\newcommand{\Z}{\mathbb{Z}}
\newcommand{\C}{\mathbb{C}}

\newcommand{\F}{\mathbb{F}}

\newcommand{\Gm}{\mathbb{G}_{\mathrm{m}}}
\newcommand{\Ga}{\mathbb{G}_{\mathrm{a}}}

\newcommand{\IC}{\mathcal{IC}}

\newcommand{\Perv}{\mathsf{Perv}}

\newcommand{\Vect}{\mathsf{Vect}}
\newcommand{\Mod}{\mathsf{Mod}}
\newcommand{\Mof}{\mathsf{Mof}}

\newcommand{\nil}{\mathrm{nil}}
\newcommand{\scO}{\mathscr{O}}
\newcommand{\hscO}{\widehat{\mathscr{O}}}

\newcommand{\Db}{D^{\mathrm{b}}}
\newcommand{\Dbet}{D^{\mathrm{b},\mathrm{et}}}
\newcommand{\Kb}{K^{\mathrm{b}}}

\DeclareMathOperator{\Hom}{Hom}
\DeclareMathOperator{\Ext}{Ext}
\DeclareMathOperator{\End}{End}

\DeclareMathOperator{\For}{For}

\newcommand{\id}{\mathrm{id}}
\newcommand{\simto}{\xrightarrow{\sim}}

\newcommand{\Av}{\mathsf{Av}}
\newcommand{\Whit}{\mathrm{Wh}}
\newcommand{\Tilt}{\mathsf{Tilt}}

\newcommand{\hD}{\widehat{D}}
\newcommand{\hDet}{\widehat{D}^{\mathrm{et}}}
\newcommand{\hT}{\widehat{T}}
\newcommand{\hTet}{\widehat{T}^{\mathrm{et}}}
\newcommand{\hP}{\widehat{P}}
\newcommand{\hDel}{\widehat{\Delta}}
\newcommand{\hnab}{\widehat{\nabla}}
\newcommand{\hTil}{\widehat{\mathscr{T}}}
\newcommand{\hL}{\widehat{\mathscr{L}}}

\newcommand{\rig}{\mathrm{r}}
\newcommand{\lef}{\mathrm{l}}
\newcommand{\gr}{\mathrm{gr}}
\newcommand{\Fun}{\mathsf{Fun}}
\newcommand{\hatstar}{\mathbin{\widehat{\star}}}
\newcommand{\wbtimes}{\mathbin{\widetilde{\boxtimes}}}
\newcommand{\pH}{{}^p \hspace{-1.5pt} \mathscr{H}}

\newcommand{\quot}{\hspace{-2.5pt} \mathord{\fatslash} \hspace{3pt}}
\newcommand{\quott}{\hspace{-2.5pt} \mathord{\fatslash} \hspace{1.5pt}}

\newcommand{\bG}{\mathbf{G}}
\newcommand{\bB}{\mathbf{B}}
\newcommand{\bT}{\mathbf{T}}
\newcommand{\bU}{\mathbf{U}}
\newcommand{\bX}{\mathbf{X}}
\newcommand{\bY}{\mathbf{Y}}

\makeatletter
\def\lotimes{\@ifnextchar_{\@lotimessub}{\@lotimesnosub}}
\def\@lotimessub_#1{\mathchoice{\mathbin{\mathop{\otimes}^L}_{#1}}%
  {\otimes^L_{#1}}{\otimes^L_{#1}}{\otimes^L_{#1}}}
\def\@lotimesnosub{\mathbin{\mathop{\otimes}^L}}
\makeatother

\numberwithin{equation}{section}
\newtheorem{thm}{Theorem}[section]
\newtheorem{lem}[thm]{Lemma}
\newtheorem{prop}[thm]{Proposition}
\newtheorem{cor}[thm]{Corollary}

\theoremstyle{definition}
\newtheorem{defn}[thm]{Definition}

\theoremstyle{remark}
\newtheorem{rmk}[thm]{Remark}

\title{A topological approach to Soergel theory}
 
\author{Roman Bezrukavnikov}
\address{Department of Mathematics \\ Massachusetts Institute of Technology \\ Cambridge MA \\ 02139 \\ USA.}
\email{bezrukav@math.mit.edu}
 
\author{Simon Riche}
\address{Universit\'e Clermont Auvergne, CNRS, LMBP, F-63000 Clermont-Ferrand, France.}
\email{simon.riche@uca.fr}

\thanks{
R.B. was partially supported by the NSF grant DMS-1601953. This project has received
funding from the European Research Council (ERC) under the European Union's Horizon 2020
research and innovation programme (grant agreement No 677147).
}

\dedicatory{To Sasha Be{\u\i}linson and Vitya Ginzburg with gratitude and admiration.}

\begin{document}

\begin{abstract}
We develop a ``Soergel theory'' for Bruhat-constructi\-ble perverse sheaves on the flag variety $G/B$ of a complex reductive group $G$, with coefficients in an arbitrary field $\bk$. Namely, we describe the endomorphisms of the projective cover of the skyscraper sheaf in terms of a ``multiplicative'' coinvariant algebra, and then establish an equivalence of categories between projective (or tilting) objects in this category and a certain category of ``Soergel modules'' over this algebra. We also obtain a description of 
the derived category of unipotently $T$-monodromic $\Bbbk$-sheaves on $G/U$ (where $U$, $T\subset B$ are the unipotent radical and the maximal torus), as a monoidal category, in terms of coherent sheaves on the formal neighborhood of the base point in $T^\vee_\bk \times_{(T^\vee_\bk)^W} T^\vee_\bk$, where $T^\vee_\bk$ is the $\bk$-torus dual to $T$.\end{abstract}

\maketitle

\section{Introduction}
\label{sec:intro}

\subsection{Soergel theory}
\label{ss:intro-Soergel-theory}

In~\cite{soergel}, Soergel developed a new approach to the study of the principal block $\mathscr{O}_0$ of the Bernstein--Gelfand--Gelfand category $\mathscr{O}$ of a complex semisimple Lie algebra $\mathfrak{g}$ (with a fixed Borel subalgebra $\mathfrak{b}$ and Cartan subalgebra $\mathfrak{h} \subset \mathfrak{b}$). Namely, let $P$ be the projective cover of the unique simple object in $\mathscr{O}_0$ with antidominant highest weight (in other words, of the unique simple Verma module). 
Then Soergel establishes the following results:
\begin{enumerate}
\item
(Endomorphismensatz) there exists a canonical algebra isomorphism
\[
\mathrm{S}(\mathfrak{h}) / \langle \mathrm{S}(\mathfrak{h})^W_+ \rangle \simto \End(P),
\]
where $W$ is the Weyl group of $(\mathfrak{g},\mathfrak{h})$, $\mathrm{S}(\mathfrak{h})$ is the symmetric algebra of $\mathfrak{h}$, and $\langle \mathrm{S}(\mathfrak{h})^W_+ \rangle$ is the ideal generated by homogeneous $W$-invariant elements of positive degree;
\item
(Struktursatz) the functor $\mathbb{V}:=\Hom_{\mathscr{O}_0}(P,-)$ is fully faithful on projective objects; in other words for any projective objects $Q,Q'$ this functor induces an isomorphism
\[
\Hom_{\mathscr{O}_0}(Q,Q') \simto \Hom_{\End(P)}(\mathbb{V}(Q),\mathbb{V}(Q'));
\]
\item
the essential image of the restriction of $\mathbb{V}$ to projective objects in $\mathscr{O}_0$ is the subcategory generated by the trivial module $\C$ under the operations of (repeatedly) applying the functors $\mathrm{S}(\mathfrak{h}) \otimes_{\mathrm{S}(\mathfrak{h})^s} -$ with $s$ is a simple reflection and taking direct sums and direct summands.
\end{enumerate}
Taken together, these results allow him to describe the category of projective objects in $\mathscr{O}_0$, and hence the category $\mathscr{O}_0$ itself, in terms of commutative algebra (``Soergel modules''). On the other hand, Soergel relates these modules to cohomology of Bruhat-constructible simple perverse sheaves on the Langlands dual flag variety
opening the way to the ideas of \emph{Koszul duality} further developed in his celebrated work with Be{\u\i}linson and Ginzburg~\cite{bgs}.
Another celebrated application of these ideas is Soergel's new proof of the Kazhdan--Lusztig conjecture~\cite{kl} proved earlier by Be{\u\i}linson--Bernstein and 
Brylinsky--Kashiwara. 

\subsection{Geometric version}
\label{ss:intro-geom-version}

If $G$ is the semisimple complex algebraic group of adjoint type whose Lie algebra is $\mathfrak{g}$, and if $B \subset G$ is the Borel subgroup whose Lie algebra is $\mathfrak{b}$, then
combining the Be{\u\i}linson--Bernstein localization theory~\cite{bb} and an equivalence due to Soergel~\cite{soergel-equiv} one obtains that the category $\mathscr{O}_0$ is equivalent to the category $\Perv_{U}(G/B,\C)$ of $U$-equivariant (equivalently, $B$-constructible) $\C$-perverse sheaves on the flag variety $G/B$, where $U$ is the unipotent radical of $B$ (see e.g.~\cite[Proposition~3.5.2]{bgs}). Under this equivalence, the simple Verma module corresponds to the skyskraper sheaf at the base point $B/B$. The main goal of the present paper is to develop a geometric approach to the results in~\S\ref{ss:intro-Soergel-theory}, purely in the framework of perverse sheaves, and moreover valid in the setting where the coefficients can be in an arbitrary field $\bk$ (of possibly positive characteristic) instead of $\C$.

In fact, a geometric proof of the Struktursatz (stated for coefficients of characteristic $0$, but in fact valid in the general case) was already found by Be{\u\i}linson, the first author and Mirkovi\'c in~\cite{bbm}. 
One of the main themes of the latter paper, which is fundamental in our approach too, is an idea introduced by Be{\u\i}linson--Ginzburg in~\cite{bg}, 
namely that it is easier (but equivalent) to work with \emph{tilting} objects in $\mathscr{O}_0$ (or its geometric counterparts) rather than with projective objects. Our main contribution is generalization of the Endomorphismensatz to the present setting; then the description of the essential image of the functor $\mathbb{V}$ follows by rather standard methods.

\subsection{Monodromy}
\label{ss:intro-monodromy}

So, we fix a field $\bk$, and consider the category $\Perv_{U}(G/B,\bk)$ of $U$-equivariant $\bk$-perverse sheaves on the complex variety $G/B$. This category has a natural highest weight structure, with weight poset the Weyl group $W$, and as in the characteristic-$0$ setting the projective cover of the skyskraper sheaf at $B/B$ is also the tilting object associated with the longest element $w_0$ in $W$; we will therefore denote it $\mathscr{T}_{w_0}$. Our first task is then to describe the $\bk$-algebra $\End_{\Perv_{U}(G/B,\bk)}(\mathscr{T}_{w_0})$. 

In the representation-theoretic context studied by Soergel (see~\S\ref{ss:intro-Soergel-theory}), the morphism $\mathrm{S}(\mathfrak{h}) / \langle \mathrm{S}(\mathfrak{h})^W_+ \rangle \simto \End(P)$ is obtained from the action of the center of the enveloping algebra $\mathcal{U} \mathfrak{g}$ on $P$. It has been known for a long time (see e.g.~\cite[\S 4.6]{bgs} or~\cite[Footnote~8 on p.~556]{bbm}) that from the geometric point of view this morphism can be obtained via the logarithm of monodromy for the action of $T$ on $G$. But of course, the logarithm will not make sense over an arbitrary field $\bk$; therefore what we consider here is the monodromy itself, which defines an algebra morphism
\[
\varphi_{\mathscr{T}_{w_0}} : \bk[X_*(T)] \to \End(\mathscr{T}_{w_0}).
\]
We then need to show that:
\begin{enumerate}
\item
the morphism $\varphi_{\mathscr{T}_{w_0}}$ factors through the quotient $\bk[X_*(T)] / \langle \bk[X_*(T)]^W_+ \rangle$, where $\langle \bk[X_*(T)]^W_+ \rangle$ is the ideal generated by $W$-invariant elements in the kernel of the natural augmentation morphism $\bk[X_*(T)] \to \bk$;
\item
the resulting morphism $\bk[X_*(T)] / \langle \bk[X_*(T)]^W_+ \rangle \to \End(\mathscr{T}_{w_0})$ is an isomorphism.
\end{enumerate}

\subsection{Free-monodromic deformation}

To prove these claims we need the second main ingredient of our approach, namely the ``completed category'' defined by Yun in~\cite[Appendix~A]{by}. This category (which is constructed using certain pro-objects in the derived category of sheaves on $G/U$) is a triangulated category endowed with a t-structure, which we will denote $\hD_U((G/U) \quot T, \bk)$, and which contains certain objects whose monodromy is ``free unipotent.'' Killing this monodromy provides a functor to the $U$-equivariant derived category $\Db_U(G/B,\bk)$. The tilting objects in $\Perv_U(G/B,\bk)$ admit ``lifts'' (or ``deformations'') to this category, and we can in particular consider the lift $\hTil_{w_0}$ of $\mathscr{T}_{w_0}$. Now the algebra $\End_{\hD_U((G/U) \quot T, \bk)}(\hTil_{w_0})$ admits \emph{two} morphisms from (the completion $\bk[X_*(T)]^\wedge$ with respect to the augmentation ideal of) $\bk[X_*(T)]$ coming from the monodromy for the left and the right actions of $T$ on $G/U$, and moreover we have a canonical isomorphism
\[
\End(\mathscr{T}_{w_0}) \cong \End(\hTil_{w_0}) \otimes_{\bk[X_*(T)]^\wedge} \bk.
\]
Hence what we have to prove transforms into the following claims:
\begin{enumerate}
\item
the monodromy morphism $\bk[X_*(T)]^\wedge \otimes_\bk \bk[X_*(T)]^\wedge \to \End(\hTil_{w_0})$ factors through $\bk[X_*(T)]^\wedge \otimes_{(\bk[X_*(T)]^\wedge)^W} \bk[X_*(T)]^\wedge$;
\item
the resulting morphism $\bk[X_*(T)]^\wedge \otimes_{(\bk[X_*(T)]^\wedge)^W} \bk[X_*(T)]^\wedge \to \End(\hTil_{w_0})$ is an isomorphism.
\end{enumerate}

\subsection{Identification of \texorpdfstring{$\End(\hTil_{w_0})$}{End(Tw0)}}

One of the main advantages of working with the category $\hD_U((G/U) \quot T, \bk)$ rather than with $\Db_U(G/B,\bk)$ is that the natural lifts $(\hDel_w : w \in W)$ of the standard perverse sheaves satisfy $\Hom(\hDel_x, \hDel_y)=0$ if $x \neq y$. This implies that the functor of ``taking the associated graded for the standard filtration'' is faithful, and we obtain an injective algebra morphism
\begin{equation}
\label{eqn:intro-gr}
\End(\hTil_{w_0}) \to \End(\mathrm{gr}(\hTil_{w_0})).
\end{equation}
Now we have $\mathrm{gr}(\hTil_{w_0}) \cong \bigoplus_{w \in W} \hDel_w$, so that the right-hand side identifies with $\bigoplus_{w \in W} \bk[X_*(T)]^\wedge$. To conclude it remains to identify the image of~\eqref{eqn:intro-gr}; for this we use some algebraic results due to Kostant--Kumar~\cite{kk} (in their study of the K-theory of flag varieties) and Andersen--Jantzen--Soergel~\cite{ajs}.

\subsection{The functor \texorpdfstring{$\mathbb{V}$}{V}}

Once we have identified $\End(\mathscr{T}_{w_0})$ and $\End(\hTil_{w_0})$, we can consider the functor
\[
\mathbb{V}:=\Hom(\mathscr{T}_{w_0},-) : \Perv_U(G/B,\bk) \to \Mod(\bk[X_*(T)] / \langle \bk[X_*(T)]^W_+ \rangle)
\]
and its version $\widehat{\mathbb{V}}$ for free-monodromic perverse sheaves. As explained in~\S\ref{ss:intro-geom-version}, a short argument from~\cite{bbm} shows that these functors are fully faithful on tilting objects. To conclude our study we need to identify their essential image. The main step for this is to show that $\widehat{\mathbb{V}}$ is monoidal. (Here the monoidal structure on tilting objects is given by a ``convolution'' construction, and the monoidal structure on modules over $\bk[X_*(T)]^\wedge \otimes_{(\bk[X_*(T)]^\wedge)^W} \bk[X_*(T)]^\wedge$ is given by tensor product over $\bk[X_*(T)]^\wedge$.)  Adapting an argument of~\cite{by} we show that there exists an isomorphism of bifunctors
\begin{equation}
\label{eqn:intro-isom-monoidal}
\widehat{\mathbb{V}}(- \hatstar -) \cong \widehat{\mathbb{V}}(-) \otimes_{\bk[X_*(T)]^\wedge} \widehat{\mathbb{V}}(-).
\end{equation}
However, constructing a \emph{monoidal} structure (i.e.~an isomorphism compatible with the relevant structures) is a bit harder. In fact we construct such a structure in the similar context of \'etale sheaves on the analogue of $G/B$ (or $G/U$) over an algebraically closed field of positive characteristic, using a ``Whittaker-type'' construction. We then deduce the similar claim in the classical topology over $\C$ using the general formalism explained in~\cite[\S 6.1]{bbd}.

With this at hand, we obtain a description of the monoidal triangulated category $(\hD_U((G/U) \quot T, \bk), \hatstar)$ and its module category $\Db_U(G/B,\bk)$ in terms of coherent sheaves on the formal neighborhood of the point $(1,1)$ in $T^\vee_\bk \times_{(T^\vee_\bk) / W} T^\vee_\bk$ and on the fiber of the quotient morphism $T^\vee_\bk \to (T^\vee_\bk) / W$ over the image of $1$ respectively (where $T^\vee_\bk$ is the split $\bk$-torus which is Langlands dual to $T$);
see Theorem~\ref{thm:soergel-triang}.

\begin{rmk}
Identification of the essential image of $\mathbb{V}$ and $\widehat{\mathbb{V}}$ does not require monoidal structure on  $\widehat{\mathbb{V}}$:
 an isomorphism as in~\eqref{eqn:intro-isom-monoidal} would be sufficient. 
However, description of the monoidal structure  provides a stronger statement.
\end{rmk}

\subsection{Some remarks}
\label{ss:intro-remarks}

We conclude this introduction with a few remarks.

As explained in~\S\ref{ss:intro-monodromy}, in the present paper we work with the group algebra $\bk[X_*(T)]$ and not with the symmetric algebra $\mathrm{S}(\bk \otimes_\Z X_*(T))$ as one might have expected from the known characteristic-$0$ setting. However one can check (see e.g.~\cite[Proposition~5.5]{modrap1}) that if $\mathrm{char}(\bk)$ is very good for $G$ then there exists a $W$-equivariant algebra isomorphism between the completions of $\bk[X_*(T)]$ and $\mathrm{S}(\bk \otimes_\Z X_*(T))$ with respect to their natural augmentation ideals. (In the characteristic-$0$ setting there exists a canonical choice of identification, given by the logarithm; in positive characteristic there exists no ``preferred'' isomorphism.) Therefore, fixing such an isomorphism, under this assumption our results can also be stated in terms of $\mathrm{S}(\bk \otimes_\Z X_*(T))$. An important observation in~\cite{soergel, bgs} is that the identification between $\End(P)$ and the coinvariant algebra allows one to define a grading on $\End(P)$, and then to define a ``graded version'' of $\mathscr{O}_0$. This graded version can be realized geometrically via mixed perverse sheaves (either in the sense of Deligne, see~\cite{bgs}, or in a more elementary sense constructed using semisimple complexes, see~\cite{ar1,modrap2}). When $\mathrm{char}(\bk)$ is not very good, the algebra $\bk[X_*(T)] / \langle \bk[X_*(T)]^W_+ \rangle$ does not admit an obvious grading; we do not know how to interpret this, and the relation with the corresponding category of ``mixed perverse sheaves'' constructed in~\cite{modrap2}. (In very good characteristic, this category indeed provides a ``graded version'' of $\Perv_U(G/B,\bk)$, as proved in~\cite{modrap1,modrap2}.)

As explained already, in the case of characteristic-$0$ coefficients our results are equivalent to those of Soergel in~\cite{soergel}. They are also proved by geometric means in this case in~\cite{by}. In the case of very good characteristic, these methods were extended in~\cite{modrap1} (except for the consideration of the free-monodromic objects). The method we follow here is completely general (in particular, new in bad characteristic), more direct (since it does not involve Koszul duality) and more canonical (since it does not rely on any choice of identification relating $\mathrm{S}(\bk \otimes_\Z X_*(T))$ and $\bk[X_*(T)]$).

In the complex coefficients setting, the category $\Perv_U(G/B,\C)$ has a represen\-tation-theoretic interpretation, in terms of the category $\mathscr{O}_0$. It also admits a representation-theoretic description in the case when $\mathrm{char}(\bk)$ is bigger than the Coxeter number of $G$, in terms of Soergel's modular category $\mathscr{O}$~\cite{soergel-relation}. This fact was first proved in~\cite[Theorem~2.4]{modrap1}; it can also be deduced more directly by comparing the results of~\cite{soergel-relation} and those of the present paper.

\subsection{Contents}

The paper starts with a detailed review of the construction of Yun's ``completed category'' (see~\cite[Appendix~A]{by}) in Sections~\ref{sec:monodromy}--\ref{sec:perverse}. 
More precisely, we adapt his constructions (performed initially for \'etale $\mathbb{Q}_\ell$-complexes) to the setting of sheaves on complex algebraic varieties, with coefficients in an arbitrary field. This adaptation does not require new ideas,
but since the wording in~\cite{by} is quite dense we reproduce most proofs, and propose alternative arguments in a few cases. 

Starting from Section~\ref{sec:flag-tilting} we concentrate on the case of the flag variety. We start by constructing the ``associated graded'' functor. Then in Section~\ref{sec:convolution} we review the construction of the convolution product on $\hD_U((G/U) \quot T, \bk)$ (again, mainly following Yun). In Section~\ref{sec:kostant-kumar} we recall some algebraic results of Kostant--Kumar, and we apply all of this to prove our ``Endomorphismensatz'' in Section~\ref{sec:main-thm}. In Section~\ref{sec:etale} we explain how to adapt our constructions in the setting of \'etale sheaves, and in Section~\ref{sec:soergel-theory} we study the functors $\mathbb{V}$ and $\widehat{\mathbb{V}}$. Finally, in Section~\ref{sec:erratum} we take the opportunity  to correct the proof of a technical lemma in~\cite{ab}.

\subsection{Acknowledgements}

Part of the work on this paper was done while the second author was a member of the Freiburg Institute for Advanced Studies, as part of the Research Focus ``Cohomology in Algebraic Geometry and Representation Theory'' led by A. Huber--Klawitter, S. Kebekus and W. Soergel.

We thank Geordie Williamson for useful discussions on the subject of this paper (in particular for the suggestion to compare with the K-theory of the flag variety), Pramod Achar for useful comments, and Valentin Gouttard for spotting several typos and minor errors.

\part{Reminder on completed categories}

We fix a field $\bk$.

\section{Monodromy}
\label{sec:monodromy}

\subsection{Construction}
\label{ss:monodromy-construction}

We consider a complex algebraic torus $A$ and an $A$-torsor\footnote{All the torsors we will encounter in the present paper will be locally trivial for the Zariski topology.} $\pi : X \to Y$. We then denote by $\Db_c(X \quot A, \bk)$ the full triangulated subcategory of $\Db_c(X,\bk)$ generated by the essential image of the functor $\pi^* : \Db_c(Y,\bk) \to \Db_c(X,\bk)$.

Fix some $\lambda \in X_*(A)$. We then set
\[
 \theta_\lambda : \left\{
 \begin{array}{ccc}
  \C \times X & \to & X \\
  (z,x) & \mapsto & \lambda(\exp(z)) \cdot x
 \end{array}.
 \right.
\]
We will also denote by $\mathrm{pr} : \C \times X \to X$ the projection.

The following claims follow from the considerations in~\cite[\S 9]{verdier}.

\begin{lem}\phantomsection
\label{lem:monodromy-def}
 \begin{enumerate}
  \item For any $\mathscr{F}$ in $\Db_c(X \quot A, \bk)$, there exists a unique morphism
  \[
   \iota^{\lambda}_{\mathscr{F}} : \theta_\lambda^*(\mathscr{F}) \to \mathrm{pr}^*(\mathscr{F})
  \]
whose restriction to $\{0\} \times X$ is $\id_{\mathscr{F}}$.
Moreover, $\iota^{\lambda}_{\mathscr{F}}$ is an isomorphism.
\item
If $\mathscr{F}, \mathscr{G}$ are in $\Db_c(X \quot A, \bk)$ and $f : \mathscr{F} \to \mathscr{G}$ is a morphism, then the following diagram commutes:
\[
 \xymatrix@C=2cm{
 \theta_\lambda^*(\mathscr{F}) \ar[r]^{\iota^\lambda_{\mathscr{F}}} \ar[d]_-{\theta_\lambda^*(f)} & \mathrm{pr}^*(\mathscr{F}) \ar[d]^-{\mathrm{pr}^*(f)} \\
 \theta_\lambda^*(\mathscr{G}) \ar[r]^{\iota^\lambda_{\mathscr{G}}} & \mathrm{pr}^*(\mathscr{G}).
 }
\]
 \end{enumerate}
\end{lem}

\begin{proof}[Sketch of proof]
The essential ingredient of the proof is the (obvious) fact that the functor $\mathrm{pr}^*$ is fully faithful, so that its essential image is a triangulated subcategory of $\Db_c(\C \times X, \bk)$. We see that for any $\mathscr{G}$ in $\Db_c(Y,\bk)$, the object $\theta_\lambda^* \pi^* (\mathscr{G})$ belongs to this essential image; hence for any $\mathscr{F}$ in $\Db_c(X \quot A, \bk)$ the object $\theta_\lambda^* (\mathscr{F})$ is isomorphic to $\mathrm{pr}^*(\mathscr{F}')$ for some $\mathscr{F}'$ in $\Db_c(X,\bk)$. Restricting to $\{0\} \times X$ we obtain an isomorphism $f : \mathscr{F} \simto \mathscr{F}'$, and we can define $\iota^{\lambda}_{\mathscr{F}}$ as the composition $\theta_\lambda^* (\mathscr{F}) \simto \mathrm{pr}^*(\mathscr{F}') \xrightarrow{\mathrm{pr}^*(f^{-1})} \mathrm{pr}^*(\mathscr{F})$.
\end{proof}

Using this lemma and restricting $\iota^\lambda_{\mathscr{F}}$ to $\{2\mathbf{i}\pi\} \times X$ we obtain an automorphism $\varphi^\lambda_{\mathscr{F}}$ of $\mathscr{F}$. This automorphism satisfies the property that if $\mathscr{F}, \mathscr{G}$ are in $\Db_c(X \quot A, \bk)$ and $f : \mathscr{F} \to \mathscr{G}$ is a morphism, then $\varphi^\lambda_{\mathscr{G}} \circ f = f \circ \varphi^\lambda_{\mathscr{F}}$.

For any $\mathscr{F}$ in $\Db_c(X \quot A, \bk)$,
the automorphism $\varphi^\lambda_{\mathscr{F}}$ is unipotent. (In fact, this automorphism is the identity if $\mathscr{F}$ belongs to the essential image of $\pi^*$, and the category $\Db_c(X \quot A, \bk)$ is generated by such objects.) Moreover, if $\lambda,\mu \in X_*(A)$ we have
\[
 \varphi_{\mathscr{F}}^{\lambda \cdot \mu} = \varphi_{\mathscr{F}}^{\lambda} \circ \varphi_{\mathscr{F}}^{\mu}.
\]
In other words, the assignment $\lambda \mapsto \varphi_{\mathscr{F}}^\lambda$ defines a group morphism
\begin{equation}
\label{eqn:monodromy-gps}
 X_*(A) \to \mathrm{Aut}(\mathscr{F}).
\end{equation}

We now set
\[
 R_A := \bk[X_*(A)].
\]
The group morphism~\eqref{eqn:monodromy-gps} induces
a $\bk$-algebra morphism
\[
 \varphi_{\mathscr{F}} : R_A \to \End(\mathscr{F}).
\]
Since each $\varphi_{\mathscr{F}}(\lambda)$ is unipotent, this morphism factors through an algebra morphism
\[
 \varphi_{\mathscr{F}}^\wedge : R_A^\wedge \to \End(\mathscr{F}),
\]
where $R_A^\wedge$ is the completion of $R_A$ with respect to the maximal ideal $\mathfrak{m}_A$ given by the kernel of the algebra map $\varepsilon_A : R_A \to \bk$ sending each $\lambda \in X_*(A)$ to $1$. This construction is functorial, in the sense that it makes $\Db_{c}(X \quot A,\bk)$ an $R_A^\wedge$-linear category. (Here, the $R_A^\wedge$-action on $\Hom_{\Db_{c}(X \quott A,\bk)}(\mathscr{F},\mathscr{G})$ is given by $r\cdot f=f \circ \varphi_{\mathscr{F}}^\wedge(r) = \varphi_{\mathscr{G}}^\wedge(r) \circ f$.)

\begin{rmk}
\label{rmk:RA}
 Geometrically, we have $R_A=\mathcal{O}(A_\bk^\vee)$, where $A_\bk^\vee$ is the $\bk$-torus such that $X^*(A_\bk^\vee) = X_*(A)$, and $R_A^\wedge$ identifies with the algebra of functions on the formal neighborhood of $1$ in $A_\bk^\vee$. Note that any choice of trivialization $A \simto (\C^\times)^r$ provides isomorphisms
 \[
 R_A \cong \bk[y_1^{\pm 1}, \cdots, y_r^{\pm 1}] \quad \text{and} \quad R_A^\wedge \cong \bk [ \hspace{-0.5pt} [x_1, \cdots, x_r] \hspace{-0.5pt} ]
 \]
 (where $x_i = y_i-1$).
\end{rmk}

\subsection{Basic properties}
\label{ss:monodromy-properties}

We denote by $\varepsilon_A^\wedge : R_A^\wedge \to \bk$ the continuous morphism which extends $\varepsilon_A$.

\begin{lem}
\label{lem:monodromy-augmentation}
For any $\mathscr{F}$ in $\Db_c(X \quot A, \bk)$ and $x \in R_A^\wedge$ we have
\[
\pi_!(\varphi^\wedge_{\mathscr{F}}(x)) = \varepsilon_A^\wedge(x) \cdot \id_{\pi_! \mathscr{F}}.
\]
\end{lem}

\begin{proof}
Let $\lambda \in X_*(A)$, and let $p : \C \times Y \to Y$ be the projection. Then both of the following squares are Cartesian:
\[
\xymatrix{
\C \times X \ar[r]^-{\theta_\lambda} \ar[d]_-{\id_\C \times \pi} & X \ar[d]^-{\pi} \\
\C \times Y \ar[r]^-{p} & Y,
}
\qquad
\xymatrix{
\C \times X \ar[r]^-{\mathrm{pr}} \ar[d]_-{\id_\C \times \pi} & X \ar[d]^-{\pi} \\
\C \times Y \ar[r]^-{p} & Y.
}
\]
By the base change theorem, we deduce canonical isomorphisms
\[
(\id_\C \times \pi)_! \theta_\lambda^*(\mathscr{F}) \cong p^* \pi_!(\mathscr{F}), \qquad (\id_\C \times \pi)_! \mathrm{pr}^*(\mathscr{F}) \cong p^* \pi_!(\mathscr{F}).
\]
Under these isomorphisms the map $(\id_\C \times \pi)_! \iota^\lambda_{\mathscr{F}}$ identifies with an endomorphism of $p^* \pi_!(\mathscr{F})$. Now the functor $p^*$ is fully faithful, hence this morphism must be of the form $p^*(f)$ for $f$ an endomorphism of $\pi_! \mathscr{F}$. Restricting to $\{0\} \times Y$ we see that $f=\id_{\pi_! \mathscr{F}}$. Hence the restriction of $(\id_\C \times \pi)_! \iota^\lambda_{\mathscr{F}}$ to $\{2\mathbf{i}\pi\} \times Y$ is also the identity. But this morphism identifies with $\pi_!(\varphi_{\mathscr{F}}^\lambda)$, which completes the proof.
\end{proof}

We now consider a second $A$-torsor $\pi' : X' \to Y'$, and an $A$-equivariant morphism $f : X \to X'$. The following claims follow easily from the definitions.

\begin{lem}\phantomsection
\label{lem:properties-monodromy}
 \begin{enumerate}
  \item  The functors $f^!$ and $f^*$ induce functors
  \[
   f^!, f^* : \Db_c(X' \quot A, \bk) \to \Db_c(X \quot A, \bk).
  \]
Moreover, for any $\mathscr{F}$ in $\Db_c(X' \quot A, \bk)$ and $r \in R_A^\wedge$ we have
\[
 \varphi_{f^! \mathscr{F}}^\wedge(r) = f^!(\varphi^\wedge_{\mathscr{F}}(r)), \quad \varphi_{f^* \mathscr{F}}^\wedge(r) = f^*(\varphi^\wedge_{\mathscr{F}}(r)).
\]
\item  The functors $f_!$ and $f_*$ induce functors
  \[
   f_!, f_* : \Db_c(X \quot A, \bk) \to \Db_c(X' \quot A, \bk).
  \]
Moreover, for any $\mathscr{F}$ in $\Db_c(X \quot A, \bk)$ and $r \in R_A^\wedge$ we have
\[
 \varphi_{f_! \mathscr{F}}^\wedge(r) = f_!(\varphi^\wedge_{\mathscr{F}}(r)), \quad \varphi_{f_* \mathscr{F}}^\wedge(r) = f_*(\varphi^\wedge_{\mathscr{F}}(r)).
\]
 \end{enumerate}
\end{lem}

Finally, we consider a second torus $A'$, and an injective morphism $\phi : A' \to A$. Of course, in this setting we can consider $X$ either as an $A$-torsor or as an $A'$-torsor, and $\Db_c(X \quot A, \bk)$ is a full subcategory in $\Db_c(X \quot A', \bk)$. In particular, for $\mathscr{F}$ in $\Db_c(X \quot A, \bk)$ we can consider the morphism $\varphi^\wedge_{\mathscr{F}}$ both for the action of $A$ (in which case we will denote it $\varphi^\wedge_{\mathscr{F},A}$) and for the action of $A'$ (in which case we will denote it $\varphi^\wedge_{\mathscr{F},A'}$). Once again, the following lemma immediately follows from the definitions.

\begin{lem}
\label{lem:monodromy-isom}
For $\mathscr{F}$ in $\Db_c(X \quot A, \bk)$, the morphism $\varphi^\wedge_{\mathscr{F},A'}$ is the composition of $\varphi^\wedge_{\mathscr{F},A}$ with the morphism $R^\wedge_{A'} \to R^\wedge_A$ induced by $\phi$.
\end{lem}

\subsection{Monodromy and equivariance}

For simplicity, in this subsection we assume that $A=\C^\times$. 
We denote by $a,p : A \times X \to X$ the action and projection maps, respectively. Recall that a perverse sheaf $\mathscr{F}$ in $\Db_c(X,\bk)$ is said to be $A$-equivariant if $a^*(\mathscr{F}) \cong p^*(\mathscr{F})$. (See~\cite[Appendix~A]{br} for the equivalence with other ``classical'' definitions.)

\begin{lem}
\label{lem:mon-equiv}
Let $\mathscr{F}$ be a perverse sheaf in $\Db_c(X \quot A, \bk)$. Then $\mathscr{F}$ is $A$-equivariant iff the morphism $\varphi_{\mathscr{F}}^\wedge$ factors through $\varepsilon_A^\wedge$.
\end{lem}

\begin{proof}
If $\mathscr{F}$ is equivariant, then there exists an isomorphism $a^*(\mathscr{F}) \simto p^*(\mathscr{F})$ whose restriction to $\{1\} \times X$ is the identity. For $\lambda \in X_*(A)$, pulling back under the morphism $\C \times X \to A \times X$ given by $(z,x) \mapsto (\lambda(\exp(z)),x)$ we obtain the morphism $\iota^\lambda_{\mathscr{F}}$ of Lemma~\ref{lem:monodromy-def}, whose restriction to $\{2\mathbf{i}\pi\} \times X$ is therefore the identity.

Conversely, assume that $\varphi^\wedge_{\mathscr{F}}$ factors through $\varepsilon_A$. Let $\lambda : \C^\times \to A$ be the tautological cocharacter, and let $f : \C \times X \to \C \times X$ be the map defined by $f(z,x)=(z+2\mathbf{i}\pi,x)$. Then $f^*(\iota_{\mathscr{F}}^\lambda)$ is a morphism $\theta_\lambda^*(\mathscr{F}) \to \mathrm{pr}^*(\mathscr{F})$ whose restriction to $\{0\} \times X$ is, by assumption, the identity of $\mathscr{F}$. Therefore, by the unicity claim in Lemma~\ref{lem:monodromy-def}, we have $f^*(\iota_{\mathscr{F}}^\lambda)=\iota_{\mathscr{F}}^\lambda$.

Now, we explain how to construct an isomorphism $\eta : a^*(\mathscr{F}) \simto p^*(\mathscr{F})$. Recall (see~\cite[Corollaire~2.1.22]{bbd}) that since we consider (shifts of) perverse sheaves, such an isomorphism can be constructed locally; more concretely, if we set $U_1=\C \smallsetminus \mathbb{R}_{\geq 0}$ and $U_2=\C \smallsetminus \mathbb{R}_{\leq 0}$, then to construct $\eta$ it suffices to construct isomorphisms on $U_1 \times X$ and $U_2 \times X$, which coincide on $(U_1 \cap U_2) \times X$. The map $\C \times X \to A \times X$ given by $(z,x) \mapsto (\lambda(\exp(z)),x)$ restricts to homeomorphisms between $\{z \in \C \mid \Im(z) \in (0,2\pi)\} \times X$ and $U_1 \times X$, and between $\{z \in \C \mid \Im(z) \in (-\pi,\pi)\} \times X$ and $U_2 \times X$. Therefore, we can obtain the isomorphisms on $U_1 \times X$ and $U_2 \times X$ by simply restricting $\iota^\lambda_{\mathscr{F}}$ to these open subsets. The intersection $U_1 \cap U_2$ has two connected components: $U_+=\{z \in \C \mid \Im(z)>0\}$ and $U_-=\{z \in \C \mid \Im(z)<0\}$. Our two isomorphisms coincide on $U_+ \times X$ by definition, and they coincide on $U_- \times X$ because of the equality $f^*(\iota_{\mathscr{F}}^\lambda)=\iota_{\mathscr{F}}^\lambda$ justified above. Hence they indeed glue to an isomorphism $\eta : a^*(\mathscr{F}) \simto p^*(\mathscr{F})$, which finishes the proof.
\end{proof}

\begin{rmk}
\begin{enumerate}
\item
Our proof of Lemma~\ref{lem:mon-equiv} can easily be adapted to the case of a general torus; we leave the details to interested readers.
\item
In~\cite{verdier}, Verdier defines (by the exact same procedure) monodromy for a more general class of objects in $\Db_c(X,\bk)$, called the \emph{monodromic complexes}, namely those complexes $\mathscr{F}$ such that the restriction of $\mathscr{H}^i(\mathscr{F})$ to each $A$-orbit is locally constant for any $i \in \Z$.
As was suggested to one of us by J.~Bernstein, one can give an alternative definition of the category $\Db_c(X \quot A,\bk)$ as the category of monodromic complexes $\mathscr{F}$ (in this sense) such that the monodromy morphism $\varphi_{\mathscr{F}} : R_A \to \End(\mathscr{F})$ is unipotent, i.e.~factors through $R_A / \mathfrak{m}_A^n$ for some $n$. Indeed, it is clear that our category $\Db_c(X \quot A,\bk)$ is included in the latter category. Now if $\mathscr{F}$ is monodromic with unipotent monodromy, then $\mathscr{F}$ is an extension of its perverse cohomology objects, which have the same property; hence we can assume that $\mathscr{F}$ is perverse. Then one can consider the (finite) filtration
\[
\mathscr{F} \supset \sum_{x \in \mathfrak{m}_A} \mathrm{Im}(x) \supset \sum_{x \in \mathfrak{m}^2_A} \mathrm{Im}(x) \supset \cdots.
\]
Each subquotient in this filtration is a perverse sheaf with trivial monodromy, hence belongs to the essential image of $\pi^*$ by (the general version of) Lemma~\ref{lem:mon-equiv}.
\end{enumerate}
\end{rmk}

\section{Completed category}

\subsection{Definition}
\label{ss:definition-hD}

As in Section~\ref{sec:monodromy} we consider a complex torus $A$ of rank $r$ and an $A$-torsor $\pi : X \to Y$. We also assume we are given a finite algebraic stratification
\[
 Y = \bigsqcup_{s \in \mathcal{S}} Y_s
\]
where each $Y_s$ is isomorphic to an affine space, and such that for any $s \in \mathcal{S}$ the restriction $\pi_s : \pi^{-1}(Y_s) \to Y_s$ is a trivial $A$-torsor. We set
\[
 \pi_\dag := \pi_! [r], \qquad \pi^\dag := \pi^![-r] \cong \pi^*[r].
\]
Then $(\pi_\dag, \pi^\dag)$ is an adjoint pair, and $\pi^\dag$ is t-exact with respect to the perverse t-structures.

We denote by $\Db_{\mathcal{S}}(Y,\bk)$ the $\mathcal{S}$-constructible derived category of $\bk$-sheaves on $Y$, and by $\Db_{\mathcal{S}}(X \quot A,\bk)$ the full triangulated subcategory of $\Db_c(X,\bk)$ generated by the essential image of the restriction of $\pi^\dag$ to $\Db_{\mathcal{S}}(Y,\bk)$.

\begin{defn}
 The category $\hD_{\mathcal{S}}(X \quot A,\bk)$ is defined as the full subcategory of the category of pro-objects\footnote{All our pro-objects are tacitly parametrized by $\Z_{\geq 0}$ (with its standard order).} in $\Db_{\mathcal{S}}(X \quot A,\bk)$ consisting of the objects $``\varprojlim" \mathscr{F}_n$ which are:
 \begin{itemize}
  \item $\pi$-constant, i.e.~such that the pro-object $``\varprojlim" \pi_\dag(\mathscr{F}_n)$ in $\Db_{\mathcal{S}}(Y,\bk)$ is isomorphic to an object of $\Db_{\mathcal{S}}(Y,\bk)$;
  \item uniformly bounded in degrees, i.e.~isomorphic to a pro-object $``\varprojlim" \mathscr{F}'_n$ such that each $\mathscr{F}'_n$ belongs to $D^{[a,b]}_{\mathcal{S}}(X \quot A,\bk)$ for some $a,b \in \Z$ (independent of $n$).
 \end{itemize}
\end{defn}

The morphisms in this category can be described as
\begin{equation}
\label{eqn:Hom-hD}
\Hom_{\hD_{\mathcal{S}}(X \quott A,\bk)}(``\varprojlim" \mathscr{F}_n, ``\varprojlim" \mathscr{G}_n) = \varprojlim_n \varinjlim_m \Hom_{\Db_{\mathcal{S}}(X \quott A,\bk)}(\mathscr{F}_m, \mathscr{G}_n).
\end{equation}

According to~\cite[Theorem~A.3.2]{by}, the category $\hD_{\mathcal{S}}(X \quot A,\bk)$ has a natural triangulated structure, for which the distinguished triangles are the triangles isomorphic to those of the form
\[
``\varprojlim" \mathscr{F}_n \xrightarrow{``\varprojlim" f_n} ``\varprojlim" \mathscr{G}_n \xrightarrow{``\varprojlim" g_n} ``\varprojlim" \mathscr{H}_n \xrightarrow{``\varprojlim" h_n} ``\varprojlim" \mathscr{F}_n [1]
\]
obtained from projective systems of distinguished triangles
\[
\mathscr{F}_n \xrightarrow{f_n} \mathscr{G}_n \xrightarrow{g_n} \mathscr{H}_n \xrightarrow{h_n} \mathscr{F}_n [1]
\]
in $\Db_{\mathcal{S}}(X \quot A,\bk)$.
By definition the functor $\pi_\dag$ induces a functor
\[
 \hD_{\mathcal{S}}(X \quot A,\bk) \to \Db_{\mathcal{S}}(Y,\bk),
\]
which will also be denoted $\pi_\dag$. From the proof of~\cite[Theorem~A.3.2]{by} we see that this functor is triangulated.

The monodromy construction from Section~\ref{sec:monodromy} makes the category $\hD_{\mathcal{S}}(X \quot A,\bk)$ an $R_A^\wedge$-linear category. More precisely, for any object $\mathscr{F} = ``\varprojlim" \mathscr{F}_n$ in $\hD_{\mathcal{S}}(X \quot A,\bk)$, we have
\[
 \End(\mathscr{F}) = \varprojlim_n \varinjlim_m \Hom_{\Db_{\mathcal{S}}(X \quott A,\bk)}(\mathscr{F}_m, \mathscr{F}_n),
\]
see~\eqref{eqn:Hom-hD}.
We have a natural algebra morphism $R_A \to \End(\mathscr{F})$, sending $r \in R_A$ to $(\varphi_{\mathscr{F}_n}(r))_{n}$. Since each $\varphi_{\mathscr{F}_n}$ factors through a quotient $R_A / \mathfrak{m}_A^{N}$ for some $N$ (depending on $n$), this morphism extends to a morphism $\varphi_{\mathscr{F}}^\wedge : R_A^\wedge \to \End(\mathscr{F})$. As in~\S\ref{ss:monodromy-construction}, this construction provides an $R_A^\wedge$-linear structure on $\hD_{\mathcal{S}}(X \quot A,\bk)$.

All the familiar functors (in particular, the pushforward and pullback functors associated with morphisms of $A$-torsors) induce functors between the appropriate completed categories, which will be denoted similarly; for details the reader might consult~\cite[Proposition~A.3.3 and Corollary~A.3.4]{by}.

\begin{rmk}
\label{rmk:completed-filtered}
As explained in~\cite[Remark~A.2.3]{by}, there exists a filtered triangulated category $\hD^F_{\mathcal{S}}(X \quot A,\bk)$ over $\hD_{\mathcal{S}}(X \quot A,\bk)$ in the sense of~\cite[Definition~A.1(c)]{beilinson}. Namely, consider a filtered triangulated category $D^F_{\mathcal{S}}(X \quot A,\bk)$ over $\Db_{\mathcal{S}}(X \quot A,\bk)$ (constructed e.g.~following~\cite[Example~A.2]{beilinson}). Then one can take as $\hD^F_{\mathcal{S}}(X \quot A,\bk)$ the category of pro-objects $``\varprojlim" \mathscr{F}_n$ in $D^F_{\mathcal{S}}(X \quot A,\bk)$ such that the filtrations on the objects $\mathscr{F}_n$ are uniformly bounded, and such that $``\varprojlim" \mathrm{gr}^F_i(\mathscr{F}_n)$ belongs to $\hD_{\mathcal{S}}(X \quot A,\bk)$ for any $i \in \Z$.
\end{rmk}

\subsection{The free-monodromic local system}
\label{ss:fm-local-system}

Let us consider the special case $X=A$ (with its natural action) and $Y=\mathrm{pt}$. Let us choose as a generator of the fundamental group $\pi_1(\C^\times)$ the anti-clockwise loop $\gamma : t \in [0,1] \mapsto \exp(2\mathbf{i}\pi t)$. Then we obtain a group isomorphism
\begin{equation}
\label{eqn:isom-fund-group}
 X_*(A) \simto \pi_1(A)
\end{equation}
by sending $\lambda \in X_*(A)$ to the class of the loop $t \mapsto \lambda(\gamma(t))$. (Here, our fundamental groups are taken with the neutral element as base point.) Of course the category of $\bk$-local systems on $A$ is equivalent to the category of finite-dimension $\bk$-representations of $\pi_1(A)$. Via the isomorphism~\eqref{eqn:isom-fund-group}, we thus obtain an equivalence between the category of $\bk$-local systems on $A$ and that of finite-dimensional $R_A$-modules. The Serre subcategory consisting of local systems which are extensions of copies of the constant local system $\underline{\bk}_A$ then identifies with the category of finite-dimensional $R_A$-modules annihilated by a power of $\mathfrak{m}_A$, or equivalently with the category of finite-dimensional $R_A^\wedge$-modules annihilated by a power of $\mathfrak{m}^\wedge_A:= \mathfrak{m}_A
 R_A^\wedge$. The latter category will be denoted $\Mod^{\nil}(R_A^\wedge)$.

For any $n \in \Z_{\geq 0}$, we denote by $\mathscr{L}_{A,n}$ the local system on $A$ corresponding to the $R_A$-module $R_A / \mathfrak{m}_A^{n+1}$. Then we have natural surjections $\mathscr{L}_{A,n+1} \to \mathscr{L}_{A,n}$, hence we can define $\widehat{\mathscr{L}}_A$ as the pro-object $``\varprojlim" \mathscr{L}_{A,n}$. It is clear that this pro-object is uniformly bounded. It is easily seen that it is also $\pi$-constant; in fact the surjections $\mathscr{L}_{A,n} \twoheadrightarrow \mathscr{L}_{A,0}=\underline{\bk}_A$ induce an isomorphism
\[
 ``\varprojlim" \pi_!(\mathscr{L}_{A,n}) \simto \mathscr{H}^{2r}(\pi_! \mathscr{L}_{A,0})[-2r] = \bk[-2r].
\]
In particular, this shows that $\widehat{\mathscr{L}}_A$ defines an object of $\hD(A \quot A,\bk)$, which satisfies
\begin{equation}
\label{eqn:pidag-L}
\pi_\dag(\widehat{\mathscr{L}}_A) \cong \bk[-r].
\end{equation}
(The stratification of $Y=\mathrm{pt}$ we consider here is the obvious one.)

\begin{rmk}
Choose a trivialization $A \simto (\C^\times)^r$. Then we obtain an isomorphism $R_A \cong (R_{\C^\times})^{\otimes r}$, see Remark~\ref{rmk:RA}. For any $n \geq 0$ we have
\begin{multline*}
\mathfrak{m}_A^{n \cdot r} \subset
\mathfrak{m}_{\C^\times}^n \otimes (R_{\C^\times})^{\otimes (r-1)} + R_{\C^\times} \otimes \mathfrak{m}_{\C^\times}^n \otimes (R_{\C^\times})^{\otimes (r-2)} + \cdots +  (R_{\C^\times})^{\otimes (r-1)} \otimes \mathfrak{m}_{\C^\times}^n \\
\subset \mathfrak{m}_A^n,
\end{multline*}
hence an isomorphism
\begin{equation}
\label{eqn:isom-wL}
\widehat{\mathscr{L}}_A \simto ``\varprojlim" \bigl( \mathscr{L}_{\C^\times,n} \bigr)^{\boxtimes r}.
\end{equation}
The definition of $\widehat{\mathscr{L}}_A$ given above is much more canonical, but the description as the right-hand side in~\eqref{eqn:isom-wL} is sometimes useful to reduce the proofs to the case $r=1$.
\end{rmk}

\subsection{``Averaging'' with the free-monodromic local system}
\label{ss:av-fm-local-system}

In this subsection, for simplicity we assume that $\ell:=\mathrm{char}(\bk)$ is positive. We will prove a technical lemma that will allow us later to prove that in the flag variety setting the convolution product admits a unit (see Lemma~\ref{lem:unit-convolution}). A reader ready to accept (or ignore) this question might skip this subsection.

We denote by $a : A \times X \to X$ the action morphism.

\begin{lem}
\label{lem:averaging-hL}
For any $\mathscr{F}$ in $\Db_{\mathcal{S}}(X \quot A,\bk)$, there exists a canonical isomorphism
\[
a_! \bigl( \widehat{\mathscr{L}}_A \boxtimes \mathscr{F} \bigr) \cong \mathscr{F}[-2r].
\]
\end{lem}

\begin{proof}
We first want to construct a morphism of functors $a_! \bigl( \widehat{\mathscr{L}}_A \boxtimes - \bigr) \to \id[-2r]$. For this, by adjunction it suffices to construct a morphism of functors
\begin{equation}
\label{eqn:morph-unit}
\bigl( \widehat{\mathscr{L}}_A \boxtimes - \bigr) \to a^![-2r].
\end{equation}

For any $s \geq 0$, we denote by $[s] : A \to A$ the morphism sending $z$ to $z^{\ell^s}$, and set $a_s:=a \circ ([s] \times \id_X)$. Since any unipotent matrix $M$ with coefficients in $\bk$ satisfies $M^{\ell^s}=1$ for $s \gg 0$, we see that for $\mathscr{F}$ in $\Db_{\mathcal{S}}(X \quot A,\bk)$, for $s \gg 0$ all the cohomology objects of $(a_s)^* \mathscr{F}$ are constant on the fibers of the projection to $X$. In fact, the techniques of~\cite[\S 5]{verdier} show that for any such $\mathscr{F}$ and for $s \gg 0$ there exists an isomorphism $f_s^{\mathscr{F}} : (a_s)^* \mathscr{F} \simto p^*(\mathscr{F})$ whose restriction to $\{1\} \times X$ is the identity. Moreover, these morphisms are essentially unique in the sense that given $s,s'$ such that $f_s^{\mathscr{F}}$ and $f_{s'}^{\mathscr{F}}$ are defined, for $t \gg s,s'$ we have
\[
([t-s] \times \id_X)^* f_s^{\mathscr{F}} = ([t-s'] \times \id_X)^* f_{s'}^{\mathscr{F}},
\]
and functorial in the sense that if $u : \mathscr{F} \to \mathscr{G}$ is a morphism then for $s \gg 0$ the diagram
\[
\xymatrix@C=2cm{
(a_s)^* \mathscr{F} \ar[r]_-{\sim}^-{f_s^{\mathscr{F}}} \ar[d]_-{(a_s)^* u} & p^* \mathscr{F} \ar[d]^-{p^* u} \\
(a_s)^* \mathscr{G} \ar[r]_-{\sim}^-{f_s^{\mathscr{G}}} & p^* \mathscr{G}
}
\]
commutes.

Now, fix $\mathscr{F}$ in $\Db_{\mathcal{S}}(X \quot A,\bk)$. For $s \gg 0$, we have the morphism
\begin{multline*}
(f_s^{\mathscr{F}})^{-1} \in \Hom \bigl( p^*(\mathscr{F}), (a_s)^* \mathscr{F} \bigr) = \Hom \bigl( p^*(\mathscr{F}), (a_s)^!\mathscr{F} [-2r]\bigr) \\
\cong \Hom \bigl( ([s] \times \id_X)_! p^*(\mathscr{F}), a^!(\mathscr{F}) [-2r]\bigr).
\end{multline*}
The ``essential unicity'' claimed above implies that these morphisms define a cano\-nical element in
\[
\varinjlim_s \Hom \bigl( ([s] \times \id_X)_! p^*(\mathscr{F}), a^!(\mathscr{F}) [-2r] \bigr) = \Hom \left( \bigl( ``\varprojlim_s " [s]_! \underline{\bk} \bigr) \boxtimes \mathscr{F}, a^! \mathscr{F} [-2r] \right).
\]
Now we observe that $[s]_! \underline{\bk} =\mathscr{L}_{A,\ell^s}$, so that $``\varprojlim_s " [s]_! \underline{\bk} \cong \widehat{\mathscr{L}}_A$, and we deduce the wished-for morphism~\eqref{eqn:morph-unit}. (The functoriality of our morphism follows from the ``functoriality'' of the morphisms $f_s^{\mathscr{F}}$ claimed above.)

To conclude the proof it remains to show that the morphism $a_! \bigl( \widehat{\mathscr{L}}_A \boxtimes \mathscr{F} \bigr) \to \mathscr{F}$ is an isomorphism for any $\mathscr{F}$ in $\Db_{\mathcal{S}}(X \quot A,\bk)$. By the 5-lemma and the definition of this category, it suffices to do so in case $\mathscr{F}=\pi^\dag \mathscr{G}$ for some $\mathscr{G}$ in $\Db_{\mathcal{S}}(Y,\bk)$. In this case, the morphism $f_t^{\mathscr{F}}$ is defined for any $t \geq 0$, and can be chosen as the obvious isomorphism
\[
(a_t)^* \mathscr{F} = (a_t)^* \pi^* \mathscr{G}[-r] = (\pi \circ a_t)^* \mathscr{G}[-r] = (\pi \circ p)^* \mathscr{G}[-r] = p^* \mathscr{F}.
\]
Then under the identification
\[
a_! \bigl( \widehat{\mathscr{L}}_A \boxtimes \mathscr{F} \bigr) = \pi^\dag (p_Y)_! (\widehat{\mathscr{L}}_A \boxtimes \mathscr{G} ) = \pi^\dag \bigl( (\pi')_! (\widehat{\mathscr{L}}_A) \boxtimes \mathscr{G} \bigr),
\]
where $p_Y : A \times Y \to Y$ and $\pi':A \to \mathrm{pt}$ are the projections, our morphism is induced by the isomorphism $(\pi')_!(\widehat{\mathscr{L}}_A) \cong \bk[-2r]$ from~\S\ref{ss:fm-local-system}. This concludes the proof.
\end{proof}

\section{The case of the trivial torsor}
\label{sec:trivial-torsor}

In this section we study the category $\hD_{\mathcal{S}}(X \quot A, \bk)$ in the special case $X=A$.

\subsection{Description of \texorpdfstring{$\hD(A \quot A,\bk)$}{D(A/A)} in terms of pro-complexes of \texorpdfstring{$R_A^\wedge$}{RA}-modules}
\label{ss:description-A/A}

As explained in~\S\ref{ss:fm-local-system}, every object of $\Mod^{\nil}(R_A^\wedge)$ defines a sheaf on $A$; this assignment therefore defines a functor $\Db \Mod^{\nil}(R_A^\wedge) \to \Db_c(A,\bk)$, which clearly takes values in $\Db(A \quot A,\bk)$. We will denote by
\[
 \Phi_A : \Db \Mod^{\nil}(R_A^\wedge) \to \Db(A \quot A,\bk)
\]
the composition of this functor with the shift of complexes by $r$ to the left (where $r$ is the rank of $A$). In this way, $\Phi_A$ is t-exact if $\Db(A \quot A,\bk)$ is equipped with the perverse t-structure.

\begin{lem}
 The functor $\Phi_A$ is an equivalence of triangulated categories.
\end{lem}

\begin{proof}
If we denote by $\bk$ the $R_A^\wedge$-module $R_A^\wedge / \mathfrak{m}_A^\wedge$, then
 it is clear that $\Phi_A(\bk) = \underline{\bk}_A [r]$. We claim that $\Phi_A$ induces an isomorphism
 \[
  \bigoplus_{n \in \Z} \Hom_{\Db \Mod^{\nil}(R_A^\wedge)}(\bk, \bk[n]) \simto \bigoplus_{n \in \Z} \Hom_{\Db(A \quott A,\bk)}(\underline{\bk}_A [r],\underline{\bk}_A [r+n]).
 \]
 Here, the right-hand side identifies with $\mathsf{H}^\bullet(A;\bk)$.
 
Choosing a trivialization of $A$ we reduce the claim to the case $r=1$, i.e.~$A=\Gm$ (see Remark~\ref{rmk:RA}). In this case the left-hand side has dimension $2$, with a basis consisting of $\id : \bk \to \bk$ and the natural extension
\[
 \bk = \mathfrak{m}_{\C^\times}^\wedge / (\mathfrak{m}_{\C^\times}^\wedge)^2 \hookrightarrow R_{\C^\times}^\wedge / (\mathfrak{m}_{\C^\times}^\wedge)^2 \twoheadrightarrow R_{\C^\times}^\wedge / \mathfrak{m}_{\C^\times}^\wedge=\bk.
\]
It is clear that $\Phi_{\C^\times}$ identifies this space with $\mathsf{H}^\bullet(\C^\times;\bk)$, and the claim is proved.

Since the object $\bk$, resp.~the object $\underline{\bk}_A[r]$, generates the triangulated category $\Db \Mod^{\nil}(R_A^\wedge)$, resp.~$\Db(A \quot A,\bk)$, this claim and Be{\u\i}linson's lemma imply that $\Phi_A$ indeed is an equivalence of categories.
\end{proof}

The category $\Db_c(\mathrm{pt},\bk)$ identifies with $\Db \Vect^{\mathrm{fd}}_\bk$, where $\Vect^{\mathrm{fd}}_\bk$ is the category of finite-dimensional $\bk$-vector spaces. Under this identification, the functor $\pi^\dag$ corresponds to the composition of $\Phi_A$ with the restriction-of-scalars functor associated with the natural surjection $R_A^\wedge \twoheadrightarrow \bk$. By adjunction, we deduce an isomorphism
\[
 \pi_\dag \circ \Phi_A \cong \bk \lotimes_{R_A^\wedge} (-).
\]

In view of these identifications, the category $\hD(A \quot A,\bk)$ is therefore equivalent to the category $\hD(R_A^\wedge)$ of pro-objects $``\varprojlim" M_n$ in $\Db \Mod^{\nil}(R_A^\wedge)$ which are uniformly bounded and such that the object
\[
 ``\varprojlim" \bk \lotimes_{R_A^\wedge} M_n
\]
is isomorphic to an object of $\Db \Vect^{\mathrm{fd}}_\bk$. 
We use this equivalence to transport the triangulated structure on $\hD(A \quot A,\bk)$ to $\hD(R_A^\wedge)$.

\subsection{Some results on pro-complexes of \texorpdfstring{$R_A^\wedge$}{RA}-modules}

We now consider
\[
 \widehat{L}_A := ``\varprojlim" R_A^\wedge / (\mathfrak{m}_A^\wedge)^{n+1},
\]
a pro-object in the category $\Db \Mod^{\nil}(R_A^\wedge)$.

\begin{lem}
\label{lem:hL}
 For any $M$ in $\Db \Mod^{\nil}(R_A^\wedge)$, there exists a canonical isomorphism
 \[
  \Hom(\widehat{L}_A,M) \cong \mathsf{H}^0(M)
 \]
(where morphisms are taken in the category of pro-objects in $\Db \Mod^{\nil}(R_A^\wedge)$).
\end{lem}

\begin{proof}
By d\'evissage it is sufficient to prove this claim when $M$ is concentrated in a certain degree $k$, i.e.~$M=N[-k]$ for some $N$ in $\Mod^{\nil}(R_A^\wedge)$.
 By definition we have
 \[
  \Hom(\widehat{L}_A,N[-k]) = \varinjlim_n \Ext^{-k}_{R_A^\wedge}(R_A^\wedge / (\mathfrak{m}_A^\wedge)^{n+1}, N).
 \]
 If $k=0$, it is easily seen that the right-hand side identifies with $N$. Now if $k \neq 0$ we use the fact that the natural functor from $\Db \Mod^{\nil}(R_A^\wedge)$ to the bounded derived category of $R_A^\wedge$-modules is fully faithful (see e.g.~\cite[Lemma~2.1]{orlov}), which implies that any morphism $f : R_A^\wedge / (\mathfrak{m}_A^\wedge)^{n+1} \to N[-k]$ is the image of a morphism in the category $\Db \Mod(R_A^\wedge / (\mathfrak{m}_A^\wedge)^{m+1})$ for $m \gg 0$. Then the image of $f$ in $\Ext^{-k}_{R_A^\wedge}(R_A^\wedge / (\mathfrak{m}_A^\wedge)^{m+1}, N)$ vanishes, since it is the image of a morphism in $\Hom_{\Db \Mod(R_A^\wedge / (\mathfrak{m}_A^\wedge)^{m+1})}(R_A^\wedge / (\mathfrak{m}_A^\wedge)^{m+1}, N[-k])=0$.
\end{proof}

As a consequence of this lemma, one obtains in particular an isomorphism
\begin{equation}
\label{eqn:hL-tensor}
 ``\varprojlim" \bk \lotimes_{R_A^\wedge} R_A^\wedge / (\mathfrak{m}_A^\wedge)^{n+1} \cong \bk
\end{equation}
in the category of pro-objects in $\Db \Vect^{\mathrm{fd}}_\bk$. This shows that $\widehat{L}_A$ belongs to $\hD(R_A^\wedge)$. (Of course, this property also follows from the fact that this object is the image of $\widehat{\mathscr{L}}_A[r]$ under the equivalence considered in~\S\ref{ss:description-A/A}.)

\begin{lem}
\label{lem:pro-objects-truncation}
 Let $``\varprojlim" M_n$ be an object of $\hD(R_A^\wedge)$, and assume that the object $``\varprojlim" \bk \lotimes_{R_A^\wedge} M_n$ belongs to $D^{\leq 0} \Vect^{\mathrm{fd}}_\bk$. Then the obvious morphism
 \[
  ``\varprojlim" \tau_{\leq 0} M_n \to ``\varprojlim" M_n
 \]
is an isomorphism in the category of pro-objects in $\Db \Mod^{\nil}(R_A^\wedge)$, where $\tau_{\leq 0}$ is the usual truncation functor for complexes of $R_A^\wedge$-modules.
\end{lem}

\begin{proof}
 By uniform boundedness, we can assume that each complex $M_n$ belongs to $D^{\leq d} \Mod^{\nil}(R_A^\wedge)$ for some $d \in \Z$. If $d \leq 0$ then there is nothing to prove. Hence we assume that $d>0$. We will prove that in this case the pro-object $``\varprojlim" H^d(M_n)$ is isomorphic to $0$. Since filtrant direct limits are exact, this will show that for any $X$ in $\Db \Mod^{\nil}(R_A^\wedge)$ the morphism
 \[
  \varinjlim \Hom(M_n,X) \to \varinjlim \Hom(\tau_{<d} M_n,X)
 \]
is an isomorphism, hence that the morphism of pro-objects
\[
 ``\varprojlim" \tau_{<d} M_n \to ``\varprojlim" M_n
\]
is an isomorphism. Of course, this property is sufficient to conclude.

We observe that the pro-object
\[
 ``\varprojlim" \bk \otimes_{R_A^\wedge} H^d(M_n) = ``\varprojlim" H^d \left( \bk \lotimes_{R_A^\wedge} M_n \right) = H^d \left( ``\varprojlim" \bk \lotimes_{R_A^\wedge} M_n\right)
\]
in the category $\Vect^{\mathrm{fd}}_\bk$
vanishes. Hence for any fixed $n$, for $m \gg n$ the map $\bk \otimes_{R_A^\wedge} H^d(M_m) \to \bk \otimes_{R_A^\wedge} H^d(M_n)$ vanishes, or in other words the map $H^d(M_m) \to H^d(M_n)$ takes values in $\mathfrak{m}_A^\wedge \cdot H^d(M_n)$. Since $H^d(M_n)$ is annihilated by $(\mathfrak{m}_A^\wedge)^q$ for some $q$, this implies that the map $H^d(M_m) \to H^d(M_n)$ vanishes for $m \gg 0$. Clearly, this implies that $``\varprojlim" H^d(M_n) \cong 0$, and concludes the proof.
\end{proof}

\begin{lem}
\label{lem:hL-generates}
 The object $\widehat{L}_A$ generates $\hD(R_A^\wedge)$ as a triangulated category.
\end{lem}

\begin{proof}
 We will prove, by induction on the length of the shortest interval $I \subset \Z$ such that $``\varprojlim" \bk \lotimes_{R_A^\wedge} M_n$ belongs to $D^I \Vect^{\mathrm{fd}}_\bk$, that any object $``\varprojlim" M_n$ of $\hD(R_A^\wedge)$ belongs to the triangulated subcategory generated by $\widehat{L}_A$.
 
 First, assume that $I=\varnothing$. Then for any $X$ in $\Db \Vect^{\mathrm{fd}}_\bk$ we have
 \[
  0=\varinjlim_n \Hom_{\Db \Vect^{\mathrm{fd}}_\bk}(\bk \lotimes_{R_A^\wedge} M_n, X) \cong \varinjlim_n \Hom_{\Db \Mod^{\nil}(R_A^\wedge)}(M_n, X).
 \]
Since the essential image of $\Db \Vect^{\mathrm{fd}}_\bk$ generates $\Db \Mod^{\nil}(R_A^\wedge)$ as a triangulated category, and since filtrant direct limits are exact, it follows that
\[
 \varinjlim_n \Hom_{\Db \Mod^{\nil}(R_A^\wedge)}(M_n, X)=0
\]
for any $X$ in $\Db \Mod^{\nil}(R_A^\wedge)$. By definition, this implies that $``\varprojlim" M_n=0$, proving the claim in this case.

Now, we assume that $I \neq \varnothing$. Shifting complexes if necessary, we can assume that $I=[-d,0]$ for some $d \in \Z_{\geq 0}$. Using Lemma~\ref{lem:pro-objects-truncation}, we can then assume that each $M_n$ belongs to $D^{\leq 0} \Mod^{\nil}(R_A^\wedge)$. Set 
\[
V:=H^0 \left( ``\varprojlim" \bk \lotimes_{R_A^\wedge} M_n \right) = ``\varprojlim" H^0(\bk \lotimes_{R_A^\wedge} M_n) = ``\varprojlim" \bk \otimes_{R_A^\wedge} H^0(M_n).
\]
Then $V$ is a finite-dimensional $\bk$-vector space, and $\id_V$ defines an element in
\[
 \Hom_\bk \left( V, ``\varprojlim" \bk \otimes_{R_A^\wedge} H^0(M_n) \right) = \varprojlim \Hom_\bk \left( V, \bk \otimes_{R_A^\wedge} H^0(M_n) \right).
\]

Consider the object
\[
 \mathcal{V} := ``\varprojlim" \left( R_A^\wedge / (\mathfrak{m}_A^\wedge)^{n+1} \otimes_\bk V \right)
\]
in $\hD(R_A^\wedge)$.
(Of course, $\mathcal{V}$ is isomorphic to a direct sum of copies of $\widehat{L}_A$.) Then by Lemma~\ref{lem:hL} we have
\[
\Hom_{\hD(R_A^\wedge)}(\mathcal{V}, ``\varprojlim" M_m) = \varprojlim_m \Hom_{\hD(R_A^\wedge)}(\mathcal{V}, M_m) \\
 \cong \varprojlim_m \Hom_{\bk}(V, H^0(M_m)).
\]
Now for any $m$ we have a surjection
\[
 H^0(M_m) \twoheadrightarrow \bk \otimes_{R_A^\wedge} H^0(M_m),
\]
which induces a surjection
\[
 \Hom_{\bk}(V, H^0(M_m)) \twoheadrightarrow \Hom_\bk \left( V, \bk \otimes_{R_A^\wedge} H^0(M_m) \right).
\]
Each vector space $\ker \bigl( \Hom_{\bk}(V, H^0(M_m)) \to \Hom_\bk ( V, \bk \otimes_{R_A^\wedge} H^0(M_m) ) \bigr)$ is finite-dimensional; therefore the projective system formed by these spaces
satisfies the Mittag--Leffler condition. This implies that the map
\begin{equation}
\label{eqn:hL-generates-proof}
 \varprojlim_m \Hom_{\bk}(V, H^0(M_m)) \to \varprojlim_m \Hom_\bk \left( V, \bk \otimes_{R_A^\wedge} H^0(M_n) \right)
\end{equation}
is surjective (see e.g.~\cite[Proposition~1.12.3]{ks1}).

Let now $f : \mathcal{V} \to ``\varprojlim" M_n$ be a morphism whose image in the right-hand side of~\eqref{eqn:hL-generates-proof} is $\id_V$. By definition (and in view of~\eqref{eqn:hL-tensor}), the morphism
\[
 \bk \lotimes_{R_A^\wedge} f : \bk \lotimes_{R_A^\wedge} \mathcal{V} \to \bk \lotimes_{R_A^\wedge} ``\varprojlim" M_n
\]
induces an isomorphism in degree-$0$ cohomology. Hence the cone $C$ of $f$ (in the triangulated category $\hD(R_A^\wedge)$) is such that $\bk \lotimes_{R_A^\wedge} C$ belongs to $D^{[-d,-1]} \Vect^{\mathrm{fd}}_\bk$. By induction, this objects belongs to the triangulated subcategory of $\hD(R_A^\wedge)$ generated by $\widehat{L}_A$. Then the distinguished triangle
\[
 \mathcal{V} \to ``\varprojlim" M_n \to C \xrightarrow{[1]}
\]
shows that $``\varprojlim" M_n$ also belongs to this subcategory, which finishes the proof.
\end{proof}

\subsection{Description of \texorpdfstring{$\hD(A \quot A,\bk)$}{D(A/A)} in terms of complexes of \texorpdfstring{$R_A^\wedge$}{RA}-modules}

Recall that the algebra $R_A^\wedge$ is isomorphic to an algebra of formal power series in $r$ indeterminates, see Remark~\ref{rmk:RA}. In particular this shows that this algebra is local, Noetherian, and of finite global dimension. We will denote by $\Mod^{\mathrm{fg}}(R^\wedge_A)$ the category of finitely-generated $R_A^\wedge$-modules.

\begin{prop}
\label{prop:equiv-RAwedge}
 There exists a natural equivalence of triangulated categories
 \[
  \Db \Mod^{\mathrm{fg}}(R^\wedge_A) \simto \hD(R_A^\wedge).
 \]
\end{prop}

\begin{proof}
 We consider the functor $\varphi$ from $\Db \Mod^{\mathrm{fg}}(R^\wedge_A)$ to the category of pro-objects in $\Db \Mod^{\nil}(R_A^\wedge)$ sending a complex $M$ to
 \[
  \varphi(M) := ``\varprojlim" \left( R_A^\wedge / (\mathfrak{m}_A^\wedge)^{n+1} \lotimes_{R_A^\wedge} M \right).
 \]
Since $R_A^\wedge$ is local, Noetherian, and of finite global dimension, any object in the category $\Db \Mod^{\mathrm{fg}}(R^\wedge_A)$ is isomorphic to a bounded complex of free $R_A^\wedge$-modules. It is clear that the image of such a complex belongs to $\hD(R_A^\wedge)$; hence $\varphi$ takes values in $\hD(R_A^\wedge)$. Once this is established, it is clear that this functor is triangulated.

By Lemma~\ref{lem:hL}, for $k \in \Z$ we have
\begin{multline}
\label{eqn:End-hL}
 \Hom_{\hD(R_A^\wedge)}(\widehat{L}_A, \widehat{L}_A[k]) = \varprojlim_n \Hom_{\hD(R_A^\wedge)}(\widehat{L}_A, R_A^\wedge / (\mathfrak{m}^\wedge_A)^{n+1}[k]) \\
 \cong \begin{cases}
        R_A^\wedge & \text{if $k=0$;}\\
        0 & \text{otherwise.}
       \end{cases}
\end{multline}
Hence $\varphi$ induces an isomorphism
\[
 \Hom_{\Db \Mod^{\mathrm{fg}}(R^\wedge_A)}(R_A^\wedge,R_A^\wedge[k]) \simto \Hom_{\hD(R_A^\wedge)}(\widehat{L}_A, \widehat{L}_A[k]).
\]
Since the object $R_A^\wedge$, resp.~$\widehat{L}_A$, generates $\Db \Mod^{\mathrm{fg}}(R^\wedge_A)$, resp.~$\hD(R_A^\wedge)$, as a triangulated category (see Lemma~\ref{lem:hL-generates}), this observation and Be{\u\i}linson's lemma imply that $\varphi$ is an equivalence of categories.
\end{proof}

Combining Proposition~\ref{prop:equiv-RAwedge} and the considerations of~\S\ref{ss:description-A/A}, we finally obtain the following result.

\begin{cor}
\label{cor:equiv-A/A-RA}
 There exists a canonical equivalence of triangulated categories
 \[
  \Db \Mod^{\mathrm{fg}}(R^\wedge_A) \simto \hD(A \quot A,\bk)
 \]
sending the free module $R_A^\wedge$ to $\widehat{\mathscr{L}}_A[r]$.
\end{cor}

From~\eqref{eqn:hL-tensor} we see that the equivalence of Proposition~\ref{prop:equiv-RAwedge} intertwines the functors $\bk \lotimes_{R_A^\wedge} (-)$ on both sides. Therefore, under the equivalence of Corollary~\ref{cor:equiv-A/A-RA} the functor $\bk \lotimes_{R_A^\wedge} (-)$ on the left-hand side corresponds to the functor $\pi_\dag$ on the right-hand side. (Here $\pi : A \to \mathrm{pt}$ is the unique map, and we identify the categories $\Db_c(\mathrm{pt},\bk)$ and $\Db \Vect^{\mathrm{fd}}_\bk$ as in~\S\ref{ss:description-A/A}.)

\section{The perverse t-structure}
\label{sec:perverse}

\subsection{Recollement}
\label{ss:recollement}

We now come back to the setting of~\S\ref{ss:definition-hD}. If $Z \subset Y$ is a locally closed union of strata, and if we denote by $h : \pi^{-1}(Z) \to X$ the embedding, in view of the results recalled in~\S\ref{ss:definition-hD} the functors $h_!$, $h_*$, $h^!$, $h^*$ induce triangulated functors
\begin{align*}
 h_!, h_* &: \hD_{\mathcal{S}}(\pi^{-1}(Z) \quot A, \bk) \to \hD_{\mathcal{S}}(X \quot A, \bk), \\
 h^*, h^! &: \hD_{\mathcal{S}}(X \quot A, \bk) \to \hD_{\mathcal{S}}(\pi^{-1}(Z) \quot A, \bk)
\end{align*}
which satisfy the usual adjunction and fully-faithfulness properties. (Here, following standard conventions we write $\hD_{\mathcal{S}}(\pi^{-1}(Z) \quot A, \bk)$ for $\hD_{\mathcal{T}}(\pi^{-1}(Z) \quot A, \bk)$ where $\mathcal{T}=\{s \in \mathcal{S} \mid Y_s \subset Z\}$.) If $\pi_Z : \pi^{-1}(Z) \to Z$ is the restriction of $\pi$, and if $\overline{h} : Z \to Y$ is the embedding, then the arguments of the proof of~\cite[Corollary~A.3.4]{by} show that we have canonical isomorphisms
\begin{equation}
\label{eqn:isom-pi-embeddings}
 (\pi_Z)_\dag \circ h_? \cong \overline{h}_? \circ \pi_\dag, \quad (\pi_Z)_\dag \circ h^? \cong \overline{h}^? \circ \pi_\dag
\end{equation}
for $? \in \{!,*\}$.

In particular, if $Z$ is closed with $U:=Y \smallsetminus Z$ its open complement, and if we denote the corresponding embeddings by $i : \pi^{-1}(Z) \to X$ and $j : \pi^{-1}(U) \to X$, then we obtain a recollement diagram
\[
 \xymatrix@C=1.5cm{
 \hD_{\mathcal{S}}(\pi^{-1}(Z) \quot A, \bk) \ar[r]|-{i_*} & \hD_{\mathcal{S}}(X \quot A, \bk)\ar[r]|-{j^*} \ar@/^0.5cm/[l]^-{i^!} \ar@/_0.5cm/[l]_-{i^*} & \hD_{\mathcal{S}}(\pi^{-1}(U) \quot A, \bk) \ar@/^0.5cm/[l]^-{j_*} \ar@/_0.5cm/[l]_-{j_!}
 }
\]
in the sense of~\cite{bbd}.

\subsection{Definition of the perverse t-structure}
\label{ss:def-perv}

Let us choose, for any $s \in \mathcal{S}$, an $A$-equivariant map $p_s : X_s \to A$, where $X_s :=\pi^{-1}(Y_s)$. (Such a map exists by assumption.) Then the functor $(p_s)^*[\dim(Y_s)] \cong (p_s)^![-\dim(Y_s)]$ induces an equivalence of triangulated categories
\begin{equation*}
 \hD(A \quot A, \bk) \simto \hD_{\mathcal{S}}(X_s \quot A, \bk).
\end{equation*}
Composing with the equivalence of Corollary~\ref{cor:equiv-A/A-RA} we deduce an equivalence of categories
\begin{equation}
\label{eqn:equiv-perverse-stratum}
 \Db \Mod^{\mathrm{fg}}(R_A^\wedge) \simto \hD_{\mathcal{S}}(X_s \quot A, \bk).
\end{equation}
The transport, via this equivalence, of the tautological t-structure on $\Db \Mod^{\mathrm{fg}}(R_A^\wedge)$, will be called the \emph{perverse} t-structure, and will be denoted
\[
 \left( {}^p \hspace{-1pt} \hD_{\mathcal{S}}(X_s \quot A, \bk)^{\leq 0}, {}^p \hspace{-1pt} \hD_{\mathcal{S}}(X_s \quot A, \bk)^{\geq 0} \right).
\]

Using the recollement formalism from~\S\ref{ss:recollement}, by gluing these t-structures we obtain a t-structure on $\hD_{\mathcal{S}}(X \quot A, \bk)$, which we also call the perverse t-structure. More precisely, for any $s \in \mathcal{S}$ we denote by $j_s : X_s \to X$ the embedding. Then the full subcategory ${}^p \hspace{-1pt} \hD_{\mathcal{S}}(X \quot A, \bk)^{\leq 0}$ consists of the objects $\mathscr{F}$ such that $j_s^* \mathscr{F}$ belongs to ${}^p \hspace{-1pt} \hD_{\mathcal{S}}(X_s \quot A, \bk)^{\leq 0}$ for any $s$, and the full subcategory ${}^p \hspace{-1pt} \hD_{\mathcal{S}}(X \quot A, \bk)^{\geq 0}$ consists of the objects $\mathscr{F}$ such that $j_s^! \mathscr{F}$ belongs to ${}^p \hspace{-1pt} \hD_{\mathcal{S}}(X_s \quot A, \bk)^{\geq 0}$ for any $s$.

The heart of the perverse t-structure will be denoted $\hP_{\mathcal{S}}(X \quot A,\bk)$, and an object of $\hD_{\mathcal{S}}(X \quot A, \bk)$ will be called perverse if it belongs to this heart.

\begin{rmk}
\label{rmk:DbS-hDS}
 By construction there exists an obvious fully-faithful triangulated functor $\Db_{\mathcal{S}}(X \quot A, \bk) \to \hD_{\mathcal{S}}(X \quot A, \bk)$. The essential image of this functor consists of the objects $\widehat{\mathscr{F}}$ such that the monodromy morphism $R_A^\wedge \to \End(\widehat{\mathscr{F}})$ factors through some quotient $R_A^\wedge / (\mathfrak{m}_A^\wedge)^n = R_A / \mathfrak{m}_A^n$. In fact, it is clear that the objects in the essential image of our functor satisfy this property. For the converse statement, using the fact that this essential image is a triangulated subcategory and the usual recollement triangles we reduce the proof to the case $X$ has only one stratum. Then the equivalence~\eqref{eqn:equiv-perverse-stratum} allows to translate the question in terms of complexes of $R_A^\wedge$-modules. Using once again the triangulated structure (and the result quoted in the proof of Lemma~\ref{lem:hL}) one can then assume that the complex is concentrated in one degree; in this case the claim is obvious.
\end{rmk}

The following (well-known) claim will be needed for certain proofs below.

\begin{lem}
\label{lem:nakayama}
Let $M$ be in $\Db \Mod^{\mathrm{fg}}(R_A^\wedge)$, and assume that $\bk \lotimes_{R_A^\wedge} M$ is concentrated in non-negative degrees. Then $M$ is isomorphic to a complex of free $R_A^\wedge$-modules with nonzero terms in non-negative degrees only.
\end{lem}

\begin{proof}
Since $R_A^\wedge$ is local and of finite global dimensional, $M$ is isomorphic to a bounded complex $N^\bullet$ of free $R_A$-modules. Let $n$ be the smallest integer with $N^n \neq 0$. If $n<0$, then our assumption implies that the morphism $\bk \otimes_{R_A^\wedge} N^n \to \bk \otimes_{R_A^\wedge} N^{n+1}$ is injective. Then by the Nakayama lemma the map $N^n \to N^{n+1}$ is a split embedding, and choosing a (free) complement to its image in $N^{n+1}$ we see that $M$ isomorphic to a complex of free $R_A^\wedge$-modules concentrated in degrees $\geq n+1$. Repeating this procedure if necessary, we obtain the desired claim.
\end{proof}

\begin{lem}
\label{lem:perv-pi}
 Let $\mathscr{F}$ in $\hD_{\mathcal{S}}(X \quot A, \bk)$. 
 \begin{enumerate}
  \item 
  \label{it:pidag-perv}
  If $\pi_\dag \mathscr{F}$ is perverse, then $\mathscr{F}$ is perverse.
  \item 
  \label{it:pidag-conservative}
  If $\pi_\dag \mathscr{F}=0$, then $\mathscr{F}=0$.
  \item
  \label{it:pidag-conservative-perv}
  If $\mathscr{F}$ is perverse and $\pH^0(\pi_\dag \mathscr{F})=0$, then $\mathscr{F}=0$.
 \end{enumerate}
\end{lem}

\begin{proof}
\eqref{it:pidag-perv}
The shifted pullback functor associated with the projection $Y_s \to \mathrm{pt}$ induces a (perverse) t-exact equivalence between $\Db \Vect^{\mathrm{fd}}_\bk$ and $\Db_{\mathcal{S}}(Y_s,\bk)$. Under this equivalence and~\eqref{eqn:equiv-perverse-stratum}, the functor $(\pi_s)_\dag$ corresponds to the functor $\bk \lotimes_{R_A^\wedge} (-)$ (see the comments after Corollary~\ref{cor:equiv-A/A-RA}). In view of the isomorphisms~\eqref{eqn:isom-pi-embeddings}, this reduces the lemma to the claim that if an object $M$ of $\Db \Mod^{\mathrm{fg}}(R_A^\wedge)$ satisfies $\mathsf{H}^k(\bk \lotimes_{R_A^\wedge} M)=0$ for all $k > 0$, resp.~for all $k<0$, then we have $\mathsf{H}^k(M)=0$ for all $k > 0$, resp.~for all $k<0$. This claim is a standard consequence of the Nakayama lemma, resp.~follows from Lemma~\ref{lem:nakayama}.

\eqref{it:pidag-conservative}--\eqref{it:pidag-conservative-perv} The proofs are similar to that of~\eqref{it:pidag-perv}; details are left to the reader.
\end{proof}

\subsection{Standard and costandard perverse sheaves}
\label{ss:standard-costandard}

For any $s \in \mathscr{S}$ we denote by $i_s : Y_s \to Y$ the embedding, and consider the objects
\[
 \Delta_s := (i_s)_! \underline{\bk}_{Y_s} [\dim Y_s], \quad \nabla_s := (i_s)_* \underline{\bk}_{Y_s} [\dim Y_s].
\]
We will also set
\[
 \hL_{A,s} := (p_s)^* \hL_A,
\]
and
consider
the objects
\[
 \hDel_s := (j_s)_! \hL_{A,s} [\dim X_s], \quad \hnab_s := (j_s)_* \hL_{A,s} [\dim X_s]
\]
in $\hD_{\mathcal{S}}(X \quot A, \bk)$. In view of~\eqref{eqn:isom-pi-embeddings} and~\eqref{eqn:pidag-L}, we have canonical isomorphisms
\begin{equation}
\label{eqn:Pidag-hD-hN}
 \pi_\dag \hDel_s \cong \Delta_s, \quad \pi_\dag \hnab_s \cong \nabla_s.
\end{equation}
We also have isomorphisms of $R_A^\wedge$-modules
\begin{equation}
\label{eqn:Hom-hDel-hnab}
\Hom_{\hD_{\mathcal{S}}(X \quot A, \bk)} \left( \hDel_s, \hnab_t[k] \right) \cong \begin{cases}
                                                                          R_A^\wedge & \text{if $s=t$ and $k=0$;} \\
                                                                        0 & \text{otherwise.}
                                                                         \end{cases}
\end{equation}

Our map $i_s$ is an affine morphism, so that the objects $\Delta_s$ and $\nabla_s$ are perverse sheaves on $Y$.
By Lemma~\ref{lem:perv-pi}\eqref{it:pidag-perv} and~\eqref{eqn:Pidag-hD-hN}, this implies that the objects $\hDel_s$ and $\hnab_s$ are perverse too.

\begin{lem}\phantomsection
\label{lem:properties-hD-hN}
 \begin{enumerate}
  \item
  \label{it:hD-hN-generate}
  The triangulated category $\hD_{\mathcal{S}}(X \quot A, \bk)$ is generated by the objects $\hDel_s$ for $s \in \mathcal{S}$, as well as by the objects $\hnab_s$ for $s \in \mathcal{S}$.
  \item
  \label{it:Edn-hD}
  For any $s \in \mathcal{S}$, the monodromy morphism $\varphi^\wedge_{\hDel_s}$ induces an isomorphism
  \[
   R_A^\wedge \simto \Hom_{\hD_{\mathcal{S}}(X \quot A, \bk)}(\hDel_s, \hDel_s).
  \]
Moreover, any nonzero endomorphism of $\hDel_s$ is injective.
 \end{enumerate}
\end{lem}

\begin{proof}
 Property~\eqref{it:hD-hN-generate} follows from 
 the equivalences~\eqref{eqn:equiv-perverse-stratum}, and the gluing formalism. 
 And in~\eqref{it:Edn-hD}, the isomorphism follows from
 the equivalence~\eqref{eqn:equiv-perverse-stratum} and the fact that $(j_s)_!$ is fully faithful.
 
 Now, let $x \in R_A^\wedge \smallsetminus \{0\}$, and consider the induced endomorphism $\varphi_{\hDel_s}^\wedge(x)$. Let $\mathscr{C}$ be the cone of this morphism; then we need to show that $\mathscr{C}$ is concentrated in non-negative perverse degrees, or in other words that for any $t \in \mathcal{S}$ the complex $j_t^! \mathscr{C}$ belongs to ${}^p \hspace{-1pt} \hD_{\mathcal{S}}(X_t \quot A, \bk)^{\geq 0}$. Fix $t \in \mathcal{S}$, and denote by $M$ the inverse image of the complex $j_t^! \hDel_s$ under the equivalence~\eqref{eqn:equiv-perverse-stratum} (for the stratum labelled by $t$); then the inverse image of $j_t^! \mathscr{C}$ is the cone of the endomorphism of $M$ induced by the action of $x$.
 
 Using~\eqref{eqn:Pidag-hD-hN} we see that $(\pi_t)_\dag(j_t^! \hDel_s) \cong i_t^! \Delta_s$. Since $\Delta_s$ is perverse this complex is concentrated in non-negative perverse degrees, which implies that the complex of vector spaces $\bk \lotimes_{R_A^\wedge} M$ is concentrated in non-negative degrees. Hence, by Lemma~\ref{lem:nakayama}, $M$ is isomorphic to a complex $N$ of free $R_A^\wedge$-modules with $N^i=0$ for all $i<0$. It is clear that the cone of the endomorphism of $N$ induced by the action of $x$ has cohomology only in non-negative degrees; therefore the same is true for $M$, and finally $j_t^! \mathscr{C}$ indeed belongs to ${}^p \hspace{-1pt} \hD_{\mathcal{S}}(X_t \quot A, \bk)^{\geq 0}$.
\end{proof}

\begin{cor}\phantomsection
\label{cor:Hom-fg}
\begin{enumerate}
\item
\label{it:Hom-fg}
For any $\mathscr{F},\mathscr{G}$ in $\hD_{\mathcal{S}}(X \quot A, \bk)$, the $R_A^\wedge$-module
\[
\Hom_{\hD_{\mathcal{S}}(X \quot A, \bk)}(\mathscr{F},\mathscr{G})
\]
is finitely generated.
\item
\label{it:Krull-Schmidt}
The category $\hD_{\mathcal{S}}(X \quot A, \bk)$ is Krull--Schmidt.
\end{enumerate}
\end{cor}

\begin{proof}
\eqref{it:Hom-fg}
 Lemma~\ref{lem:properties-hD-hN}\eqref{it:hD-hN-generate} reduces the claim to the special case $\mathscr{F}=\hDel_s$, $\mathscr{G}=\hnab_t$ for some $s,t \in \mathcal{S}$, which is clear from~\eqref{eqn:Hom-hDel-hnab}.
 
 \eqref{it:Krull-Schmidt}
 Since the triangulated category $\hD_{\mathcal{S}}(X \quot A, \bk)$ admits a bounded t-structure, it is Karoubian by~\cite{lc}. By~\eqref{it:Hom-fg} and~\cite[Example~23.3]{lam}, the endomorphism ring of any of its objects is semi-local. By~\cite[Theorem~A.1]{cyz}, this implies that $\hD_{\mathcal{S}}(X \quot A, \bk)$ is Krull--Schmidt.
\end{proof}

The standard objects also allow one to describe the perverse t-structure on the category $\hD_{\mathcal{S}}(X \quot A, \bk)$, as follows.

\begin{lem}
\label{lem:perverse-hDel}
The subcategory
${}^p \hspace{-1pt} \hD_{\mathcal{S}}(X \quot A, \bk)^{\leq 0}$ is generated under extensions by the objects of the form $\hDel_s [n]$ for $s \in \mathcal{S}$ and $n \geq 0$.
\end{lem}

\begin{proof}
This claim follows from the yoga of recollement, starting from the observation that the subcategory $\Db \Mod^{\mathrm{fg}}(R_A^\wedge)^{\leq 0}$ is generated under extensions by the objects of the form $R_A^\wedge[n]$ with $n \geq 0$. (Here we use the fact that $R_A^\wedge$ is local, so that any finitely generated projective module is free.)
\end{proof}

\begin{rmk}
It is \emph{not} true that the subcategory
${}^p \hspace{-1pt} \hD_{\mathcal{S}}(X \quot A, \bk)^{\geq 0}$ is generated under extensions by the objects of the form $\hnab_s [n]$ for $s \in \mathcal{S}$ and $n \leq 0$. (This is already false if $Y=\mathrm{pt}$ and $r>0$.)
\end{rmk}

\begin{cor}
\label{cor:pidag-exact}
The functor $\pi_\dag$ is right t-exact with respect to the perverse t-structures.
\end{cor}

\begin{proof}
This follows from Lemma~\ref{lem:perverse-hDel} and~\eqref{eqn:Pidag-hD-hN}.
\end{proof}

\subsection{Tilting perverse sheaves}

It is a standard fact (see e.g.~\cite{bgs}) that under our assumptions the category $\Perv_{\mathcal{S}}(Y,\bk)$ of $\mathcal{S}$-constructible perverse sheaves on $Y$ is a highest weight category, with weight poset $\mathcal{S}$ (for the order induced by inclusions of closures of strata), standard objects $(\Delta_s : s \in \mathcal{S})$, and costandard objects $(\nabla_s : s \in \mathcal{S})$. Hence we can consider the tilting objects in this category, i.e.~those which admit both a filtration with subquotients of the form $\Delta_s$ ($s \in \mathcal{S}$) and a filtration with subquotients of the form $\nabla_s$ ($s \in \mathcal{S}$). If $\mathscr{F}$ is a tilting object, the number of occurences of $\Delta_s$, resp.~$\nabla_s$, in a filtration of the first kind, resp.~second kind, does not depend on the choice of filtration, and equals the dimension of $\Hom(\mathscr{F}, \nabla_s)$, resp.~$\Hom(\Delta_s, \mathscr{F})$. This number will be denoted $(\mathscr{F} : \Delta_s)$, resp.~$(\mathscr{F} : \nabla_s)$. The indecomposable tilting objects are parametrized (up to isomorphism) by $\mathcal{S}$; the object corresponding to $s$ will be denoted $\mathscr{T}_s$.

Similarly, an object $\mathscr{F}$ of $\hP_{\mathcal{S}}(X \quot A,\bk)$ will be called tilting if it admits both a filtration with subquotients of the form $\hDel_s$ ($s \in \mathcal{S}$) and a filtration with subquotients of the form $\hnab_s$ ($s \in \mathcal{S}$). From~\eqref{eqn:Hom-hDel-hnab} we see that the number of occurences of $\hDel_s$, resp.~$\hnab_s$, in a filtration of the first kind, resp.~second kind, does not depend on the choice of filtration, and equals the rank of $\Hom(\mathscr{F}, \hnab_s)$, resp.~$\Hom(\hDel_s, \mathscr{F})$, as an $R_A^\wedge$-module. (These modules are automatically free of finite rank.) This number will be denoted $(\mathscr{F} : \hDel_s)$, resp.~$(\mathscr{F} : \hnab_s)$.

It is clear from definitions and~\eqref{eqn:Pidag-hD-hN} that if $\mathscr{F}$ is a tilting object in $\hP_{\mathcal{S}}(X \quot A,\bk)$, then $\pi_\dag(\mathscr{F})$ is a tilting perverse sheaf, and that moreover
\begin{equation}
\label{eqn:multiplicities-pidag}
 (\pi_\dag(\mathscr{F}) : \Delta_s) = (\mathscr{F} : \hDel_s), \quad (\pi_\dag(\mathscr{F}) : \nabla_s) = (\mathscr{F} : \hnab_s).
\end{equation}

\begin{lem}\phantomsection
\label{lem:tiltings}
 \begin{enumerate}
  \item 
  \label{it:tilting-pi}
  If $\mathscr{F}$ belongs to $\hD_{\mathcal{S}}(X \quot A, \bk)$, then $\mathscr{F}$ is a tilting perverse sheaf iff $\pi_\dag(\mathscr{F})$ is a tilting perverse sheaf.
  \item 
  \label{it:tilting-Hom}
  If $\mathscr{F},\mathscr{G}$ are tilting perverse sheaves in $\hD_{\mathcal{S}}(X \quot A, \bk)$, then we have
  \[
   \Hom_{\hD_{\mathcal{S}}(X \quott A, \bk)}(\mathscr{F},\mathscr{G}[k])=0 \quad \text{if $k \neq 0$,}
  \]
the $R_A^\wedge$-module $\Hom_{\hD_{\mathcal{S}}(X \quot A, \bk)}(\mathscr{F},\mathscr{G})$ is free of finite rank, and the functor $\pi_\dag$ induces an isomorphism
  \[
   \bk \otimes_{R_A^\wedge} \Hom_{\hD_{\mathcal{S}}(X \quott A, \bk)}(\mathscr{F},\mathscr{G}) \simto \Hom_{\Db_{\mathcal{S}}(Y, \bk)}(\pi_\dag \mathscr{F},\pi_\dag \mathscr{G}).
  \]
 \end{enumerate}
\end{lem}

\begin{proof}
 \eqref{it:tilting-pi}
 Using recollement triangles, it is easy to show that $\mathscr{F}$ is a tilting perverse sheaf iff for any $s \in \mathcal{S}$ the objects $j_s^* \mathscr{F}$ and $j_s^! \mathscr{F}$ are direct sums of copies of $\hL_{A,s} [\dim X_s]$ (see~\cite{bbm} for this point of view in the case of usual tilting perverse sheaves). In turn, this condition is equivalent to the requirement that the inverse images of $j_s^* \mathscr{F}$ and $j_s^! \mathscr{F}$ under the equivalence~\eqref{eqn:equiv-perverse-stratum} are isomorphic to a free $R_A^\wedge$-module. It is well known that the latter condition is equivalent to the property that the image under $\bk \lotimes_{R_A^\wedge} (-)$ of these objects is concentrated in degree $0$. We deduce that $\mathscr{F}$ is a tilting perverse sheaf iff for any $s \in \mathcal{S}$ the complexes $(\pi_s)_\dag j_s^* \mathscr{F}$ and $(\pi_s)_\dag j_s^! \mathscr{F}$ are concentrated in perverse degree $0$ (see the proof of Lemma~\ref{lem:perv-pi}). Since
 \[
  (\pi_s)_\dag j_s^* \mathscr{F} \cong i_s^* \pi_\dag \mathscr{F} \quad \text{and} \quad (\pi_s)_\dag j_s^! \mathscr{F} \cong i_s^! \pi_\dag \mathscr{F}
 \]
by~\eqref{eqn:isom-pi-embeddings}, we finally obtain that $\mathscr{F}$ is a tilting perverse sheaf iff the object $\mathscr{G}:=\pi_\dag \mathscr{F}$ is such that for any $s \in \mathcal{S}$ the complexes $i_s^* \mathscr{G}$ and $i_s^! \mathscr{G}$ are concentrated in perverse degree $0$. This condition is equivalent to the fact that $\mathscr{G}$ is a tilting perverse sheaf, see~\cite{bbm}, which concludes the proof.
 
 \eqref{it:tilting-Hom} 
 By Lemma~\ref{lem:monodromy-augmentation}, the morphism
   \[
   \Hom_{\hD_{\mathcal{S}}(X \quott A, \bk)}(\mathscr{F},\mathscr{G}) \to \Hom_{\Db_{\mathcal{S}}(Y, \bk)}(\pi_\dag \mathscr{F},\pi_\dag \mathscr{G})
  \]
  induced by $\pi_\dag$ factors through the quotient $\bk \otimes_{R_A^\wedge} \Hom_{\hD_{\mathcal{S}}(X \quot A, \bk)}(\mathscr{F},\mathscr{G})$. Then
 the desired properties follow from~\eqref{eqn:Hom-hDel-hnab} and the 5-lemma.
\end{proof}

\begin{rmk}
\label{rmk:filtrations-pidag}
The arguments in the proof of Lemma~\ref{lem:tiltings}\eqref{it:tilting-pi} show more generally that if $\mathscr{F}$ belongs to $\hD_{\mathcal{S}}(X \quot A, \bk)$ and if $\pi_\dag(\mathscr{F})$ is a perverse sheaf admitting a standard filtration, then $\mathscr{F}$ is perverse and admits a filtration with subquotients of the form $\hDel_s$ for $s \in \mathcal{S}$, with $\hDel_s$ occuring as many times as $\Delta_s$ occurs in $\pi_\dag(\mathscr{F})$. Of course, a similar claim holds for costandard filtrations.
\end{rmk}

We will denote by $\hT_{\mathcal{S}}(X \quot A, \bk)$ the full subcategory of $\hD_{\mathcal{S}}(X \quot A, \bk)$ whose objects are the tilting perverse sheaves.
Lemma~\ref{lem:tiltings}\eqref{it:tilting-Hom} has the following consequence.

\begin{prop}
\label{prop:realization-equiv}
There exists an equivalence of triangulated categories
\[
\Kb \hT_{\mathcal{S}}(X \quot A, \bk) \simto \hD_{\mathcal{S}}(X \quot A, \bk).
\]
\end{prop}

\begin{proof}
As explained in Remark~\ref{rmk:completed-filtered}, the category $\hD_{\mathcal{S}}(X \quot A, \bk)$ admits a filtered version. Hence, by~\cite[Proposition~2.2]{amrw} (see also~\cite[\S A.6]{beilinson}), there exists a triangulated functor $\Kb \hT_{\mathcal{S}}(X \quot A, \bk) \to \hD_{\mathcal{S}}(X \quot A, \bk)$ whose restriction to $\hT_{\mathcal{S}}(X \quot A, \bk)$ is the natural embedding. The fact that this functor is an equivalence follows from Be{\u\i}linson's lemma.
\end{proof}

\subsection{Classification of tilting perverse sheaves}
\label{ss:classification-tiltings}

It follows from Corollary~\ref{cor:Hom-fg}\eqref{it:Krull-Schmidt} that the category $\hT_{\mathcal{S}}(X \quot A, \bk)$ is Krull--Schmidt. To proceed further, we need to classify its indecomposable objects.

The following classification result is proved in~\cite[Lemma~A.7.3]{by}. Here we provide a different proof, based on some ideas developed in~\cite{rsw} and~\cite[Appendix~B]{modrap1}. (These ideas are themselves closely inspired by the methods of~\cite{bgs}.)

\begin{prop}
\label{prop:classification-tiltings}
 For any $s \in \mathcal{S}$, there exists a unique (up to isomorphism) object $\hTil_s$ in $\hD_{\mathcal{S}}(X \quot A, \bk)$ such that $\pi_\dag(\hTil_s) \cong \mathscr{T}_s$. Moreover, $\hTil_s$ is an indecomposable tilting perverse sheaf, and the assignment $s \mapsto \hTil_s$ induces a bijection between $\mathcal{S}$ and the set of isomorphism classes of indecomposable tilting objects in $\hD_{\mathcal{S}}(X \quot A, \bk)$.
\end{prop}

We begin with the following lemma, where we fix $s \in \mathcal{S}$.

\begin{lem}
\label{lem:prelim-tilting}
 For any open subset $U \subset \overline{Y_s}$ which is a union of strata, there exists a tilting perverse sheaf in $\hD_{\mathcal{S}}(\pi^{-1}(U) \quot A,\bk)$ whose restriction to $X_s$ is $\hL_{A,s} [\dim X_s]$.
\end{lem}

\begin{proof}
We proceed by induction on the number of strata in $U$, the initial case being when $U=Y_s$ (which is of course obvious).
 
 Consider now a general $U$ as in the statement, and $t \in \mathcal{S}$ such that $Y_t \subset U$ and $Y_t$ is closed in $U$. Then we set $V:=U \smallsetminus Y_t$, and assume (by induction) that we have a suitable object $\hTil_V$ in $\hD_{\mathcal{S}}(\pi^{-1}(V) \quot A,\bk)$. We then denote by $j : V \to U$ the embedding, and consider the object $j_! \hTil_V$. This object admits a filtration (in the sense of triangulated categories) whose subquotients are standard objects in $\hD_{\mathcal{S}}(\pi^{-1}(U) \quot A,\bk)$. In particular, it is perverse. We now consider the $R_A^\wedge$-module
 \[
  E:=\Ext^1_{\hP_{\mathcal{S}}(\pi^{-1}(U) \quot A,\bk)}(\hDel_t^U, j_! \hTil_V) = \Hom_{\hD_{\mathcal{S}}(\pi^{-1}(U) \quot A,\bk)}(\hDel_t^U, j_! \hTil_V [1])
 \]
(where $\hDel_t^U$ is the standard object in $\hD_{\mathcal{S}}(\pi^{-1}(U) \quot A,\bk)$ associated with $t$). By Corollary~\ref{cor:Hom-fg}\eqref{it:Hom-fg}, $E$ is finitely generated as an $R_A^\wedge$-module; therefore we can choose a non-negative integer $n$ and a surjection $(R_A^\wedge)^{\oplus n} \twoheadrightarrow E$. This morphism defines an element in
\[
 \Hom_{R_A^\wedge} \left( (R_A^\wedge)^{\oplus n},E \right) \cong E^{\oplus n} \cong \Ext^1_{\hP_{\mathcal{S}}(\pi^{-1}(U) \quot A,\bk)} \left( (\hDel_t^U)^{\oplus n}, j_! \hTil_V \right),
\]
and therefore an extension
\begin{equation}
\label{eqn:extension-hTU}
 j_! \hTil_V \hookrightarrow \hTil_U \twoheadrightarrow (\hDel_t^U)^{\oplus n}
\end{equation}
in $\hP_{\mathcal{S}}(\pi^{-1}(U) \quot A,\bk)$, for some object $\hTil_U$. It is clear that this object admits a filtration with subquotients of the form $\hDel_u$ ($u \in \mathcal{S}$) and has the appropriate restriction to $X_s$. Hence to conclude the proof of the claim, in view of Remark~\ref{rmk:filtrations-pidag}
it suffices to prove that if $\mathscr{T}_U:=(\pi_U)_\dag \hTil_U$ (where $\pi_U$ is the restriction of $\pi$ to $\pi^{-1}(U)$) then $\mathscr{T}_U$ admits a costandard filtration in the highest weight category $\Perv_{\mathcal{S}}(U,\bk)$, or in other words that
\[
 \Ext^1_{\Perv_{\mathcal{S}}(U,\bk)}(\Delta_u^U, \mathscr{T}_U)=0
\]
for any $u \in \mathcal{S}$ such that $Y_u \subset U$. (Here, $\Delta_u^U$ is the standard perverse sheaf in $\Perv_{\mathcal{S}}(U,\bk)$ associated with $u$.)

The case $u \neq t$ is easy, and left to the reader. We then remark that applying the functor $(\pi_U)_\dag$ to~\eqref{eqn:extension-hTU} we obtain an exact sequence
\begin{equation}
\label{eqn;exact-sequence-TU}
 \overline{\jmath}_! \mathscr{T}_V \hookrightarrow \mathscr{T}_U \twoheadrightarrow (\Delta_t^U)^{\oplus n}
\end{equation}
in $\Perv_{\mathcal{S}}(U,\bk)$, where $\overline{\jmath} : V \to U$ is the embedding
and $\mathscr{T}_V:=(\pi_V)_\dag \hTil_V$ for $\pi_V : \pi^{-1}(V) \to V$ the restriction of $\pi$.

We now claim that there exists a canonical isomorphism
\begin{equation}
\label{eqn:E-otimes}
 \bk \otimes_{R_A^\wedge} E \cong \Ext^1_{\Perv_{\mathcal{S}}(U,\bk)}(\Delta_u^U, \overline{\jmath}_! \mathscr{T}_V).
\end{equation}
In fact, using the natural exact sequences
\[
 \ker \hookrightarrow \overline{\jmath}_! \mathscr{T}_V \twoheadrightarrow \overline{\jmath}_{!*} \mathscr{T}_V, \quad \overline{\jmath}_{!*} \mathscr{T}_V \hookrightarrow \overline{\jmath}_{*} \mathscr{T}_V \twoheadrightarrow \mathrm{coker}
\]
and the fact that $\mathscr{T}_V$ admits a standard filtration, it is easily checked that
\[
 \Ext^i_{\Perv_{\mathcal{S}}(U,\bk)}(\Delta_u^U, \overline{\jmath}_! \mathscr{T}_V)=0 \quad \text{for $i \geq 2$}
\]
(see~\cite[Proof of Proposition~B.2]{modrap1} for details). If $M$ is the inverse image of $j_t^! j_! \hTil_V$ under the equivalence~\eqref{eqn:equiv-perverse-stratum}, this means that the complex $\bk \lotimes_{R_A^\wedge} M$ is concentrated in degrees $\leq 1$. This implies that $M$ itself is concentrated in degrees $\leq 1$, and that we have a canonical isomorphism $\bk \otimes_{R_A^\wedge} \mathsf{H}^1(M) \cong \mathsf{H}^1(\bk \lotimes_{R_A^\wedge} M)$. This isomorphism is precisely~\eqref{eqn:E-otimes}.

Once~\eqref{eqn:E-otimes} is established, we see that our surjection $(R_A^\wedge)^{\oplus n} \twoheadrightarrow E$ induces a surjection $\bk^{\oplus n} \twoheadrightarrow \Ext^1_{\Perv_{\mathcal{S}}(U,\bk)}(\Delta_u^U, \overline{\jmath}_! \mathscr{T}_V)$. Using this fact and considering the long exact sequence obtained by applying the functor $\Hom(\Delta_t^U, -)$ to~\eqref{eqn;exact-sequence-TU} we conclude that $\Ext^1_{\Perv_{\mathcal{S}}(U,\bk)}(\Delta_t^U, \mathscr{T}_U)=0$, which finishes the proof.
\end{proof}

\begin{proof}[Proof of Proposition~{\rm \ref{prop:classification-tiltings}}]
 By Lemma~\ref{prop:classification-tiltings}, there exists a tilting object $\hTil_s$ in the category $\hD_{\mathcal{S}}(X \quot A, \bk)$ which is supported on $\overline{X}_s$ and whose restriction to $X_s$ is $\hL_{A,s}$. Of course, we can (and will) further require that this object is indecomposable. By Lemma~\ref{lem:tiltings}, the object $\pi_\dag \hTil_s$ is then a tilting perverse sheaf, and its endomorphism ring is a quotient of $\End(\hTil_s)$, hence is local; in other words, $\pi_\dag \hTil_s$ is indecomposable. Since it is supported on $\overline{Y_s}$, and since its restriction to $Y_s$ is $\underline{\bk}_{Y_s} [\dim(Y_s)]$ it follows that $\pi_\dag(\hTil_s) \cong \mathscr{T}_s$.
 
 These arguments show more generally that if $\hTil$ is any indecomposable tilting object in $\hD_{\mathcal{S}}(X \quot A, \bk)$, the object $\pi_\dag(\hTil)$ is isomorphic to $\mathscr{T}_t$ for some $t \in \mathcal{S}$. To conclude the proof, it remains to prove that in this case we must have $\hTil \cong \hTil_t$. By Lemma~\ref{lem:tiltings}\eqref{it:tilting-Hom}, the functor $\pi_\dag$ induces an isomorphism
 \[
  \bk \otimes_{R_A^\wedge} \Hom_{\hD_{\mathcal{S}}(X \quot A, \bk)}(\hTil, \hTil_t) \simto \Hom_{\Db_{\mathcal{S}}(Y, \bk)}(\pi_\dag \hTil, \pi_\dag \hTil_t).
 \]
Hence there exists a morphism $f : \hTil \to \hTil_t$ such that $\pi_\dag(f)$ is an isomorphism. Then the cone $\mathscr{C}$ of $f$ satisfies $\pi_\dag(\mathscr{C})=0$. By Lemma~\ref{lem:perv-pi}\eqref{it:pidag-conservative} this implies that $\mathscr{C}=0$, hence that $f$ is an isomorphism.
\end{proof}

\part{The case of flag varieties}

\section{Study of tilting perverse objects}
\label{sec:flag-tilting}

\subsection{Notation}

From now on we fix a complex connected reductive algebraic group $G$, and choose a maximal torus and a Borel subgroup $T \subset B \subset G$. We will denote by $U$ the unipotent radical of $B$, and by $W$ the Weyl group of $(G,T)$. The choice of $B$ determines a subset $S \subset W$ of simple reflections, and a choice of positive roots (such that $B$ is the \emph{negative} Borel subgroup).

We will study further the previous constructions in the special case
\[
 X = G/U, \quad Y=G/B
\]
(with the action of $A=T$ given by $t \cdot gU = gtU$),
$\pi : X \to Y$ is the natural projection and the stratification is
\[
 Y = \bigsqcup_{w \in W} Y_w \quad \text{with} \quad Y_w := BwB/B.
\]
The corresponding categories in this case will be denoted
\[
 \Db_{U}(Y,\bk), \quad \hD_{U}(X \quot T,\bk).
\]
(Here, $\Db_{U}(Y,\bk)$ is indeed equivalent to the $U$-equivariant constructible derived category in the sense of Bernstein--Lunts, which explains the notation.)

Recall that to define the objects $\hDel_w$ and $\hnab_w$ we need to choose a $T$-equivariant morphism $X_w \to T$, where $X_w=\pi^{-1}(Y_w)$. For this we choose a lift $\dot w$ of $w$ in $N_G(T)$, and consider the subgroup $U_{w^{-1}} \subset U$ defined as in~\cite[Lemma~8.3.5]{springer}. Then the map $u \mapsto u \dot w B$ induces an isomorphism $U_w \simto Y_w$, and the map $(u,t) \mapsto u \dot w t U$ induces an isomorphism $U_w \times T \simto X_w$, see~\cite[Lemma~8.3.6]{springer}. We will choose $p_w$ as the composition of the inverse isomorphism with the projection to the $T$ factor.

The category $\Db_{U}(Y,\bk)$ admits a natural perverse t-structure; its heart will be denoted
\[
\scO := \Perv_{U}(Y,\bk).
\]
Similarly, the constructions of~\S\ref{ss:def-perv} provide a perverse t-structure on the category $\hD_{U}(X \quot T,\bk)$, whose heart will be denoted
\[
\hscO := \Perv_{U}(X \quot T, \bk).
\]

\subsection{Right and left monodromy}
\label{ss:lr-monodromy}

By the general formalism of the completed monodromic category (see~\S\ref{ss:definition-hD}), for any $\mathscr{F}$ in $\hD_{U}(X \quot T, \bk)$ we have an algebra morphism
\[
 \varphi^{\wedge}_{\mathscr{F}} : R_T^\wedge \to \End(\mathscr{F}).
\]
Since this monodromy comes from the action of $T$ by right multiplication, we will denote it in this case by $\varphi^{\wedge}_{\rig,\mathscr{F}}$. 

Now, let $a : G \to G/U$ be the projection (a locally trivial fibration, with fibers isomorphic to affine spaces). Then the functor $a^* : \Db_c(Y,\bk) \to \Db_c(G,\bk)$ is fully-faithful since $a_* \circ a^* \cong \id$.
The triangulated category $\Db_{U}(X,\bk)$ is generated by the image of the forgetful functor $\Db_B(X,\bk) \to \Db_c(X,\bk)$; therefore, if
$\mathscr{G}$ belongs to $\Db_{U}(X \quot T,\bk)$ then $a^*(\mathscr{G})$ belongs to the monodromic category $\Db_c(G \quot T,\bk)$ where $T$ acts on $G$ via $t \cdot g = tg$. Hence we can consider the morphism $\varphi_{a^*(\mathscr{G})}^\wedge$. Since $a^*$ is fully-faithful, this morphism can be interpreted as a morphism
\[
 \varphi_{\lef,\mathscr{G}}^\wedge : R_T^\wedge \to \End(\mathscr{G})
\]
(where ``l'' stands for left).  Passing to projective limits we deduce, for any $\mathscr{F}$ in $\hD_{U}(X \quot T, \bk)$, an algebra morphism
\[
 \varphi^{\wedge}_{\lef,\mathscr{F}} : R_T^\wedge \to \End(\mathscr{F}).
\]

Combining these two constructions, we obtain an algebra morphism
\[
 \varphi^{\wedge}_{\lef\rig,\mathscr{F}} : R_T^\wedge \otimes_\bk R_T^\wedge \to \End(\mathscr{F})
\]
sending $r \otimes r'$ to $\varphi^{\wedge}_{\lef,\mathscr{F}}(r) \circ \varphi^{\wedge}_{\rig,\mathscr{F}} (r') = \varphi^{\wedge}_{\rig,\mathscr{F}}(r') \circ \varphi^{\wedge}_{\lef,\mathscr{F}}(r)$.

\begin{lem}
\label{lem:mondromy-leftright-hDel}
 For any $w \in W$, the morphism $\varphi_{\rig,\hDel_w}^\wedge$, resp.~$\varphi_{\rig,\hnab_w}^\wedge$, is the composition of $\varphi_{\lef,\hDel_w}^\wedge$, resp.~$\varphi_{\lef,\hnab_w}^\wedge$, with the automorphism of $R_T^\wedge$ induced by $w$.
\end{lem}

\begin{proof}
 We treat the case of $\hDel_w$; the case of $\hnab_w$ is similar. More precisely we will prove a similar claim for the monodromy endomorphisms of each object $\Delta_w^n := (j_w)_! p_w^*(\mathscr{L}_{T,n})[\dim X_w]$.
 
 By the base change theorem we have
 \[
  a^*(\Delta^n_w) \cong (\widetilde{\jmath}_w)_! (p_w \circ a_w)^* \mathscr{L}_{T,n}[\dim X_w],
 \]
where $\widetilde{\jmath}_w : a^{-1}(X_w) \hookrightarrow G$ is the embedding and $a_w : a^{-1}(X_w) \to X_w$ is the restriction of $a$.
 By Lemma~\ref{lem:properties-monodromy}, we deduce that for any $r \in R^\wedge_T$ we have
 \begin{equation}
  \label{eqn:monodromy-hDel-right}
  a^* \bigl( \varphi^\wedge_{\rig,\Delta^n_w}(r) \bigr) = \varphi^\wedge_{\rig,a^*(\Delta^n_w)}(r) = (\widetilde{\jmath}_w)_! (p_w \circ a_w)^* \varphi^\wedge_{\mathscr{L}_{T,n}}(r) [\dim X_w],
 \end{equation}
 where in the first two terms we consider the monodromy operation with respect to the action of $T$ on $G/U$ and $G$ by multiplication on the right.
 
 Now we consider the actions induced by multiplication on the left. It is not difficult to check that
 \[
  (p_w \circ a_w)(t \cdot x) = w^{-1}(t) (p_w \circ a_w)(x)
 \]
for any $t \in T$ and $x \in a^{-1}(X_w)$. In other words, $p_w \circ a_w$ is $T$-equivariant when $T$ acts on $a^{-1}(X_w)$ by multiplication on the left, and on $T$ via the natural action twisted by $w^{-1}$. From this, using the same arguments as above and Lemma~\ref{lem:monodromy-isom}, we deduce that
\begin{equation}
\label{eqn:monodromy-hDel-left}
 a^* \bigl( \varphi^{\wedge}_{\lef, \Delta^n_w}(r) \bigr) = (\widetilde{\jmath}_w)_! (p_w \circ a_w)^* \varphi^\wedge_{\mathscr{L}_{T,n}}(w^{-1} (r)) [\dim X_w].
\end{equation}
Comparing~\eqref{eqn:monodromy-hDel-right} and~\eqref{eqn:monodromy-hDel-left}, and using the fact that $a^*$ is fully-faithful, we deduce the desired claim.
\end{proof}

Similar considerations hold for objects in $\Db_{U}(Y,\bk)$. Below we will only consider the case of perverse sheaves, so we restrict to this setting. Let $b=\pi \circ a : G \to Y$ be the natural projection, and let $\mathscr{F}$ in $\scO$. Then the object $b^*(\mathscr{F})$ belongs to $\Db_c(G \quot T, \bk)$, where the $T$-action on $G$ is induced by multiplication on the \emph{left}. Hence the monodromy construction from Section~\ref{sec:monodromy} provides a morphism $R_T^\wedge \to \End(b^*(\mathscr{F}))$. Now the functor $b^*$ is fully-faithful on perverse sheaves since $b$ is smooth with connected fibers (see~\cite[Proposition~4.2.5]{bbd}); hence this morphism can be considered as an algebra morphism
\[
\varphi^\wedge_{\lef,\mathscr{F}} : R_T^\wedge \to \End(\mathscr{F}).
\]
It is clear that if $\mathscr{F}$ belongs to $\hD_{U}(X \quot T, \bk)$ and $\pi_\dag(\mathscr{F})$ is perverse, the composition
\[
 R_T^\wedge \xrightarrow{\varphi^\wedge_{\lef,\mathscr{F}}} \End(\mathscr{F}) \xrightarrow{\pi_\dag} \End(\pi_\dag(\mathscr{F}))
\]
coincides with $\varphi^\wedge_{\lef,\pi_\dag(\mathscr{F})}$.

\subsection{The associated graded functor}

Let us now fix a total order $\preceq$ on $W$ that refines the Bruhat order. We then denote by $j_{\prec w}$ the embedding of the closed subvariety $\bigsqcup_{y \prec w} X_y$ in $X$. For any $\hTil$ in $\hT_{U}(X \quot T, \bk)$, the adjunction morphism
\[
 \hTil \to (j_{\prec w})_* (j_{\prec w})^* \hTil
\]
is surjective. If we denote its kernel by $\hTil_{\succeq w}$, then the family of subobjects of $\hTil$ given by $(\hTil_{\succeq w})_{w \in W}$ is an exhaustive filtration on $\hTil$ indexed by $W$, endowed with the order opposite to $\preceq$ (meaning that $\hTil_{\succeq w} \subset \hTil_{\succeq y}$ if $y \preceq w$). Moreover, if we set
\[
 \gr_w(\hTil):= \hTil_{\succeq w} / \hTil_{\succeq w'},
\]
where $w'$ is the successor of $w$ for $\preceq$, then $\gr_w(\hTil)$ is a direct sum of copies of $\hDel_w$. (Here by convention $\hTil_{\succeq w'}=0$ if $w$ has no successor, i.e.~if $w$ is the longest element in $W$.) Since by adjunction we have $\Hom_{\hD_{U}(X \quott T, \bk)}(\hDel_y, \hDel_w)=0$ if $y \succ w$, we see that if $f : \hTil \to \hTil'$ is a morphism in $\hT_{U}(X \quot T, \bk)$, then $f(\hTil_{\succeq w}) \subset \hTil'_{\succeq w}$ for any $w \in W$. In other words, the assignment $\hTil \mapsto \hTil_{\succeq w}$ is functorial. This allows us to define the functor
\[
 \gr : \left\{
 \begin{array}{ccc}
  \hT_{U}(X \quot T, \bk) & \to & \hP_{U}(X \quot T, \bk) \\
  \hTil & \mapsto & \bigoplus_{w \in W} \gr_w(\hTil)
 \end{array}
 \right. .
\]
This functor is clearly additive.

\begin{lem}
\label{lem:Hom-vanishing-hDel}
 For any $y,w \in W$ with $y \neq w$, we have
 \[
  \Hom_{\hD_{U}(X \quott T, \bk)}(\hDel_y, \hDel_w)=0.
 \]
\end{lem}

\begin{proof}
 Let $f : \hDel_y \to \hDel_w$ be a nonzero morphism. We denote by $\mathscr{F}$ the image of $f$, and write $f=f_1 \circ f_2$ with $f_2 : \hDel_y \to \mathscr{F}$ the natural surjection and $f_1 : \mathscr{F} \to \hDel_w$ the natural embedding. Then for any $r \in R^\wedge_T$ we have a commutative diagram
 \[
  \xymatrix@C=3cm{
  \hDel_y \ar[r]^-{f_2} \ar@/^/[d]^-{\varphi^\wedge_{\rig,\hDel_y}(r)} \ar@/_/[d]_-{\varphi^\wedge_{\lef,\hDel_y}(r)} & \mathscr{F} \ar[r]^-{f_1} \ar@/^/[d]^-{\varphi^\wedge_{\rig,\mathscr{F}}(r)} \ar@/_/[d]_-{\varphi^\wedge_{\lef,\mathscr{F}}(r)} & \hDel_w \ar@/^/[d]^-{\varphi^\wedge_{\rig,\hDel_w}(r)} \ar@/_/[d]_-{\varphi^\wedge_{\lef,\hDel_w}(r)} \\
  \hDel_w \ar[r]^-{f_2} & \mathscr{F} \ar[r]^-{f_1} & \hDel_w 
  }
 \]
By Lemma~\ref{lem:properties-hD-hN}, if $r \neq 0$ then $\varphi^\wedge_{\rig,\hDel_w}(r)$ is injective. Hence 
\begin{equation}
\label{eqn:monodromy-injective}
 \varphi^\wedge_{\rig,\mathscr{F}}(r) \text{ is injective (in particular, nonzero) if $r \neq 0$.}
\end{equation}

On the other hand, using Lemma~\ref{lem:mondromy-leftright-hDel} we see that
\[
 f_1 \circ \varphi^\wedge_{\rig,\mathscr{F}}(r) = \varphi^\wedge_{\rig,\hDel_w}(r) \circ f_1 = \varphi^\wedge_{\lef,\hDel_w}(w(r)) \circ f_1 = f_1 \circ \varphi^\wedge_{\lef,\mathscr{F}}(w(r)),
\]
which implies that $\varphi^\wedge_{\rig,\mathscr{F}}(r) = \varphi^\wedge_{\lef,\mathscr{F}}(w(r))$ since $f_1$ is injective, and that
\[
 \varphi^\wedge_{\rig,\mathscr{F}}(r) \circ f_2 = f_2 \circ \varphi^\wedge_{\rig,\hDel_y}(r) = f_2 \circ \varphi^\wedge_{\lef,\hDel_y}(y(r)) = \varphi^\wedge_{\lef,\mathscr{F}}(y(r)) \circ f_2,
\]
which implies that $\varphi^\wedge_{\rig,\mathscr{F}}(r) = \varphi^\wedge_{\lef,\mathscr{F}}(y(r))$ since $f_2$ is surjective. Comparing these two equations, we deduce that
$\varphi^\wedge_{\rig,\mathscr{F}}(r) = \varphi^\wedge_{\rig,\mathscr{F}}(y^{-1}w(r))$, or in other words that
\[
 \varphi^\wedge_{\rig,\mathscr{F}}(r - y^{-1}w(r))=0,
\]
for any $r \in R_T^\wedge$. In view of~\eqref{eqn:monodromy-injective}, this implies that $r=y^{-1} w(r)$ for any $r \in R_T^\wedge$, hence that $y=w$.
\end{proof}

As a consequence we obtain the following claim.

\begin{cor}
\label{cor:gr-faithful}
 The functor $\gr$ is faithful.
\end{cor}

\begin{proof}
 Let $\hTil,\hTil'$ be in $\hT_{U}(X \quot T, \bk)$, and let $f : \hTil \to \hTil'$ be a nonzero morphism. Let $w \in W$ be an element which is maximal with respect to the property that $f(\hTil_{\succeq w}) \neq 0$. Then $f$ induces a nonzero morphism $\tilde{f}_w : \gr_w(\hTil) \to \hTil'$. We have $f(\hTil_{\succeq w}) \subset \hTil'_{\succeq w}$, hence $\tilde{f}_w$ factors through a nonzero morphism $\gr_w(\hTil) \to \hTil_{\succeq w}'$. Lemma~\ref{lem:Hom-vanishing-hDel} implies that the natural morphism
 \[
  \Hom(\gr_w(\hTil), \hTil_{\succeq w}') \to \Hom(\gr_w(\hTil), \gr_w(\hTil'))
 \]
is injective; hence $\gr_w(f) \neq 0$, so that a fortiori $\gr(f) \neq 0$.
\end{proof}

Note that, by functoriality of monodromy, for any $r \in R^\wedge_T$ we have
\begin{equation}
 \label{eqn:monodromy-gr}
 \varphi_{\lef,\gr(\hTil)}^\wedge(r) = \gr \bigl( \varphi_{\lef,\hTil}^\wedge(r) \bigr), \qquad \varphi_{\rig,\gr(\hTil)}^\wedge(r) = \gr \bigl( \varphi_{\rig,\hTil}^\wedge(r) \bigr).
\end{equation}

\subsection{Monodromy and coinvariants}

\begin{prop}
\label{prop:mon-tilting}
  For any $\hTil$ in $\hT_{U}(X \quot T, \bk)$, the morphism $\varphi^\wedge_{\lef\rig,\hTil}$ factors through an algebra morphism
  \[
   R_T^\wedge \otimes_{(R_T^\wedge)^W} R_T^\wedge \to \End(\hTil).
  \]
\end{prop}

\begin{proof}
 We have to prove that $\varphi_{\lef,\hTil}^\wedge(r) = \varphi_{\rig,\hTil}^\wedge(r)$ for any $r \in (R_T^\wedge)^W$. Since the functor $\gr$ is faithful (see Corollary~\ref{cor:gr-faithful}), for this it suffices to prove that $\gr(\varphi_{\lef,\hTil}^\wedge(r)) = \gr(\varphi_{\rig,\hTil}^\wedge(r))$. This equality follows from~\eqref{eqn:monodromy-gr} and Lemma~\ref{lem:mondromy-leftright-hDel}, since $\gr(\hTil)$ is a direct sum of copies of objects $\hDel_w$.
\end{proof}

\subsection{The case of \texorpdfstring{$\hTil_{s}$}{Ts}}
\label{ss:Ts}

In this subsection we 
fix a simple reflection $s$, and denote by $\alpha$ the associated simple root.

We consider the closure $\overline{Y_s} = Y_s \sqcup Y_e$. This subvariety of $Y$ is isomorphic to $\mathbb{P}^1$, in such a way that $Y_e$ identifies with $\{0\}$. The structure of the category $\Perv_{U}(\overline{Y_s}, \bk)$ of $\bk$-perverse sheaves on $\overline{Y_s}$ constructible with respect to the stratification $\overline{Y_s} = Y_s \sqcup Y_e$ is well known: this category admits 5 indecomposable objects (up to isomorphism):
\begin{itemize}
\item two simple objects $\IC_e$ and $\IC_s$;
\item two indecomposable objects of length $2$, namely $\Delta_s$ and $\nabla_s$, which fit into nonsplit exact sequences
\[
\IC_e \hookrightarrow \Delta_s \twoheadrightarrow \IC_s, \qquad \IC_s \hookrightarrow \nabla_s \twoheadrightarrow \IC_e;
\]
\item one indecomposable object of length $3$, namely the tilting object $\mathscr{T}_s$, which fits into nonsplit exact sequences
\[
\Delta_s \hookrightarrow \mathscr{T}_s \twoheadrightarrow \IC_e, \qquad \IC_e \hookrightarrow \mathscr{T}_s \twoheadrightarrow \nabla_s.
\]
\end{itemize}

We now fix a cocharacter $\lambda : \C^\times \to T$, and consider the full subcategory $\Perv_{\C^\times,U}(\overline{Y_s}, \bk) \subset \Perv_{U}(\overline{Y_s}, \bk)$ consisting of perverse sheaves which are $\C^\times$-equi\-variant for the action determined by $z \cdot xB = \lambda(z) x B$.

\begin{lem}
\label{lem:Ts-equivariance}
If the image of $\langle \lambda,\alpha \rangle$ in $\bk$ is nonzero, then $\mathscr{T}_s$ does not belong to $\Perv_{\C^\times,U}(\overline{Y_s}, \bk)$.
\end{lem}

\begin{proof}
Let $B^+ \subset G$ be the Borel subgroup opposite to $B$ with respect to $T$, let $U^+$ be its unipotent radical, and let $U_s^+ \subset U^+$ be the root subgroup associated with $s$. If we set $Y_s^\circ := \overline{Y_s} \smallsetminus \{sB\}$, then the map $u \mapsto u \cdot B$ induces an isomorphism $U_s^+ \simto Y_s^\circ$. In particular, this open subset is $\C^\times$-stable, with an action of $\C^\times$ via the character $\langle \lambda,\alpha \rangle$.

The object $\mathscr{T}_s$ is the unique nonsplit extension of $\IC_e$ by $\Delta_s$ in $\Perv_{U}(\overline{Y_s}, \bk)$; hence to conclude it suffices to show that $\Ext^1_{\Perv_{\C^\times,U}(\overline{Y_s}, \bk)}(\IC_e,\Delta_s)=0$ if the image of $\langle \lambda,\alpha \rangle$ in $\bk$ is nonzero. Note that we have
\[
\Ext^1_{\Perv_{\C^\times,U}(\overline{Y_s}, \bk)}(\IC_e,\Delta_s) = \Hom_{\Db_{\C^\times,U}(\overline{Y_s}, \bk)}(\IC_e,\Delta_s[1])
\]
where $\Db_{\C^\times,U}(\overline{Y_s}, \bk)$ is the $\C^\times$-equivariant constructible derived category in the sense of Bern\-stein--Lunts.
Let us consider the long exact sequence
\begin{multline*}
\Hom_{\Db_{\C^\times,U}(\overline{Y_s}, \bk)}(\IC_e,\IC_e[1]) \to \Hom_{\Db_{\C^\times,U}(\overline{Y_s}, \bk)}(\IC_e,\Delta_s[1]) \\
\to \Hom_{\Db_{\C^\times,U}(\overline{Y_s}, \bk)}(\IC_e,\IC_s[1]) \to \Hom_{\Db_{\C^\times,U}(\overline{Y_s}, \bk)}(\IC_e,\IC_e[2])
\end{multline*}
obtained from the short exact sequence $\IC_e \hookrightarrow \Delta_s \twoheadrightarrow \IC_s$. Here the first, resp.~four\-th, term identifies with the degree-$1$, resp.~degree-$2$, $\C^\times$-equivariant cohomology of the point. In particular, this term vanishes, resp.~is canonically isomorphic to $\bk$. Now, we observe that restriction induces an isomorphism
\[
\Hom_{\Db_{\C^\times,U}(\overline{Y_s}, \bk)}(\IC_e,\IC_s[1]) \simto \Hom_{\Db_{\C^\times,U}(Y_s^\circ, \bk)}(\underline{\bk}_{Y_e},\underline{\bk}_{Y_s^\circ}[2]).
\]
The right-hand side is $1$-dimensional, with a basis consisting of the adjunction morphism associated with the embedding $Y_e \hookrightarrow Y_s^\circ$. Moreover, in view of the classical description of the $\C^\times$-equivariant cohomology of the point recalled e.g.~in~\cite[\S 1.10]{lusztig}, the map
\[
\Hom_{\Db_{\C^\times,U}(\overline{Y_s}, \bk)}(\IC_e,\IC_s[1]) \to \Hom_{\Db_{\C^\times,U}(\overline{Y_s}, \bk)}(\IC_e,\IC_e[2])
\]
considered above identifies with the map $\bk \to \bk$ given by multiplication by $\langle \lambda,\alpha \rangle$. Our assumption is precisely that this map is injective; we deduce that the vector space $\Ext^1_{\Perv_{\C^\times,U}(\overline{Y_s}, \bk)}(\IC_e,\Delta_s)$ vanishes, as claimed.
\end{proof}

\begin{cor}
\label{cor:mon-Ts}
Assume that there exists $\lambda \in X_*(T)$ such that the image of $\langle \lambda, \alpha \rangle$ in $\bk$ is nonzero. Then in the special case $\hTil=\hTil_s$, the morphism
   \[
   R_T^\wedge \otimes_{(R_T^\wedge)^W} R_T^\wedge \to \End(\hTil_s)
  \]
  of Proposition~{\rm \ref{prop:mon-tilting}} is surjective.
\end{cor}

\begin{proof}
 By Nakayama's lemma and Lemma~\ref{lem:tiltings}\eqref{it:tilting-Hom}, it suffices to prove that the morphism
 \[
  \varphi^\wedge_{\lef,\mathscr{T}_s} : R_T^\wedge \to \End(\mathscr{T}_s)
 \]
 of~\S\ref{ss:lr-monodromy} is surjective. Now we have $\dim(\End(\mathscr{T}_s))=2$, hence for this it suffices to prove that the image of $\varphi^\wedge_{\lef,\mathscr{T}_s}$ is not reduced to $\bk \cdot \id_{\mathscr{T}_s}$. However, if $\lambda \in X_*(T)$ is such that $\langle \lambda, \alpha \rangle \neq 0$ in $\bk$, 
 then by Lemma~\ref{lem:monodromy-isom}, Lemma~\ref{lem:mon-equiv} and Lemma~\ref{lem:Ts-equivariance}, the automorphism $\varphi^\wedge_{\lef,\mathscr{T}_s}(\lambda)$ is unipotent but not equal to $\id_{\mathscr{T}_s}$; therefore it does not belong to $\bk \cdot \id_{\mathscr{T}_s}$, and the claim is proved.
\end{proof}

\subsection{Properties of \texorpdfstring{$\mathscr{T}_{w_0}$}{Tw0}}

We finish this section with a reminder of some properties of the category $\mathscr{O}$ which are well known (at least in the case $\mathrm{char}(\bk)=0$).

The following claim is fundamental. It is proved in~\cite[Lemma in~\S 2.1]{bbm} under the assumption that $\mathrm{char}(\bk)=0$; but the arguments apply in full generality.

\begin{lem}
\label{lem:soc-top-Delta-nabla}
For any $w \in W$, the socle of the object $\Delta_w$ is $\IC_e$, and all the composition factors of $\Delta_w / \mathrm{soc}(\Delta_w)$ are of the form $\IC_v$ with $v \neq e$. Dually, the top of the object $\nabla_w$ is $\IC_e$, and all the composition factors of the kernel of the surjection $\nabla_w \to \mathrm{top}(\nabla_w)$ are of the form $\IC_v$ with $v \neq e$.
\end{lem}

This lemma has the following important consequence.

\begin{cor}
\label{cor:sub-quo-filtered}
If $\mathscr{F}$ is an object of $\scO$ which admits a standard filtration, then its socle is a direct sum of copies of $\IC_e$.
In other words, any nonzero subobject of $\mathscr{F}$ admits $\IC_e$ as a composition factor. Dually, if $\mathscr{F}$ is an object of $\scO$ which admits a costandard filtration, then its top is a direct sum of copies of $\IC_e$.
In other words, any nonzero quotient of $\mathscr{F}$ admits $\IC_e$ as a composition factor.
\end{cor}

To finish this section we recall the main properties of the object $\mathscr{T}_{w_0}$ that we will need in Section~\ref{sec:main-thm}.

\begin{lem} \phantomsection
\label{lem:properties-Tw0}
 \begin{enumerate}
  \item 
  \label{it:Tw0-1}
  For any $w \in W$ we have $(\mathscr{T}_{w_0} : \Delta_w)=1$.
  \item
  \label{it:Tw0-2}
  The object $\mathscr{T}_{w_0}$ is both the projective cover and the injective hull of $\IC_e$ in $\scO$.
 \end{enumerate}
\end{lem}

\begin{proof}
 Both of these claims are consequences of Lemma~\ref{lem:soc-top-Delta-nabla}. For details, see~\cite[Lemma~5.25]{modrap1} for~\eqref{it:Tw0-1}, and~\cite[Proposition~5.26]{modrap1} for~\eqref{it:Tw0-2}.
\end{proof}

\begin{lem}
\label{lem-Tw0-s}
 Let $s$ be a simple reflection, and let $\overline{\imath}_s : \overline{Y_s} \to Y$ be the embedding. Then we have $\overline{\imath}_s^*(\mathscr{T}_{w_0}) \cong \mathscr{T}_s$.
\end{lem}

\begin{proof}
 Since $\mathscr{T}_{w_0}$ is tilting (in particular, admits a standard filtration), the object $\overline{\imath}_s^*(\mathscr{T}_{w_0})$ is perverse and admits a standard filtration. More precisely, in view of Lemma~\ref{lem:properties-Tw0}\eqref{it:Tw0-1} we have
 \[
  (\overline{\imath}_s^*(\mathscr{T}_{w_0}) : \Delta_e) = (\overline{\imath}_s^*(\mathscr{T}_{w_0}) : \Delta_s) = 1.
 \]
By the description of the indecomposable objects of $\Perv_{U}(\overline{Y_s},\bk)$ recalled in~\S\ref{ss:Ts}, we deduce that $\overline{\imath}_s^*(\mathscr{T}_{w_0})$ is isomorphic to either $\mathscr{T}_s$ or $\Delta_e \oplus \Delta_s$. However, we have
\[
 \Hom(\overline{\imath}_s^*(\mathscr{T}_{w_0}), \IC_s) = \Hom(\mathscr{T}_{w_0}, \IC_s)=0
\]
by adjunction and Lemma~\ref{lem:properties-Tw0}\eqref{it:Tw0-2} respectively; hence this object cannot admit $\Delta_s$ as a direct summand.
\end{proof}

\section{Convolution}
\label{sec:convolution}

\subsection{Definition}

Let us denote by
\[
 m : G \times^U X \to X
\]
the map defined by $m([g : hU]) = ghU$. If $\mathscr{F}, \mathscr{G}$ belong to $\Db_U(X,\bk)$, there exists a unique object $\mathscr{F} \wbtimes \mathscr{G}$ in $\Db_U(G \times^U X, \bk)$ whose pullback under the quotient map $G \times X \to G \times^U X$ is $a^*(\mathscr{F}) \boxtimes \mathscr{G}$ (where $a$ is as in~\S\ref{ss:lr-monodromy}). We then set
\[
 \mathscr{F} \star^U \mathscr{G} := m_!(\mathscr{F} \wbtimes \mathscr{G}) [\dim T].
\]
This construction defines a functor $\Db_U(X,\bk) \times \Db_U(X,\bk) \to \Db_U(X,\bk)$, which is associative up to (canonical) isomorphism.

Similarly, we denote by
\[
 m' : G \times^U Y \to Y
\]
the map defined by $m([g : hB]) = ghB$. If $\mathscr{F}$ belongs to $\Db_U(X,\bk)$ and $\mathscr{G}$ belongs to $\Db_U(Y,\bk)$, there exists a unique object $\mathscr{F} \wbtimes \mathscr{G}$ in $\Db_U(G \times^U Y, \bk)$ whose pullback under the quotient map $G \times Y \to G \times^U Y$ is $a^*(\mathscr{F}) \boxtimes \mathscr{G}$. We then set
\[
 \mathscr{F} \star^U \mathscr{G} := m'_!(\mathscr{F} \wbtimes \mathscr{G}) [\dim T].
\]
This construction defines a functor $\Db_U(X,\bk) \times \Db_U(Y,\bk) \to \Db_U(Y,\bk)$, which is compatible with the product $\star^U$ on $\Db_U(X,\bk)$ in the obvious sense.

\begin{rmk}
Since the quotient $G/U$ is not proper, there exist two possible conventions to define the convolution product on $\Db_U(X,\bk)$: one involving the functor $m_!$, and one involving the functor $m_*$. We insist that here we consider the version with $!$-pushforward.
\end{rmk}

It is straightforward (using the base change theorem) to check that for $\mathscr{F},\mathscr{G}$ in $\Db_U(X,\bk)$ and $\mathscr{G}'$ in $\Db_U(Y,\bk)$ there exist canonical isomorphisms
\begin{gather}
\label{eqn:conv-pi_!}
\pi_!(\mathscr{F} \star^U \mathscr{G}) \cong \mathscr{F} \star^U \pi_!(\mathscr{G}), \\
\label{eqn:conv-pi^!}
\pi^!(\mathscr{F} \star^U \mathscr{G}') \cong \mathscr{F} \star^U \pi^!(\mathscr{G}'), \\
\label{eqn:conv-pi^*}
\pi^*(\mathscr{F} \star^U \mathscr{G}') \cong \mathscr{F} \star^U \pi^*(\mathscr{G}').
\end{gather}

Instead of the $U$-equivariant categories, one can also consider the $B$-equivariant categories. In particular, very similar considerations lead to the definition of a functor
\[
(-) \star^B (-) : \Db_U(Y,\bk) \times \Db_B(Y,\bk) \to \Db_U(Y,\bk).
\]
(Here we do not insert any cohomological shift in the definition. Note also that, since $G/B$ is proper, there is no difference between the $*$- and $!$-versions of convolution.)

We will denote by $\For^B_U : \Db_B(Y,\bk) \to \Db_U(Y,\bk)$ the natural forgetful functor. The following fact is standard.

\begin{lem}
\label{lem:convolution-Bequiv}
For any $\mathscr{F}$ in $\Db_U(X,\bk)$ and $\mathscr{G}$ in $\Db_B(Y,\bk)$, there exists a canonical isomorphism
\[
\mathscr{F} \star^U \For^B_U(\mathscr{G}) \cong \pi_\dag(\mathscr{F}) \star^B \mathscr{G}.
\]
\end{lem}

\subsection{Convolution and monodromy}

\begin{lem}
\label{lem:monodromy-convolution}
For any $\mathscr{F}$, $\mathscr{G}$ in $\Db_{U}(X \quot T, \bk)$, the object $\mathscr{F} \star^U \mathscr{G}$ belongs to the subcategory $\Db_{U}(X \quot T, \bk)$. Moreover, for any $x \in R_T^\wedge$ we have
 \begin{gather*}
 \varphi^{\wedge}_{\lef,\mathscr{F} \star^U \mathscr{G}}(x) = \varphi^{\wedge}_{\lef,\mathscr{F}}(x) \star^U \id_{\mathscr{G}}, \\
 \varphi^{\wedge}_{\rig,\mathscr{F} \star^U \mathscr{G}}(x) = \id_{\mathscr{F}} \star^U \varphi^{\wedge}_{\rig,\mathscr{G}}(x), \\
  \varphi^{\wedge}_{\rig,\mathscr{F}} (x) \star^U \id_{\mathscr{G}} = \id_{\mathscr{F}} \star^U \varphi^{\wedge}_{\lef,\mathscr{G}}(x).
 \end{gather*}
\end{lem}

\begin{proof}
The first claim is clear from~\eqref{eqn:conv-pi^*}.
The proof of the first two isomorphisms is easy, and left to the reader. To prove the third one, we write the map $m$ as a composition $m=m_1 \circ m_2$ where $m_1 : G \times^B X \to X$ and $m_2 : G \times^U X \to G \times^B X$ are the obvious morphisms. Then we have
\[
\mathscr{F} \star^U \mathscr{G} = m_! (\mathscr{F} \wbtimes \mathscr{G}) [\dim T] = (m_1)_! (m_2)_! (\mathscr{F} \wbtimes \mathscr{G}) [\dim T].
\]
We consider the action of $T$ on $G \times^U X$ defined by $t \cdot [g : hU] = [gt^{-1} : thU]$. Then $\mathscr{F} \wbtimes \mathscr{G}$ belongs to $\Db_c \bigl( (G \times^U X) \quot T, \bk \bigr)$ for this action, and the corresponding monodromy morphism satisfies
\[
\varphi^\wedge_{\mathscr{F} \wbtimes \mathscr{G}}(\lambda) = \varphi^\wedge_{\mathscr{F}}(\lambda^{-1}) \wbtimes \varphi^\wedge_{\mathscr{G}}(\lambda)
\]
for any $\lambda \in X_*(T)$. (In fact this equality can be checked after pullback to $G \times G/U$, where it follows from Lemma~\ref{lem:monodromy-isom}.) Now $m_2$ is the quotient map for this $T$-action; hence Lemma~\ref{lem:monodromy-augmentation} implies that
\[
(m_2)_! \varphi^\wedge_{\mathscr{F} \wbtimes \mathscr{G}}(\lambda) = \id,
\]
or in other words that
\[
(m_2)_! \bigl( \varphi^\wedge_{\mathscr{F}}(\lambda) \wbtimes \id_{\mathscr{G}} \bigr) = (m_2)_! \bigl( \id_{\mathscr{F}} \wbtimes \varphi^\wedge_{\mathscr{G}}(\lambda) \bigr).
\]
Applying $(m_1)_!$ we deduce the desired equality.
\end{proof}

\subsection{Extension to the completed category}

We now explain how to extend the construction of the convolution product to the framework of the completed category $\hD_{U}(X \quot T, \bk)$.

\begin{lem}
\label{lem:convolution-representable}
Let $``\varprojlim_n" \mathscr{F}_n$ be an object of $\hD_{U}(X \quot T, \bk)$. If $\mathscr{G}$ is in $\Db_U(X \quot T,\bk)$, resp.~if $\mathscr{G}'$ is in $\Db_{U}(Y,\bk)$, then the pro-object
\[
``\varprojlim_n" \mathscr{F}_n \star^U \mathscr{G}, \quad \text{resp.} \quad ``\varprojlim_n" \mathscr{F}_n \star^U \mathscr{G}',
\]
is representable by an object of $\Db_{U}(X \quot T,\bk)$, resp.~of $\Db_{U}(Y,\bk)$.
\end{lem}

\begin{proof}[Sketch of proof]
This property is proved along the lines of~\cite[\S 4.3]{by}; we sketch the proof in the second case, and leave the details and the first case to the reader. If $\mathscr{G}'$ is of the form $\For^B_U(\mathscr{G}'')$ for some $\mathscr{G}''$ in $\Db_B(Y,\bk)$, then by Lemma~\ref{lem:convolution-Bequiv} we have $\mathscr{F}_n \star^U \mathscr{G}' \cong \pi_\dag(\mathscr{F}_n) \star^B \mathscr{G}''$. Hence the claim follows from the assumption that the pro-object $``\varprojlim_n" \pi_\dag(\mathscr{F}_n)$ is representable. The general case follows since the objects of this form generate $\Db_{U}(Y,\bk)$ as a triangulated category, using the following observation (which can be checked using the methods of~\cite[Appendix~A]{by}): given a projective system of distinguished triangles $\mathscr{A}_n \to \mathscr{B}_n \to \mathscr{C}_n \xrightarrow{[1]}$ in $\Db_{U}(Y,\bk)$, if the pro-objects $``\varprojlim_n" \mathscr{A}_n$ and $``\varprojlim_n" \mathscr{B}_n$ are representable, then $``\varprojlim_n" \mathscr{C}_n$ is representable too (and this object is a cone of the induced morphism $``\varprojlim_n" \mathscr{A}_n \to ``\varprojlim_n" \mathscr{B}_n$).
\end{proof}

Using Lemma~\ref{lem:convolution-representable}, we already see that the functor $\star^U : \Db_U(X,\bk) \times \Db_U(Y,\bk) \to \Db_U(Y,\bk)$ induces a functor
\begin{equation}
\label{eqn:hatstar-Y}
\hatstar : \hD_{U}(X \quot T, \bk) \times \Db_{U}(Y,\bk) \to \Db_{U}(Y,\bk).
\end{equation}

Now, let $\mathscr{F} = ``\varprojlim_n" \mathscr{F}_n$ and $\mathscr{G} = ``\varprojlim_m" \mathscr{G}_m$ be two objects of $\hD_{U}(X \quot T, \bk)$. For any fixed $m$, by Lemma~\ref{lem:convolution-representable} the pro-object $``\varprojlim_n" \mathscr{F}_n \star^U \mathscr{G}_m$ is representable by an object of $\Db_{U}(X \quot T,\bk)$. Therefore, we can consider the pro-object
\[
\mathscr{F} \hatstar \mathscr{G} := ``\varprojlim_m" ``\varprojlim_n" \mathscr{F}_n \star^U \mathscr{G}_m.
\]
We claim that this pro-object belongs to $\hD_{U}(X \quot T, \bk)$. Indeed it is clearly uniformly bounded. And using~\eqref{eqn:conv-pi_!} we see that
\begin{multline*}
``\varprojlim_m" \pi_\dag \left( ``\varprojlim_n" \mathscr{F}_n \star^U \mathscr{G}_m \right) \cong ``\varprojlim_m" \left( ``\varprojlim_n" \mathscr{F}_n \star^U \pi_\dag (\mathscr{G}_m) \right) \\
\cong ``\varprojlim_m" \left( \mathscr{F} \hatstar \pi_\dag (\mathscr{G}_m) \right) \cong \mathscr{F} \hatstar \left( ``\varprojlim_m" \pi_\dag (\mathscr{G}_m) \right).
\end{multline*}
Since by assumption the pro-object $``\varprojlim_m" \pi_\dag (\mathscr{G}_m)$ is representable, this shows that $\mathscr{F} \hatstar \mathscr{G}$ is $\pi$-constant, which finishes the proof of our claim.

\begin{rmk}
\label{rmk:convolution-inverse}
Let $\mathscr{F}$ and $\mathscr{G}$ be as above. Using similar arguments one can check that, for any fixed $n \geq 0$, the pro-object $``\varprojlim_m" \mathscr{F}_n \star^U \mathscr{G}_m$ is representable, so that it makes sense to consider the pro-object
\[
``\varprojlim_n" ``\varprojlim_m" \mathscr{F}_n \star^U \mathscr{G}_m.
\]
Using standard results on inverse limits (see e.g.~\cite[Proposition~2.1.7]{ks}) one can show that this pro-object is canonically isomorphic to $\mathscr{F} \hatstar \mathscr{G}$.
\end{rmk}

This construction provides us with a functor 
\[
\hatstar : \hD_{U}(X \quot T, \bk) \times \hD_{U}(X \quot T, \bk) \to \hD_{U}(X \quot T, \bk).
\]
This functor is associative in the obvious sense, and compatible with~\eqref{eqn:hatstar-Y} in the sense that for $\mathscr{F}, \mathscr{G}$ in $\hD_{U}(X \quot T, \bk)$ and $\mathscr{H}$ in $\Db_{U}(Y,\bk)$ we have canonical isomorphisms
\begin{gather}
(\mathscr{F} \hatstar \mathscr{G}) \hatstar \mathscr{H} \cong \mathscr{F} \hatstar (\mathscr{G} \hatstar \mathscr{H}), \\
\label{eqn:hatstar-pidag}
\pi_\dag(\mathscr{F} \hatstar \mathscr{G}) \cong \mathscr{F} \hatstar \pi_\dag(\mathscr{G}).
\end{gather}

The object $\hDel_e=\hnab_e$ is a unit for this product (at least in the case when $\mathrm{char}(\bk)>0),$\footnote{This assumption is probably unnecessary. But since this is the setting we are mostly interested in, we will not consider the possible extension of this claim to the characteristic-$0$ setting.} as proved in the following lemma.

\begin{lem}
\label{lem:unit-convolution}
Assume that $\mathrm{char}(\bk)>0$. Then
for any $\mathscr{F}$ in $\hD_{U}(X \quot T, \bk)$, there exist canonical isomorphisms
\[
\hDel_e \hatstar \mathscr{F} \cong \mathscr{F} \cong \mathscr{F} \hatstar \hDel_e. 
\]
\end{lem}

\begin{proof}
For any $\mathscr{G}$ in $\Db_{U}(X \quot T, \bk)$, we have
\[
\hDel_e \hatstar \mathscr{G} \cong ``\varprojlim_n" a_!(\mathscr{L}_{T,n} \boxtimes \mathscr{G}) [2r],
\]
where $a : T \times X \to X$ is the action morphism defined by $a(t,gU)=tgU$. Now we have canonical identifications
\[
\Db_{U}(G/U,\bk) \cong \Db_{U \times U}(G,\bk) \cong \Db_U(U \backslash G,\bk).
\]
Under these identifications the full subcategory $\Db_{U}(X \quot T, \bk) \subset \Db_{U}(G/U,\bk)$ coincides with the category $\Db_{U}(U \backslash G \quot T, \bk)$ defined relative to the $T$-action on $U \backslash G$ defined by $t \cdot Ug=Utg$ and the stratification of $B \backslash G$ by $B$-orbits. Hence Lemma~\ref{lem:averaging-hL} provides a canonical isomorphism $\hDel_e \hatstar \mathscr{G} \cong \mathscr{G}$. Passing to (formal) projective limits we deduce a similar isomorphism for any $\mathscr{G}$ in $\hD_{U}(X \quot T, \bk)$.

The proof of the isomorphism $\mathscr{F} \cong \mathscr{F} \hatstar \hDel_e$ follows from similar considerations together with Remark~\ref{rmk:convolution-inverse}.
\end{proof}

One can easily check that these constructions provide $\Db_{U}(X \quot T, \bk)$ with the structure of a monoidal category (in the case when $\mathrm{char}(\bk)>0$).

\subsection{Convolution of standard, costandard, and tilting objects}

\begin{lem}\phantomsection
\label{lem:convolution-hDel}
\begin{enumerate}
\item
\label{it:conv-hDel-hnab}
For any $w \in W$ we have $\hnab_{w^{-1}} \hatstar \hDel_w \cong \hDel_e$.
\item
\label{it:conv-hDel}
If $v,w \in W$ and if $\ell(vw)=\ell(v)+\ell(w)$, then we have
\[
\hDel_v \hatstar \hDel_w \cong \hDel_{vw}, \qquad \hnab_v \hatstar \hnab_w \cong \hnab_{vw}.
\]
\end{enumerate}
\end{lem}

\begin{proof}
We prove the first isomorphism in~\eqref{it:conv-hDel}; the other claims can be obtained similarly.
By~\eqref{eqn:hatstar-pidag} and~\eqref{eqn:Pidag-hD-hN} we have
\[
\pi_\dag(\hDel_v \hatstar \hDel_w) \cong \hDel_v \hatstar \pi_\dag(\hDel_w) \cong \hDel_v \hatstar \Delta_w.
\]
Since $\Delta_w$ is a $B$-equivariant perverse sheaf, using Lemma~\ref{lem:convolution-Bequiv} we deduce that
\[
\pi_\dag(\hDel_v \hatstar \hDel_w) \cong \Delta_v \star^B \Delta_w.
\]
Now it is well known that the right-hand side is isomorphic to $\Delta_{vw}$, see e.g.~\cite[\S 2.2]{bbm} or~\cite[Proposition~4.4]{modrap2}. Then the claim follows from Remark~\ref{rmk:filtrations-pidag}.
\end{proof}

\begin{lem}
\label{lem:conv-hTil}
Let $s \in S$. For any tilting perverse sheaf $\hTil$ in $\hD_{U}(X \quot T,\bk)$, the object $\hTil_s \hatstar \hTil$ is a tilting perverse sheaf, and for any $w \in W$ we have
\[
(\hTil_s \hatstar \hTil : \hDel_w) = (\hTil : \hDel_w) + (\hTil : \hDel_{sw}).
\]
\end{lem}

\begin{proof}
We will prove that for any $w \in W$ the object $\hTil_s \hatstar \hDel_w$ admits a standard filtration, the multiplicity of $\hDel_v$ being $1$ if $v \in \{w,sw\}$, and $0$ otherwise. Similar arguments show that $\hTil_s \hatstar \hnab_w$ admits a costandard filtration, and the desired claim will follow. First, assume that $sw>w$. Then using the exact sequence $\hDel_s \hookrightarrow \hTil_s \twoheadrightarrow \hDel_e$ (see~\S\ref{ss:Ts}) and applying $(-) \hatstar \hDel_w$ we obtain a distinguished triangle
\[
\hDel_s \hatstar \hDel_w \to \hTil_s \hatstar \hDel_w \to \hDel_e \hatstar \hDel_w \xrightarrow{[1]}.
\]
Here Lemma~\ref{lem:convolution-hDel}\eqref{it:conv-hDel} implies that the first term is isomorphic to $\hDel_{sw}$, and that the third term is isomorphic to $\hDel_w$, which shows the desired property. If now $sw<w$, we use the exact sequence $\hDel_e \hookrightarrow \hTil_s \twoheadrightarrow \hnab_s$ to obtain a distinguished triangle
\[
\hDel_e \hatstar \hDel_w \to \hTil_s \hatstar \hDel_w \to \hnab_s \hatstar \hDel_w \xrightarrow{[1]}.
\]
We conclude as above, using also Lemma~\ref{lem:convolution-hDel}\eqref{it:conv-hDel-hnab} to see that the third term is isomorphic to $\hDel_{sw}$.
\end{proof}

\begin{rmk}
\label{rmk:convolution-tilting}
One can easily deduce from Lemma~\ref{lem:conv-hTil} that the tilting objects in $\hscO$ are the direct sums of direct summands of objects of the form $\hTil_{s_1} \hatstar \cdots \hatstar \hTil_{s_r}$ with $s_1, \cdots, s_r \in S$, and moreover that the convolution product of two tilting objects is again a tilting object. Similarly, the tilting objects in $\scO$ are the direct sums of direct summands of objects of the form $\hTil_{s_1} \hatstar \cdots \hatstar \hTil_{s_r} \hatstar \Delta_e$ with $s_1, \cdots, s_r \in S$, and $\hTil \hatstar \mathscr{T}$ is tilting in $\scO$ if $\hTil$ is tilting in $\hscO$ and $\mathscr{T}$ is tilting in $\scO$. In particular, this provides a ``Bott--Samelson type'' construction of these tilting objects.
\end{rmk}

\begin{prop}
\label{prop:convolution-hTw0}
 For any $v,w \in W$, we have
 \[
  \hDel_v \hatstar \hTil_{w_0} \hatstar \hDel_w \cong \hTil_{w_0}.
 \]
\end{prop}

\begin{proof}
Of course, it is enough to prove that for $v,w \in W$ we have
\[
\hDel_v \hatstar \hTil_{w_0} \cong \hTil_{w_0} \quad \text{and} \quad \hTil_{w_0} \hatstar \hDel_w \cong \hTil_{w_0}.
\]
And for this,
in view of Proposition~\ref{prop:classification-tiltings} it suffices to prove that
\begin{equation}
\label{eqn:Tw0-convolution}
\pi_\dag \bigl( \hDel_v \hatstar \hTil_{w_0} \bigr) \cong \mathscr{T}_{w_0} \quad \text{and} \quad
\pi_\dag \bigl( \hTil_{w_0} \hatstar \hDel_w \bigr) \cong \mathscr{T}_{w_0}.
\end{equation}

We first prove the second isomorphism in~\eqref{eqn:Tw0-convolution}.
By~\eqref{eqn:hatstar-pidag} and~\eqref{eqn:Pidag-hD-hN} we have
\[
\pi_\dag \bigl( \hTil_{w_0} \hatstar \hDel_w \bigr) \cong \hTil_{w_0} \hatstar \pi_\dag(\hDel_w) \cong \hTil_{w_0} \hatstar \Delta_w.
\]
Since $\Delta_w$ is a $B$-equivariant perverse sheaf, using Lemma~\ref{lem:convolution-Bequiv} we deduce that
\[
\hTil_{w_0} \hatstar \Delta_w \cong \mathscr{T}_{w_0} \star^B \Delta_w.
\]
Hence to prove the second isomorphism in~\eqref{eqn:Tw0-convolution} we only have to prove that
\begin{equation}
\label{eqn:Tw0-convolution-2}
\mathscr{T}_{w_0} \star^B \Delta_w \cong \mathscr{T}_{w_0}.
\end{equation}
It is known that any object of the form $\nabla_u \star^B \Delta_v$ is perverse (see e.g.~\cite[Proposition~4.6]{modrap2} or~\cite[Proposition~8.2.4]{abg:qglg} for similar claims). In particular, it follows that $\mathscr{T}_{w_0} \star^B \Delta_w$ is perverse. And since $\Delta_w \star^B \nabla_{w^{-1}} \cong \Delta_e$, for any $x \in W$ and $n \in \Z$ we have
\[
\Hom_{\Db_{U}(Y,\bk)}(\mathscr{T}_{w_0} \star^B \Delta_w, \IC_x[n]) \cong \Hom_{\Db_{U}(Y,\bk)}(\mathscr{T}_{w_0}, \IC_x \star^B \nabla_{w^{-1}}[n]).
\]
Now since the realization functor $\Db \scO \to \Db_{U}(Y,\bk)$ is an equivalence of categories (see e.g.~\cite[Corollary~3.3.2]{bgs}, whose proof applies to any field of coefficients), and in view of Lemma~\ref{lem:properties-Tw0}\eqref{it:Tw0-2}, for $y \in W$ and $m \in \Z$ we have
\[
\Hom_{\Db_{U}(Y,\bk)}(\mathscr{T}_{w_0}, \IC_y [m]) \cong
\begin{cases}
\bk & \text{if $y=e$ and $m=0$;} \\
0 & \text{otherwise.}
\end{cases}
\]
It is not difficult to see that if $x \neq e$ and if $\mathscr{G}$ belongs to $\Db_B(Y,\bk)$ then all the composition factors of the perverse cohomology objects of $\IC_x \star^B \mathscr{G}$ are of the form $\IC_y$ with $y \neq e$; using also Lemma~\ref{lem:soc-top-Delta-nabla} we deduce that
\[
\Hom_{\Db_{U}(Y,\bk)}(\mathscr{T}_{w_0}, \IC_x \star^B \nabla_{w^{-1}}[n]) \cong \begin{cases}
\bk & \text{if $x=e$ and $n=0$;} \\
0 & \text{otherwise.}
\end{cases}
\]
It follows that the perverse sheaf $\mathscr{T}_{w_0} \star^B \Delta_w$ is the projective cover of $\IC_e$, hence that it is isomorphic to $\mathscr{T}_{w_0}$ by Lemma~\ref{lem:properties-Tw0}\eqref{it:Tw0-2}. This finally proves~\eqref{eqn:Tw0-convolution-2}, hence also the second isomorphism in~\eqref{eqn:Tw0-convolution}.

We now consider the first isomorphism in~\eqref{eqn:Tw0-convolution}.
If $v=e$ then it follows from Lemma~\ref{lem:convolution-hDel} that $\hDel_e \hatstar \hTil_{w_0}$ is a tilting perverse sheaf, and has the same standard multiplicities as $\hTil_{w_0}$; therefore it is isomorphic to $\hTil_{w_0}$. 
Now assume the claim is known for $v \neq w_0$, and choose $s \in S$ such that $vs>v$. By the same arguments as in the proof of Lemma~\ref{lem:conv-hTil} we have an exact sequence of perverse sheaves
\begin{equation}
\label{eqn:ses}
\hDel_{vs} \hookrightarrow \hDel_v \hatstar \hTil_s \twoheadrightarrow \hDel_v.
\end{equation}
From Lemma~\ref{lem:conv-hTil} we deduce that $\hTil_s \hatstar \mathscr{T}_{w_0} \cong (\mathscr{T}_{w_0})^{\oplus 2}$. Therefore,
convolving~\eqref{eqn:ses} with $\mathscr{T}_{w_0}$ on the right and using induction we obtain a distinguished triangle
\begin{equation*}
\hDel_{vs} \hatstar \mathscr{T}_{w_0} \to (\mathscr{T}_{w_0})^{\oplus 2} \to \mathscr{T}_{w_0} \xrightarrow{[1]}
\end{equation*}
in $\Db_{U}(Y,\bk)$.
As above the object $\hDel_{vs} \hatstar \mathscr{T}_{w_0}$ is perverse; hence this triangle is a short exact sequence in $\scO$. Since $\mathscr{T}_{w_0}$ is projective (see Lemma~\ref{lem:properties-Tw0}\eqref{it:Tw0-2}) the surjection $(\mathscr{T}_{w_0})^{\oplus 2} \twoheadrightarrow \mathscr{T}_{w_0}$ must be split, and we finally obtain that $\hDel_{vs} \hatstar \mathscr{T}_{w_0} \cong \mathscr{T}_{w_0}$, as desired.
\end{proof}

\section{Variations on some results of Kostant--Kumar}
\label{sec:kostant-kumar}

From now on we assume that $G$ is semisimple, of adjoint type. (Of course this assumption is harmless if one is mainly interested in the category $\mathscr{O}$.) We will denote by $\Phi^\vee$ the coroot system of $(G,T)$, and by $\Phi^\vee_+ \subset \Phi^\vee$ the positive coroots.

\subsection{The Pittie--Steinberg theorem}

We set
\[
 \mathsf{d} = \prod_{\alpha^\vee \in \Phi^\vee_+} (1-e^{\alpha^\vee}) \quad \in R_T,
\]
and denote by $\rho^\vee \in X_*(T)$ the halfsum of the positive coroots.

The following result is an easy application of the Pittie--Steinberg theorem.

\begin{thm}
\label{thm:pittie-steinberg}
 The $(R_T^\wedge)^W$-module $R_T^\wedge$ is free of rank $\# W$. 
 More precisely, this module admits a basis $(\mathsf{e}_w)_{w \in W}$ such that
 \begin{equation}
 \label{eqn:det-Steinberg}
  \det \bigl( (w(\mathsf{e}_v))_{v,w \in W} \bigr) = \bigl( (-1)^{|\Phi^\vee_+|} e^{-\rho^\vee} \mathsf{d} \bigr)^{|W|/2}.
 \end{equation}
\end{thm}

\begin{proof}
By the Pittie--Steinberg theorem (see~\cite{steinberg}) we know that under our assumptions $\Z[X_*(T)]$ is free over $\Z[X_*(T)]^W$, of rank $\# W$. Moreover, from the proof in~\cite{steinberg} one sees that this module admits a basis such that~\eqref{eqn:det-Steinberg} holds (see e.g.~\cite[Proof of Theorem~4.4]{kk}). Now there are canonical isomorphisms
\[
 \bk \otimes_{\Z} \Z[X_*(T)] \simto \bk[X_*(T)], \qquad \bk \otimes_{\Z} \Z[X_*(T)]^W \simto \bk[X_*(T)]^W.
\]
(For the second one, we remark that $\Z[X_*(T)]^W$ is a free $\Z$-module, with a basis consisting of the elements $\sum_{\lambda \in \mathbb{O}} e^\lambda$ where $\mathbb{O}$ runs over $W$-orbits in $X_*(T)$. Since a similar fact holds for $\bk[X_*(T)]^W$, we deduce that the natural morphism $\bk \otimes_{\Z} \Z[X_*(T)]^W \to \bk[X_*(T)]^W$ is indeed an isomorphism.) Hence $R_T$ is free over $(R_T)^W$, of rank $\# W$, and admits a basis $(\mathsf{e}_w)_{w \in W}$ such that~\eqref{eqn:det-Steinberg} holds.

Now we consider completions. Let $a \in R_T^\wedge$, and write $a$ as the limit of a sequence $(a_n)_{n \geq 0}$ of elements of $R_T$. For any $n \geq 0$, there exist (unique) elements $(p^n_w)_{w \in W}$ in $(R_T)^W$ such that
\begin{equation}
\label{eqn:relation-a-p}
 a_n = \sum_{w \in W} p_w^n \cdot \mathsf{e}_w.
\end{equation}
We claim that each sequence $(p_w^n)_{n \geq 0}$ converges to a certain $p_w \in R_T^\wedge$; then $p_w$ will belong to $(R_T^\wedge)^W$, and we will have $a=\sum_{w \in W} p_w \cdot \mathsf{e}_w$, which will prove that the elements $(\mathsf{e}_w)_{w \in W}$ generate $R_T^\wedge$ over $(R_T^\wedge)^W$.

Consider the matrix $M:=(v(\mathsf{e}_w))_{v,w \in W}$, with rows and columns parametrized by $W$, and coefficients in $R_T$. Then the equalities~\eqref{eqn:relation-a-p} imply that for any $n \geq 0$ we have
\[
 (v (a_n))_{v \in W} = M \cdot (p^n_w)_{w \in W}
\]
in the space of vectors parametrized by $W$, and with values in the ring $R_T$. Now~\eqref{eqn:det-Steinberg} shows that $M$ is invertible in the space of matrices with coefficients in the fraction field of $R_T$, and that $\mathsf{d}^{|W|/2} \cdot M^{-1}$ in fact has coefficients in $R_T$. Moreover, we have
\begin{equation}
\label{eqn:relation-a-p-2}
 (\mathsf{d}^{|W|/2} \cdot p^n_w)_{w \in W} = (\mathsf{d}^{|W|/2} \cdot M^{-1}) \cdot (v (a_n))_{v \in W}.
\end{equation}
From this, we will deduce that each sequence $(p_w^n)_{n \geq 0}$ is Cauchy, which will prove our claim. In fact,
by the Artin--Rees lemma (applied to the $R_T$-modules $\mathsf{d}^{|W|/2} \cdot R_T \subset R_T$ and the ideal $\mathfrak{m}_T$), there exists an integer $c$ such that
\[
 \mathfrak{m}_T^n \cap \mathsf{d}^{|W|/2} \cdot R_T \subset \mathsf{d}^{|W|/2} \cdot \mathfrak{m}_T^{n-c}
\]
for any $n \geq c$. Now if $k \geq 0$ is fixed, for $n,m \gg 0$ we have $a_n - a_m \in \mathfrak{m}_T^{c+k}$. From~\eqref{eqn:relation-a-p-2} we deduce that $\mathsf{d}^{|W|/2} \cdot (p_w^n - p_w^m)$ belongs to $\mathfrak{m}_T^{c+k}$ also, hence to $\mathsf{d}^{|W|/2} \cdot \mathfrak{m}_T^{k}$. Hence $p_w^n - p_w^m$ belongs to $\mathfrak{m}_T^k$, which finishes the proof of the claim.

To conclude the proof, it remains to check that the elements $(\mathsf{e}_w)_{w \in W}$ are linearly independent over $(R_T^\wedge)^W$. However, if
\[
 \sum_{w \in W} p_w \cdot \mathsf{e}_w = 0
\]
for some elements $p_w$ in $(R_T^\wedge)^W$, then as above we have $M \cdot (p_w)_{w \in W} = 0$. Since $M$ is invertible (as a matrix with coefficients in the fraction field of $R_T^\wedge$), it follows that $p_w=0$ for any $w \in W$.
\end{proof}

Let us note the following consequences of this theorem:
 \begin{itemize}
  \item the $R_T^\wedge$-module $R_T^\wedge \otimes_{(R_T^\wedge)^W} R_T^\wedge$ is free of rank $\# W$;
  \item the $\bk$-vector space $R_T^\wedge / (R_T^\wedge)^W_+$ has dimension $\# W$, where $(R_T^\wedge)^W_+$ is the kernel of the map $(R_T^\wedge)^W \hookrightarrow R_T^\wedge \xrightarrow{\varepsilon_T^\wedge} \bk$.
 \end{itemize}
 
\subsection{Some \texorpdfstring{$R_T^\wedge$}{RT}-modules}
\label{ss:some-RT-modules}

In this subsection we recall some constructions due to Kostant--Kumar~\cite{kk} (replacing everywhere the $T^\vee$-equivariant $\mathsf{K}$-theory of the point---where $T^\vee$ is the torus dual to $T$---by $R_T^\wedge$).

We will denote by $Q_T^\wedge$ the fraction field of $R_T^\wedge$. We then denote by $Q_W$ the smash product of $Q_T^\wedge$ and $W$; in other words $Q_W$ is a $Q_T^\wedge$-vector space with a basis $(\delta_w)_{w \in W}$, with the multiplication determined by
\[
 (a \delta_w) \cdot (b \delta_v) = aw(b) \delta_{wv}. 
\]
Of course, $(\delta_w)_{w \in W}$ is also a basis for the action of $Q_T^\wedge$ given by right multiplication in $Q_W$. We will denote by $\iota$ the anti-involution of $Q_W$ determined by
\[
 \iota(a)=a, \qquad \iota(\delta_w)=\delta_{w^{-1}}
\]
for $a \in Q_T^\wedge$ and $w \in W$.

Following~\cite{kk}, for $s \in S$ we set
\[
y_s := (\delta_e + \delta_s) \frac{1}{1-e^{-\alpha_s^\vee}} = \frac{1}{1-e^{-\alpha_s^\vee}} (\delta_e - e^{-\alpha_s^\vee} \delta_s),
\]
where $\alpha_s^\vee$ is the simple coroot associated with $s$.
The same computation as for~\cite[Proposition~2.4]{kk} shows that these elements satisfy the braid relations of $W$; therefore, by Matsumoto's lemma, for $w \in W$ we can set
\[
 y_w := y_{s_1} \cdot y_{s_2} \cdot (\cdots) \cdot y_{s_r},
\]
where $w=s_1 \cdots s_r$ is any reduced expression. It is clear from definitions that the matrix expressing these elements in the basis $(\delta_w)_{w \in W}$ is upper triangular with respect to the Bruhat order; in particular $(y_w)_{w \in W}$ is also a $Q_T^\wedge$-basis of $Q_W$. We set
\[
 Y_W := \bigoplus_{w \in W} R_T^\wedge \cdot y_w,
\]
a free $R_T^\wedge$-module of rank $\# W$.
As in~\cite[Corollary~2.5]{kk}, one sees that $Y_W$ is a subring in $Q_W$, and that $(y_w)_{w \in W}$ is also a basis of $Y_W$ as an $R_T^\wedge$-module for the action induced by right multiplication.

We now consider
\[
 \Omega_W := \Hom_{Q_T^\wedge}(Q_W, Q_T^\wedge),
\]
where $Q_W$ is regarded as a $Q_T^\wedge$-vector space for the action by right multiplication. We will regard $\Omega_W$ as a $Q_T^\wedge$-vector space via $(a \cdot \psi)(b) = a \psi(b) = \psi(ba)$ for $a \in Q^\wedge_T$ and $b \in Q_W$. We will sometimes identify this vector space with the vector space $\Fun(W,Q_T^\wedge)$ of functions from $W$ to $Q_T^\wedge$, by sending the map $\psi$ to the function $w \mapsto \psi(\delta_w)$.

The space $\Omega_W$ admits an action of $Q_W$ (by $Q_T^\wedge$-vector space automorphisms) defined by
\[
 (y \cdot \psi)(z) = \psi(\iota(y) \cdot z)
\]
for $y,z \in Q_W$ and $\psi \in \Omega_W$. (Note that the action of $Q_T^\wedge \cdot \delta_e \subset Q_W$ does \emph{not} coincide with the action of $Q_T^\wedge$ considered above.) Explicitly, we have
\begin{equation}
\label{eqn:action-y-psi}
 (y_s \cdot \psi)(\delta_w) = \frac{\psi(\delta_w) - e^{-w^{-1} \alpha_s^\vee} \psi(\delta_{sw})}{1 - e^{-w^{-1} \alpha_s^\vee}}.
\end{equation}

We will be interested in the subspace
\[
 \Psi_W := \{\psi \in \Omega_W \mid \forall y \in Y_W, \ \psi(\iota(y)) \in R_T^{\wedge}\}.
\]
Of course, this subspace is stable under the action of $R_T^\wedge \subset Q_T^\wedge$. Since $Y_W$ is a subalgebra in $Q_W$, $\Psi_W$ is also stable under the action of $Y_W \subset Q_W$. Since $(\iota(y_w))_{w \in W}$ is a basis of $\iota(Y_W)$ as a right $R_T^\wedge$-module, $\Psi_W$ is free as an $R_T^\wedge$-module, with a basis $(\psi_w)_{w \in W}$ determined by
\[
 \psi_w(\iota(y_v)) = \begin{cases}
                       1 & \text{if $v=w$;} \\
                       0 & \text{otherwise.}
                      \end{cases}
\]

The following properties can be checked as in~\cite[Proposition~2.22]{kk}.

\begin{lem}\phantomsection
\label{lem:properties-psi}
\begin{enumerate}
 \item
 \label{it:psi-delta}
 For any $v,w \in W$, the element $\psi_v(\delta_w)$ belongs to $R_T^\wedge$, and vanishes unless $v \leq w$.
 \item
 \label{it:psi-delta-2}
 For any $w \in W$ we have
 \[
  \psi_w(\delta_w) = \prod_{\substack{\alpha^\vee \in \Phi_+^\vee \\ w(\alpha^\vee) \in -\Phi^\vee_+}} (1-e^{\alpha^\vee}).
 \]
 \item
 \label{it:y-psi}
 For $w \in W$ and $s \in S$, we have
 \[
  y_s \cdot \psi_w = \begin{cases}
                      \psi_w + \psi_{sw} & \text{if $sw<w$;}\\
                      0 & \text{otherwise.}
                     \end{cases}
 \]
\end{enumerate}
\end{lem}

In particular, Point~\eqref{it:psi-delta} in this lemma shows that under the identification of $\Omega_W$ with $\Fun(W,Q_T^\wedge)$ considered above, $\Psi_W$ is contained in the subset $\Fun(W,R_T^\wedge)$ of functions taking values in $R_T^\wedge$. 

\subsection{An isomorphism of \texorpdfstring{$R_T^\wedge$}{RT}-modules}

Our goal in this subsection is to relate the algebra $R_T^\wedge \otimes_{(R_T^\wedge)^W} R_T^\wedge$ with the objects introduced in~\S\ref{ss:some-RT-modules}. Our proofs are based on ``$\mathsf{K}$-theoretic analogues'' of some arguments from~\cite[Appendix~D]{ajs}.

Below we will need the following lemma.

\begin{lem}
\label{lem:divisibility}
Let $f \in R_T^\wedge$, and let $\alpha^\vee, \beta^\vee$ be distinct positive coroots. If $(1-e^{\alpha^\vee}) \cdot f$ is divisible (in $R_T^\wedge$) by $1-e^{\beta^\vee}$, then $f$ is divisible by $1-e^{\beta^\vee}$.
\end{lem}

\begin{proof}
Let us first prove the similar claim where $R_T^\wedge$ is replaced by $R_T$ everywhere.
For this, we denote by $T_\bk^\vee$ the torus dual to $T$, and consider $\alpha^\vee$ and $\beta^\vee$ as characters of $T_\bk^\vee$. Since $\alpha^\vee$ and $\beta^\vee$ are linearly independent in the $\Z$-module $X_*(T)$, the group morphism
\[
 (\alpha^\vee, \beta^\vee) : T_\bk^\vee \to (\Bbbk^\times)^2
\]
is dominant, hence surjective. It follows that $\dim(\ker(\alpha^\vee) \cap \ker(\beta^\vee)) = \dim(T^\vee)-2$. 
Now $R_T$ is a UFD, and the dimension condition means that $1-e^{\alpha^\vee}$ and $1-e^{\beta^\vee}$ have no common prime factor. If $f \in R_T$ and $(1-e^{\alpha^\vee}) \cdot f$ is divisible by
$1-e^{\beta^\vee}$,
each prime factor in the decomposition of $1-e^{\beta^\vee}$ must appear in $f$, with at least the same multiplicity. It follows that $1-e^{\beta^\vee}$ divides $f$, as desired.

The claim we have just proved can be translated into the fact that
the ``Koszul complex''
\[
 0 \to R_T \xrightarrow{f \mapsto ((1-e^{\beta^\vee})f, (1-e^{\alpha^\vee})f)} R_T \oplus R_T \xrightarrow{(g,h) \mapsto (1-e^{\alpha^\vee})g - (1-e^{\beta^\vee})h} R_T \to 0
\]
(with nonzero terms in degrees $-2$, $-1$ and $0$)
has no cohomology in degree $-1$. Since $R_T^\wedge$ is flat over $R_T$, applying the functor $R^\wedge_T$ we deduce that the complex
\[
 0 \to R_T^\wedge \xrightarrow{f \mapsto ((1-e^{\beta^\vee})f, (1-e^{\alpha^\vee})f)} R_T^\wedge \oplus R_T^\wedge \xrightarrow{(g,h) \mapsto (1-e^{\alpha^\vee})g - (1-e^{\beta^\vee})h} R_T^\wedge \to 0
\]
has no cohomology in degree $-1$ either, which implies our lemma.
\end{proof}

\begin{thm}
\label{thm:kk-ajs}
 The morphism
 \[
  \tau : R_T^\wedge \otimes_{(R_T^\wedge)^W} R_T^\wedge \to \Fun(W,R_T^\wedge)
 \]
sending $a \otimes b$ to the function $w \mapsto a \cdot w^{-1}(b)$ is injective. Its image consists of the functions $f$ such that
\[
f(w) \equiv f(ws_{\alpha^\vee}) \mod (1-e^{\alpha^\vee})
\]
for any $w \in W$ and any coroot $\alpha^\vee$.
\end{thm}

\begin{proof}
 Consider the basis $(\mathsf{e}_w)_{w \in W}$ of $R_T^\wedge$ as an $(R_T^\wedge)^W$-module considered in Theorem~\ref{thm:pittie-steinberg}. Then $(1 \otimes \mathsf{e}_w)_{w \in W}$ is a basis of $R_T^\wedge \otimes_{(R_T^\wedge)^W} R_T^\wedge$ as an $R_T^\wedge$-module. Moreover, $\tau(1 \otimes \mathsf{e}_w)$ is the function $v \mapsto v^{-1}(\mathsf{e}_w)$. In view of~\eqref{eqn:det-Steinberg}, these functions are linearly independent in $\Fun(W,Q_T^\wedge)$. Hence indeed our map is injective, and its image is (freely) spanned by these functions as an $R_T^\wedge$-module.
 
 Now, let us identify $\Psi_W$ with a subset of $\Fun(W,R_T^\wedge)$ (see Lemma~\ref{lem:properties-psi}). We claim that $\psi_{w_0}$ belongs to the image of $\tau$. In fact, this is equivalent to the existence of elements $(p_w)_{w \in W}$ in $R_T^\wedge$ such that
 \[
  \tau \left( \sum_{w \in W} p_w \otimes \mathsf{e}_w \right) = \psi_{w_0},
 \]
or in other words (using Lemma~\ref{lem:properties-psi}\eqref{it:psi-delta}--\eqref{it:psi-delta-2}) such that
\[
 \sum_{w \in W} p_w v(\mathsf{e}_w) = \begin{cases}
                              \mathsf{d} & \text{if $v=w_0$;} \\
                              0 & \text{otherwise.}
                             \end{cases}
\]
The arguments above show that there exist unique elements $(p_w)_{w \in W}$ in $Q_T^\wedge$ which satisfy these equalities.
As explained in~\cite[Proof of Theorem~4.4]{kk}, these elements in fact belong to $R_T$, hence in particular to $R_T^\wedge$.

Recall the action of $Q_W$ on $\Omega_W$ considered in~\S\ref{ss:some-RT-modules}. Using the formula~\eqref{eqn:action-y-psi} one sees that for any $a,b \in R_T^\wedge$ and $s \in S$ we have
\[
 y_s \cdot \tau(a \otimes b) = \tau \left( a \otimes \frac{b - e^{-\alpha_s^\vee}s(b)}{1 - e^{-\alpha_s^\vee}} \right).
\]
In particular, this shows that the image of $\tau$ is stable under the operators $y_s$ ($s \in S$). Since (as we have seen above) this image contains $\psi_{w_0}$, by Lemma~\ref{lem:properties-psi}\eqref{it:y-psi} it contains all the elements $\psi_w$ ($w \in W$), hence $\Psi_W$.

It is clear that any function $f$ in the image of $\tau$ satisfies
\[
f(w) \equiv f(ws_{\alpha^\vee}) \mod (1-e^{\alpha^\vee})
\]
for any $w \in W$ and any coroot $\alpha^\vee$. To conclude the proof, it only remains to prove that any function which satisfies these conditions is a linear combination of the elements $(\psi_w)_{w \in W}$. For this we choose a total order on $W$ which extends the Bruhat order, and argue by descending induction on the smallest element $w \in W$ such that $f(w) \neq 0$. Fix $f$, and let $w$ be this smallest element. Then for any positive coroot $\alpha^\vee$ such that $w(\alpha^\vee) \in -\Phi^\vee_+$ we have $ws_{\alpha^\vee} < w$ in the Bruhat order. Hence $f(ws_{\alpha^\vee})=0$, which implies that $f(w)$ is divible by $1-e^{\alpha^\vee}$. By Lemma~\ref{lem:properties-psi}\eqref{it:psi-delta-2} and Lemma~\ref{lem:divisibility}, we deduce that there exists $a \in R_T^\wedge$ such that
\[
f(w) = a \psi_w(\delta_w).
\]
Then $f-a\psi_w$ vanishes on $w$ and all the elements smaller than $w$ (by definition of $w$ and Lemma~\ref{lem:properties-psi}\eqref{it:psi-delta}). By induction we deduce that  $f-a\psi_w$ is a linear combination of elements $(\psi_v)_{v \in W}$, which concludes the proof.
\end{proof}

\subsection{A different description of the algebra \texorpdfstring{$R_T^\wedge \otimes_{(R_T^\wedge)^W} R_T^\wedge$}{RT}}

The results in this subsection do not play any significant role below; we state them only for completeness.

As in Remark~\ref{rmk:RA}, the algebra $R_T^\wedge$ identifies with the algebra of functions on the formal neighborhood $\mathsf{FN}_{T^\vee_\bk}(\{1\})$ of the identity in the $\bk$-torus $T^\vee_\bk$ which is Langlands dual to $T$ (considered as a scheme). Hence $R_T^\wedge \otimes_{(R_T^\wedge)^W} R_T^\wedge$ identifies with the algebra of functions on the fiber product
\[
\mathsf{FN}_{T^\vee_\bk}(\{1\}) \times_{( \mathsf{FN}_{T^\vee_\bk}(\{1\}) )/W} \mathsf{FN}_{T^\vee_\bk}(\{1\}).
\]
On the other hand, consider the formal neighborhood $\mathsf{FN}_{T^\vee_\bk \times_{(T^\vee_\bk)/W} T^\vee_\bk}(\{(1,1)\})$ of the base point in $T^\vee_\bk \times_{(T^\vee_\bk)/W} T^\vee_\bk$ (again, considered as a scheme). By the universal property of the fiber product, there exists a natural morphism of schemes
\begin{equation}
\label{eqn:morph-FN}
\mathsf{FN}_{T^\vee_\bk \times_{(T^\vee_\bk)/W} T^\vee_\bk}(\{(1,1)\}) \to
\mathsf{FN}_{T^\vee_\bk}(\{1\}) \times_{( \mathsf{FN}_{T^\vee_\bk}(\{1\}) )/W} \mathsf{FN}_{T^\vee_\bk}(\{1\}).
\end{equation}

\begin{lem}
\label{lem:formal-neighb}
The morphism~\eqref{eqn:morph-FN} is an isomorphism.
\end{lem}

\begin{proof}
We have to prove that the natural algebra morphism
\begin{equation}
\label{eqn:morph-FN-functions}
R_T^\wedge \otimes_{(R_T^\wedge)^W} R_T^\wedge \to \mathcal{O}(T^\vee_\bk \times_{(T^\vee_\bk)/W} T^\vee_\bk)^\wedge
\end{equation}
is an isomorphism, where the right-hand side is the completion of $\mathcal{O}(T^\vee_\bk \times_{(T^\vee_\bk)/W} T^\vee_\bk)$ at its natural augmentation ideal $J$.

Let $I:=\ker(\varepsilon_T) \subset R_T$; then we have $J = I \otimes_{(R_T)^W} R_T + R_T \otimes_{(R_T)^W} I$. For any $n \in \Z_{\geq 1}$ we have $J^{2n} \subset I^n \otimes_{(R_T)^W} R_T + R_T \otimes_{(R_T)^W} I^n$. Hence for any $w \in W$ the morphism $R_T \otimes_{(R_T)^W} R_T \to R_T / I^n$ sending $a \otimes b$ to the class of $a \cdot w^{-1}(b)$ factors through a morphism  $(R_T \otimes_{(R_T)^W} R_T)/J^{2n} \to R_T / I^n$. From this observation it follows that the morphism $\tau$ of Theorem~\ref{thm:kk-ajs} factors through~\eqref{eqn:morph-FN-functions}, proving that the latter morphism is injective.

On the other hand, from Theorem~\ref{thm:pittie-steinberg} we see that the natural morphism
\[
R_T^\wedge \otimes_{(R_T)^W} R_T \to R_T^\wedge \otimes_{(R_T^\wedge)^W} R_T^\wedge
\]
is an isomorphism; hence $R_T^\wedge \otimes_{(R_T^\wedge)^W} R_T^\wedge$ is the completion of $R_T  \otimes_{(R_T)^W} R_T$ with respect to the ideal $I \otimes_{(R_T)^W} R_T$. Since $I \otimes_{(R_T)^W} R_T \subset J$, we have for any $n$ a surjection
\[
(R_T  \otimes_{(R_T)^W} R_T)/(I \otimes_{(R_T)^W} R_T)^n \twoheadrightarrow (R_T  \otimes_{(R_T)^W} R_T)/J^n.
\]
Since these algebras are finite-dimensional, passing to inverse limits we deduce that~\eqref{eqn:morph-FN-functions} is surjective (by the Mittag--Leffler condition), which finishes the proof.
\end{proof}

\section{Endomorphismensatz}
\label{sec:main-thm}

\subsection{Statement and strategy of proof}

Our goal in this section is to prove the following theorem, which constitutes the main result of this article.

\begin{thm}
\label{thm:main}
  In the case $\hTil=\hTil_{w_0}$, the monodromy morphism of Proposition~{\rm \ref{prop:mon-tilting}} is an algebra isomorphism
  \[
   R_T^\wedge \otimes_{(R_T^\wedge)^W} R_T^\wedge \simto \End(\hTil_{w_0}).
  \]
\end{thm}

Let us note the following consequence, which does not involve the completed category.

\begin{cor}
\label{cor:main}
 The morphism $\varphi_{\lef,\mathscr{T}_{w_0}}^\wedge$ of~{\rm \S\ref{ss:lr-monodromy}} induces an algebra isomorphism
 \[
  R_T / (R_T)^W_+ \simto \End(\mathscr{T}_{w_0}),
 \]
 where $(R_T)^W_+$ is the kernel of the composition $(R_T)^W \hookrightarrow R_T \xrightarrow{\varepsilon_T} \bk$.
\end{cor}

\begin{proof}
 Theorem~\ref{thm:main} and Lemma~\ref{lem:tiltings}\eqref{it:tilting-Hom} imply that monodromy induces an algebra isomorphism
 \[
  R_T^\wedge / (R_T^\wedge)^W_+ \simto \End(\mathscr{T}_{w_0}),
 \]
where $(R_T^\wedge)^W_+$ is the kernel of the composition $(R_T^\wedge)^W \hookrightarrow R_T^\wedge \xrightarrow{\varepsilon_T^\wedge} \bk$.
Hence to conclude it suffices to prove that the morphism
\[
 R_T / (R_T)^W_+ \to R_T^\wedge / (R_T^\wedge)^W_+
\]
induced by the inclusion $R_T \hookrightarrow R_T^\wedge$ is an isomorphism.
However, this morphism is easily seen to be injective. Since (by Theorem~\ref{thm:pittie-steinberg} and its proof) both sides have dimension $\# W$, the desired claim follows.
\end{proof}

In order to prove Theorem~\ref{thm:main} we first remark that, by Lemma~\ref{lem:properties-Tw0}\eqref{it:Tw0-1} and~\eqref{eqn:multiplicities-pidag}, we have
\[
 \gr_w(\hTil_{w_0}) \cong \hDel_w
\]
for any $w \in W$.
We fix such isomorphisms, which provides an isomorphism
\[
 \gr(\hTil_{w_0}) \cong \bigoplus_{w \in W} \hDel_w.
\]
By Lemma~\ref{lem:properties-hD-hN}\eqref{it:Edn-hD}, the right monodromy morphism induces an isomorphism
\[
 R_T^\wedge \simto \End(\hDel_w)
\]
for any $w \in W$. Taking also
Lemma~\ref{lem:Hom-vanishing-hDel} into account, we deduce an algebra isomorphism
\begin{equation}
\label{eqn:End-gr-hTil-w0}
 \End \bigl( \gr(\hTil_{w_0}) \bigr) \cong \bigoplus_{w \in W} R_T^\wedge.
\end{equation}

We now consider the morphisms
\begin{equation}
\label{eqn:morphs-proof-main}
 R_T^\wedge \otimes_{(R_T^\wedge)^W} R_T^\wedge \simto R_T^\wedge \otimes_{(R_T^\wedge)^W} R_T^\wedge \to \End(\hTil_{w_0}) \to \bigoplus_{w \in W} R_T^\wedge,
\end{equation}
where:
\begin{itemize}
 \item 
 the first arrow is given by $a \otimes b \mapsto b \otimes a$;
 \item
 the second arrow is the morphism from Proposition~\ref{prop:mon-tilting};
 \item
 the third arrow is induced by the functor $\gr$, taking into account the isomorphism~\eqref{eqn:End-gr-hTil-w0}.
\end{itemize}
By~\eqref{eqn:monodromy-gr} and Lemma~\ref{lem:mondromy-leftright-hDel}, the composition of the morphisms in~\eqref{eqn:morphs-proof-main} is the morphism $\tau$ of Theorem~\ref{thm:kk-ajs}, if we identify $\bigoplus_{w \in W} R_T^\wedge$ with $\Fun(W,R_T^\wedge)$ in the obvious way. In particular this composition is injective, which proves that the morphism
\[
 R_T^\wedge \otimes_{(R_T^\wedge)^W} R_T^\wedge \to \End(\hTil_{w_0})
\]
from Theorem~\ref{thm:main} is injective. We also deduce (using Theorem~\ref{thm:kk-ajs}) that the image of the third morphism in~\eqref{eqn:morphs-proof-main} contains the subset of vectors $(a_w)_{w \in W}$ such that
\begin{equation}
\label{eqn:congruence-image}
 a_{ws_{\alpha^\vee}} \equiv a_w \mod (1-e^{\alpha^\vee})
\end{equation}
for any coroot $\alpha^\vee$. Below we will prove the following claim.

\begin{prop}
\label{prop:image-conditions-kk}
 If $(a_y)_{y \in W}$ belongs to the image of the third morphism in~\eqref{eqn:morphs-proof-main}, then we have~\eqref{eqn:congruence-image}
for any $w \in W$ and any coroot $\alpha^\vee$.
\end{prop}

This proposition will complete the proof of Theorem~\ref{thm:main}. Indeed, from Corollary~\ref{cor:gr-faithful} we know that the third arrow in~\eqref{eqn:morphs-proof-main} is injective. The discussion above shows that its image coincides with the image of its composition with the second arrow in~\eqref{eqn:morphs-proof-main}. Hence this second arrow (i.e.~the morphism from Theorem~\ref{thm:main}) is surjective.

\subsection{A special case}
\label{ss:special-case}

In this subsection we will prove that if $(a_y)_{y \in W}$ belongs to the image of third morphism in~\eqref{eqn:morphs-proof-main}, then~\eqref{eqn:congruence-image} holds when $w=e$ and $\alpha^\vee$ is a simple coroot. We will denote by $\alpha$ the (simple) root associated with $\alpha^\vee$. To simplify notation, we set $s:=s_{\alpha^\vee}$.

We will denote by $\overline{\jmath}_s$ the (closed) embedding of $\pi^{-1}(\overline{Y_s}) = X_s \sqcup X_e$ in $X$.

\begin{lem}
\label{lem:restr-hTil-w0}
 We have $\overline{\jmath}_s^*(\hTil_{w_0}) \cong \hTil_s$. Moreover, the morphism
 \[
  \gr_w(\hTil_{w_0}) \to \gr_w(\hTil_s)
 \]
induced by the adjunction morphism $\hTil_{w_0} \to (\overline{\jmath}_s)_* \overline{\jmath}_s^* \hTil_{w_0} = \hTil_s$ is an isomorphism if $w \in \{e,s\}$, and $0$ otherwise.
\end{lem}

\begin{proof}
 Since $\hTil_{w_0}$ is tilting (in particular, admits a standard filtration), it is clear that the adjunction morphism $\hTil_{w_0} \to (\overline{\jmath}_s)_* (\overline{\jmath}_s)^* \hTil_{w_0}$ is surjective, and that the induced morphism $\gr_w(\hTil_{w_0}) \to \gr_w \bigl( (\overline{\jmath}_s)_* \overline{\jmath}_s^* \hTil_{w_0} \bigr)$ is an isomorphism if $w \in \{e,s\}$, and $0$ otherwise. Hence it suffices to prove the isomorphism $\overline{\jmath}_s^*(\hTil_{w_0}) \cong \hTil_s$. However, if we still denote by $\pi$ the morphism $\pi^{-1}(\overline{Y_s}) \to \overline{Y_s}$ induced by $\pi$, we have
 \[
  \pi_\dag(\overline{\jmath}_s^* \hTil_{w_0}) \cong \overline{\imath}_s^* \pi_\dag (\hTil_{w_0}) = \overline{\imath}_s^* \mathscr{T}_{w_0}
 \]
where $\overline{\imath}_s : \overline{Y_s} \to Y$ is the embedding (see~\eqref{eqn:isom-pi-embeddings}). By Lemma~\ref{lem-Tw0-s}, it follows that
\[
 \pi_\dag(\overline{\jmath}_s^* \hTil_{w_0}) \cong \mathscr{T}_s.
\]
We deduce the desired isomorphism, in view of Proposition~\ref{prop:classification-tiltings}.
\end{proof}

\begin{rmk}
\label{rmk:hTils-canonical}
The objects $\hTil_w$ are not canonical; they can be chosen only up to isomorphism. (This does not affect Theorem~\ref{thm:main}, since monodromy commutes with any morphism, hence is invariant under conjugation in the obvious sense.) However, the proof of Lemma~\ref{lem:restr-hTil-w0} shows that once $\hTil_{w_0}$ is chosen, the object $\hTil_s$ (for any $s \in S$) can be defined canonically as $(\overline{\jmath}_s)_* \overline{\jmath}_s^* \hTil_{w_0}$.
\end{rmk}

From Lemma~\ref{lem:restr-hTil-w0} we deduce that the composition
\[
 \End(\hTil_{w_0}) \to \bigoplus_{w \in W} R_T^\wedge \xrightarrow{(a_w)_{w \in W} \mapsto (a_e,a_s)} R_T^\wedge \oplus R_T^\wedge
\]
(where the first arrow is the third morphism in~\eqref{eqn:morphs-proof-main})
factors as the composition
\begin{equation}
\label{eqn:image-s}
 \End(\hTil_{w_0}) \xrightarrow{\overline{\jmath}_s^*} \End(\hTil_s) \xrightarrow{\gr} R_T^\wedge \oplus R_T^\wedge.
\end{equation}
Now by Corollary~\ref{cor:mon-Ts} the morphism
\[
 R_T^\wedge \otimes_{(R_T^\wedge)^W} R_T^\wedge \to \End(\hTil_s)
\]
of Proposition~\ref{prop:mon-tilting} is surjective, and its composition with the second arrow in~\eqref{eqn:image-s} identifies with the morphism
\[
 a \otimes b \mapsto (a \cdot b, s(a) \cdot b)
\]
(see Lemma~\ref{lem:mondromy-leftright-hDel}). Since $ab \equiv s(a) b \mod (1-e^{\alpha^\vee})$, this proves that if $(a_y)_{y \in W}$ belongs to the image of the third morphism in~\eqref{eqn:morphs-proof-main}, then indeed we have $a_e \equiv a_s \mod (1-e^{\alpha^\vee})$.

\subsection{The general case}

In this subsection we deduce Proposition~\ref{prop:image-conditions-kk} from the special case considered in~\S\ref{ss:special-case}. The main idea will be the following: recall diagram~\eqref{eqn:morphs-proof-main}. We have natural actions of $W \times W$ on the first, second, and fourth terms in this diagram respectively defined by
\begin{gather*}
 \vartheta^{(1)}_{(w,v)}(a \otimes b) = w(a) \otimes v(b), \quad \vartheta^{(2)}_{(w,v)}(a \otimes b) = v(a) \otimes w(b), \\
 \bigl( \vartheta^{(4)}_{(w,v)}(f) \bigr)(x) = w(f(v^{-1} x w))
\end{gather*}
for $w,v \in W$, $a,b \in R_T^\wedge$, $f \in \Fun(W,R_T^\wedge)$, $x \in W$. It is easily seen that the first arrow and the composition of the second and third arrows are equivariant with respect to these actions.
We will now define an action of $W \times W$ on $\End(\hTil_{w_0})$ that makes the whole diagram~\eqref{eqn:morphs-proof-main} equivariant. This will imply that the image of the third morphism in this diagram is stable under this $(W \times W)$-action.

For $w,v \in W$ we denote by $\vartheta^{(3)}_{(w,v)}$ the automorphism of $\End(\hTil_{w_0})$ defined as the composition
\[
 \End(\hTil_{w_0}) \simto \End(\hDel_v \hatstar \hTil_{w_0} \hatstar \hDel_{w^{-1}}) \simto \End(\hTil_{w_0})
\]
where the first arrow is induced by the functor $\hDel_v \hatstar (-) \hatstar \hDel_{w^{-1}}$ and the second arrow is induced by any choice of isomorphism as in Proposition~\ref{prop:convolution-hTw0}.

\begin{lem}
For any $w,v \in W$,
 the automorphism $\vartheta^{(3)}_{(w,v)}$ does not depend on the choice of isomorphism as in Proposition~{\rm \ref{prop:convolution-hTw0}}. Moreover, these isomorphisms define an action of $W \times W$ on $\End(\hTil_{w_0})$, and the second and third arrows in~\eqref{eqn:morphs-proof-main} are equivariant with respect to this action and the ones defined above.
\end{lem}

\begin{proof}
 First, we claim that the second morphism in~\eqref{eqn:morphs-proof-main} intertwines the automorphisms $\vartheta^{(2)}_{(w,v)}$ and $\vartheta^{(3)}_{(w,v)}$. For this we remark that the image under this morphism of $a \otimes b$ is $\varphi_{\lef,\hTil_{w_0}}^\wedge(a) \circ \varphi_{\rig,\hTil_{w_0}}^\wedge(b)$. Now we have
 \begin{multline*}
  \id_{\hDel_v} \hatstar \left( \varphi_{\lef,\hTil_{w_0}}^\wedge(a) \circ \varphi_{\rig,\hTil_{w_0}}^\wedge(b) \right) \hatstar \id_{\hDel_{w^{-1}}} = \varphi_{\rig, \hDel_v}^\wedge(a) \hatstar \id_{\hTil_{w_0}} \hatstar \varphi_{\lef,\hDel_{w^{-1}}}^\wedge(b) \\
  = \varphi_{\lef, \hDel_v}^\wedge(v(a)) \hatstar \id_{\hTil_{w_0}} \hatstar \varphi_{\rig,\hDel_{w^{-1}}}^\wedge(w(b)) \\
  = \varphi_{\lef,\hDel_v \hatstar \hTil_{w_0} \hatstar \hDel_{w^{-1}}}^\wedge(v(a)) \circ \varphi_{\rig,\hDel_v \hatstar \hTil_{w_0} \hatstar \hDel_{w^{-1}}}^\wedge(w(b)),
 \end{multline*}
where the first and third equalities follow from Lemma~\ref{lem:monodromy-convolution}, and the second one from Lemma~\ref{lem:mondromy-leftright-hDel}. Now, by functoriality of monodromy, the conjugate of this automorphism with any choice of isomorphism $\hDel_v \hatstar \hTil_{w_0} \hatstar \hDel_{w^{-1}} \simto \hTil_{w_0}$ is $\varphi_{\lef,\hTil_{w_0}}^\wedge(v(a)) \circ \varphi_{\rig,\hTil_{w_0} }^\wedge(w(b))$, which concludes the proof of our claim.

We have already remarked that all the $R_T^\wedge$-modules appearing in~\eqref{eqn:morphs-proof-main} are free of rank $\# W$ (see in particular Lemma~\ref{lem:tiltings}\eqref{it:tilting-Hom} and Theorem~\ref{thm:pittie-steinberg}). Moreover, from the proof of Theorem~\ref{thm:kk-ajs} we see that the image under the functor $Q_T^\wedge \otimes_{R_T^\wedge} -$ (where, as in~\S\ref{ss:some-RT-modules}, $Q_T^\wedge$ is the fraction field of $R_T^\wedge$) of the composition of the three arrows in this diagram is an isomorphism. Hence the same property holds for any of the maps in this diagram. Since the composition of the second and third maps intertwines $\vartheta^{(2)}_{(w,v)}$ and $\vartheta^{(4)}_{(w,v)}$, and since the second map intertwines $\vartheta^{(2)}_{(w,v)}$ and $\vartheta^{(3)}_{(w,v)}$, we deduce that the third map intertwines $\vartheta^{(3)}_{(w,v)}$ and $\vartheta^{(4)}_{(w,v)}$. Since this map is injective, from this property we see that $\vartheta^{(3)}_{(w,v)}$ does not depend on the choice of isomorphism as in Proposition~\ref{prop:convolution-hTw0}, and that these isomorphisms define an action of $W \times W$ on $\End(\hTil_{w_0})$.
\end{proof}

\begin{proof}[Proof of Proposition~{\rm \ref{prop:image-conditions-kk}}]
 First, we assume that $\alpha^\vee$ is a simple coroot. Then since $\vartheta^{(4)}_{(e,w^{-1})}((a_y)_{y \in W})$ also belongs to the image of the third map in~\eqref{eqn:morphs-proof-main}, by the special case considered in~\S\ref{ss:special-case} we must have
 \[
  a_w \equiv a_{ws_{\alpha^\vee}} \mod (1-e^{\alpha^\vee}),
 \]
as desired.

Now we consider the general case. We choose $v \in W$ such that $\beta^\vee := v(\alpha^\vee)$ is a simple coroot. To prove that $a_w \equiv a_{ws_{\alpha^\vee}} \mod (1-e^{\alpha^\vee})$ we only have to prove that
 \[
  v(a_w) \equiv v(a_{ws_{\alpha^\vee}}) \mod (1-e^{\beta^\vee}).
 \]
 However, since $ws_{\alpha^\vee} = w v^{-1} s_{\beta^\vee} v$, this fact follows from the observation that $\vartheta^{(4)}_{(v,e)}((a_y)_{y \in W})$ also belongs to the image of the third map in~\eqref{eqn:morphs-proof-main}, and the case of simple coroots treated above (applied with ``$w$'' replaced by $wv^{-1}$).
\end{proof}

\section{Variant: the \'etale setting}
\label{sec:etale}

All the constructions we have considered so far have counterparts in the world of \'etale sheaves, which we briefly review in this section. Here we need to assume that $\bk$ is a finite field, and will denote its characteristic by $\ell$.

\subsection{Completed derived categories}

We choose an algebraically closed field $\F$ of characteristic $p \neq \ell$. Instead of considering a \emph{complex} connected reductive group, one can consider a connected reductive group $\bG$ over $\F$, a Borel subgroup $\bB \subset \bG$ and a maximal torus $\bT \subset \bB$. Then we denote by $\bU$ the unipotent radical of $\bB$, and we set $\mathbf{X}:=\bG/\bU$, $\mathbf{Y}:=\bG/\bB$. We will denote by $\Dbet_{c}(\mathbf{X},\bk)$ and $\Dbet_{c}(\mathbf{Y},\bk)$ the bounded constructible derived categories of \'etale $\bk$-sheaves on $\mathbf{X}$ and $\mathbf{Y}$, respectively. Then one can define the subcategory $\Dbet_{\bU}(\mathbf{Y},\bk) \subset \Dbet_{c}(\mathbf{Y},\bk)$ as the $\bU$-equivariant\footnote{Recall that in the \'etale setting the $\bU$-equivariant and $\bB$-constructible derived categories are different if $p>0$, due to the existence of non-constant local systems on affine spaces. Here $\Dbet_{\bU}(\mathbf{Y},\bk)$ is the full triangulated subcategory of $\Dbet_{c}(\mathbf{Y},\bk)$ generated by pushforwards of \emph{constant} local systems on strata.} derived category of $\mathbf{Y}$, and out of that define the associated categories $\Dbet_{\bU}(\mathbf{X} \quot \bT,\bk)$ and $\hDet_{\bU}(\mathbf{X} \quot \bT,\bk)$ exactly as above.

In this setting, the monodromy construction (see Section~\ref{sec:monodromy}) is a bit more subtle, but the required work has been done by Verdier~\cite{verdier}. Namely, we start by choosing once and for all a topological generator $(x_n)_{n \geq 0}$ of the pro-finite group
\[
\varprojlim_n \{x \in \F \mid x^{\ell^n}=1\}
\]
(where the transition maps are given by $x \mapsto x^\ell$).
As in the proof of Lemma~\ref{lem:averaging-hL} we denote, for $n \geq 0$, by $[n] : \bT \to\bT$ the morphism $z \mapsto z^{\ell^n}$, and set $a_n:=a \circ ([n] \times \id_\bX)$, where $a : \bT \times \mathbf{X} \to \mathbf{X}$ is the action morphism. Then given $\mathscr{F}$ in $\Dbet_{\bU}(\mathbf{X} \quot \bT,\bk)$, for $n \gg 0$ there exists an isomorphism
\[
f_n^{\mathscr{F}} : (a_n)^* \mathscr{F} \simto p^* \mathscr{F}
\]
whose restriction to $\{1\} \times \bX$ is the identity.
Moreover, these isomorphisms are essentially unique and functorial in the same sense as in the proof of Lemma~\ref{lem:averaging-hL}; see~\cite[Proposition~5.1]{verdier}. Given $\lambda \in X_*(\bT)$, restricting the isomorphism $f_n^{\mathscr{F}}$ to $\{\lambda(x_n)\} \times \bX$ (for $n \gg 0$) provides a canonical automorphism of $\mathscr{F}$, which by definition is $\varphi_{\mathscr{F}}^\lambda$. Starting with these automorphisms one obtains the morphism $\varphi^\wedge_{\mathscr{F}}$, which still satisfies the properties of~\S\ref{ss:monodromy-properties}.

Lemma~\ref{lem:mon-equiv} continues to hold in this setting, but its proof has to be adapted to the new definition of monodromy. Note that when $\mathscr{F}$ is a perverse sheaf the morphisms $f_n^{\mathscr{F}}$ are unique when they exist; in other words they are determined by the condition that their restriction to $\{1\} \times \bX$ is the identity. So, if $\mathscr{F}$ is as in Lemma~\ref{lem:mon-equiv}, there exists $n$ and an isomorphism $f_n^{\mathscr{F}} : (a_n)^* \mathscr{F} \simto p^* \mathscr{F}$ whose restriction to $\{1\} \times \bX$ is the identity. The fact that the monodromy is trivial means that the restriction of $f_n^{\mathscr{F}}$ to $\{x_n\} \times \bX$ is the identity also. Hence the pullback of $f_n^{\mathscr{F}}$ under the automorphism of $\Gm \times \bX$ sending $(z,x)$ to $(zx_n,x)$ is also an isomorphism $(a_n)^* \mathscr{F} \simto p^* \mathscr{F}$ whose restriction to $\{1\} \times \bX$ is the identity; therefore this isomorphism coincides with $f_n^{\mathscr{F}}$. Now the morphism $[n] \times \id_\bX$ is \'etale since $p \neq \ell$, and our observation amounts to saying that the morphism $f_n^{\mathscr{F}}$ satisfies the property that its pullbacks under both projections $(\Gm \times \bX) \times_{(\Gm \times \bX)} (\Gm \times \bX) \to \Gm \times \bX$ (where the fiber product is taken with respect to the morphism $[n] \times \bX$ on both sides) coincide. Since perverse sheaves form a stack for the \'etale topology (see~\cite[\S 2.2.19]{bbd}), it follows that this morphism descends to an isomorphism $a^* \mathscr{F} \simto p^* \mathscr{F}$; in other words $\mathscr{F}$ is a $\Gm$-equivariant perverse sheaf.

Next, the \'etale fundamental group $\pi_1^{\mathrm{et}}(\Gm)$ of $\Gm$ is more complicated than $\pi_1(\C^\times)$. However, the \'etale covers $[n] : \Gm \to \Gm$ define a surjective morphism
\[
\pi_1^{\mathrm{et}}(\Gm) \twoheadrightarrow \varprojlim_n \{x \in \F \mid x^{\ell^n}=1\}.
\]
Recall that we have fixed a topological generator of the right-hand side; this allows us to identify this group with $\varprojlim_n \Z/\ell^n \Z$.
We have a natural isomorphism
\[
X_*(\bT) \otimes_\Z \pi_1^{\mathrm{et}}(\Gm) \simto \pi_1^{\mathrm{et}}(\bT),
\]
hence a natural surjection
\[
\pi_1^{\mathrm{et}}(\bT) \to X_*(\bT) \otimes_\Z \left( \varprojlim_n \Z/\ell^n \Z \right).
\]
For $n \geq 0$, one can then consider the quotient $R_\bT / \mathfrak{m}_{\bT}^{\ell^n}$, with its natural action of $X_*(\bT)$. This action factors through an action of $X_*(\bT) \otimes_\Z \Z/\ell^n \Z$, hence it defines an action of $X_*(\bT) \otimes_\Z \left( \varprojlim_n \Z/\ell^n \Z \right)$. By pullback we deduce a finite-dimensional continuous $\pi_1^{\mathrm{et}}(\bT)$-module; the corresponding $\bk$-local system on $\bT$ will be denoted $\mathscr{L}^{\mathrm{et}}_{\bT,n}$. Then we can define the pro-unipotent local system as
\[
\widehat{\mathscr{L}}^{\mathrm{et}}_{\bT} = ``\varprojlim_n"  \mathscr{L}^{\mathrm{et}}_{\bT,n}.
\]
Using this object as a replacement for $\widehat{\mathscr{L}}_\bT$, all the constructions of Sections~\ref{sec:trivial-torsor}--\ref{sec:perverse} carry over to the present context, with identical proofs.

\subsection{Soergel's Endomorphismensatz}

Once the formalism of completed categories is in place, all the considerations of Sections~\ref{sec:flag-tilting}--\ref{sec:convolution} carry over also. This allows one to extend the results of Section~\ref{sec:main-thm}, in particular Theorem~\ref{thm:main} and Corollary~\ref{cor:main}, to the \'etale setting (assuming that $\bG$ is semisimple, of adjoint type).

\subsection{Whittaker derived category}
\label{ss:Whittaker}

The main point of introducing the \'etale variant is that one can combine our considerations with the following ``Whittaker-type'' construction. Here we have to assume that there exists a primitive $p$-th root of unity in $\F$; we will fix once and for all a choice of such a root.

Let $\bU^+$ be the unipotent radical of the Borel subgroup of $\bG$ opposite to $\bB$ with respect to $\bT$, and choose for any $s$ an isomorphism between the root subgroup of $\bG$ associated with the simple root corresponding to $s$ and the additive group $\Ga$. (Here we assume that the roots of $\bB$ are the \emph{negative} roots.) We deduce an isomorphism $\bU^+ / [\bU^+,\bU^+] \cong (\Ga)^{S}$. Composing with the addition map $(\Ga)^S \to \Ga$ we deduce a ``non-degenerate'' morphism $\chi : \bU^+ \to \Ga$. Our choice of primitive $p$-th root of unity determines an Artin--Schreier local system on $\Gm$, whose pullback to $\bU^+$ will be denoted $\mathscr{L}_\chi$. Then we can define the ``Whittaker'' derived category $\Dbet_\Whit(\mathbf{Y},\bk)$ as the full subcategory of $\Dbet_c(\mathbf{Y},\bk)$ consisting of $(\bU^+,\mathscr{L}_\chi)$-equivariant objects. (See e.g.~\cite[Appendix~A]{modrap1} for a reminder on the construction of this category.) If $j : \bU^+ \bB/\bB \hookrightarrow \mathbf{Y}$ is the (open) embedding then, for any $\mathscr{F}$ in $\Dbet_\Whit(\mathbf{Y},\bk)$, adjunction provides isomorphisms
\[
j_! j^* \mathscr{F} \simto \mathscr{F} \simto j_* j^* \mathscr{F}.
\]

Next, we can define the corresponding category $\Dbet_\Whit(\mathbf{X} \quot \bT,\bk)$ as the triangulated subcategory generated by the objects of the form $\pi^\dag \mathscr{F}$ with $\mathscr{F}$ in $\Dbet_\Whit(\mathbf{Y},\bk)$, and deduce a completed category $\hDet_\Whit(\mathbf{X} \quot \bT,\bk)$. 
If $\widehat{\jmath} : \pi^{-1}(\bU^+ \bB/\bB) \hookrightarrow \mathbf{X}$ is the embedding then, for any object $\mathscr{F}$ in $\hDet_\Whit(\mathbf{X} \quot \bT,\bk)$, adjunction provides isomorphisms
\[
\widehat{\jmath}_! \widehat{\jmath}^* \mathscr{F} \simto \mathscr{F} \simto \widehat{\jmath}_* \widehat{\jmath}^* \mathscr{F}.
\]
In particular, using the obvious projection $\pi^{-1}(\bU^+ \bB/\bB) = \bU^+ \bB / \bU \cong \bU^+ \times \bT \to \bT$ we obtain a canonical equivalence of triangulated categories
\begin{equation}
\label{eqn:equiv-hDWhit}
 \Db \Mod^{\mathrm{fg}}(R_{\bT}^\wedge) \simto \hDet_\Whit(\mathbf{X} \quot \bT,\bk).
\end{equation}
The image of the free rank-$1$ $R_\bT^\wedge$-module is the standard object $\hDel_\chi$ constructed as in~\S\ref{ss:standard-costandard} (with respect to the orbit $\bU^+ \bB/\bB \subset \mathbf{X}$). This object is canonically isomorphic to the corresponding costandard object $\hnab_\chi$. Transporting the tautological t-structure along the equivalence~\eqref{eqn:equiv-hDWhit} we obtain a t-structure on $\hDet_\Whit(\mathbf{X} \quot \bT,\bk)$ which we will call the perverse t-structure, and whose heart will be denoted $\hP^{\mathrm{et}}_\Whit(\mathbf{X} \quot \bT,\bk)$.

The categories $\Dbet_\Whit(\mathbf{Y},\bk)$, $\Dbet_\Whit(\mathbf{X} \quot \bT,\bk)$ and $\hDet_\Whit(\mathbf{X} \quot \bT,\bk)$ are related to the categories $\Dbet_{\bU}(\mathbf{Y},\bk)$, $\Dbet_{\bU}(\mathbf{X} \quot \bT,\bk)$ and $\hDet_{\bU}(\mathbf{X} \quot \bT,\bk)$ in several ways. First, the convolution construction of Section~\ref{sec:convolution} defines a right action of the monoidal category $\bigl( \hDet_{\bU}(\mathbf{X} \quot \bT,\bk), \hatstar \bigr)$ on $\hDet_\Whit(\mathbf{X} \quot \bT,\bk)$; the corresponding bifunctor will again be denoted $\hatstar$. Next, we have ``averaging'' functors $\Dbet_\bU(\mathbf{Y},\bk) \to \Dbet_\Whit(\mathbf{Y},\bk)$ and $\Dbet_\bU(\mathbf{X} \quot \bT,\bk) \to \Dbet_\Whit(\mathbf{X} \quot \bT,\bk)$, sending a complex $\mathscr{F}$ to $(a_{\bU^+})_!(\mathscr{L}_\chi \boxtimes \mathscr{F})[\dim \bU^+]$, where $a_{\bU^+} : \bU^+ \times \mathbf{Y} \to \mathbf{Y}$ and $a_{\bU^+} : \bU^+ \times \mathbf{X} \to \mathbf{X}$ are the natural morphisms. Standard arguments (see~\cite{bbm, by}) show that $(a_{\bU^+})_!$ can be replaced by $(a_{\bU^+})_*$ in this formula without changing the functor up to isomorphism. These functors will be denoted $\Av_\chi$; then we have canonical isomorphisms
\[
\Av_\chi \circ \pi_\dag \cong \pi_\dag \circ \Av_\chi, \qquad \Av_\chi \circ \pi^\dag \cong \pi^\dag \circ \Av_\chi.
\]
In particular, we obtain an induced functor
\[
\Av_\chi : \hDet_{\bU}(\mathbf{X} \quot \bT,\bk) \to \hDet_{\Whit}(\mathbf{X} \quot \bT,\bk).
\]
By construction, this functor satisfies
\[
\Av_\chi(\hDel_e) = \hDel_\chi.
\]

We also have averaging functors in the other direction, defined in terms of the action morphisms $a_\bU : \bU \times \mathbf{Y} \to \mathbf{Y}$ and $a_\bU : \bU \times \mathbf{X} \to \mathbf{X}$ and the constant local system on $\bU$. This time, the versions with $*$- and $!$-pushforwards are different, and will be denoted $\Av^{\bU}_{*}$ and $\Av^{\bU}_{!}$. Here also we have isomorphisms
\[
\Av^{\bU}_{?} \circ \pi_\dag \cong \pi_\dag \circ \Av^{\bU}_?, \qquad \Av^{\bU}_{?} \circ \pi^\dag \cong \pi^\dag \circ \Av^{\bU}_?
\]
for $? \in \{*,!\}$
(see the arguments in~\cite[Proof of Corollary~A.3.4]{by} for the first isomorphism in the case $?=*$). Hence we deduce induced functors
\[
\Av^{\bU}_! : \hDet_{\Whit}(\mathbf{X} \quot \bT,\bk) \to \hDet_{\bU}(\mathbf{X} \quot \bT,\bk), \quad \Av^{\bU}_* : \hDet_{\Whit}(\mathbf{X} \quot \bT,\bk) \to \hDet_{\bU}(\mathbf{X} \quot \bT,\bk).
\]
Standard arguments (see e.g.~\cite[Lemma~4.4.5]{by} or~\cite[Lemma~5.15]{modrap1}) show that the pairs $(\Av^{\bU}_!, \Av_\chi)$ and $(\Av_\chi, \Av^\bU_*)$ form adjoint pairs of functors.

\subsection{Geometric construction of \texorpdfstring{$\hTil_{w_0}$}{Tw0}}

The Whittaker constructions of~\S\ref{ss:Whittaker} allow us in particular to give a concrete and explicit description of the objects $\hTil_{w_0}$ and $\mathscr{T}_{w_0}$, as follows.

\begin{lem}
\label{lem:Tw0-etale}
There exist isomorphisms
\[
\mathscr{T}_{w_0} \cong \Av^\bU_! \circ \Av_\chi(\Delta_e) \cong \Av^\bU_* \circ \Av_\chi(\Delta_e), \quad
\hTil_{w_0} \cong \Av^\bU_! \circ \Av_\chi(\hDel_e) \cong \Av^\bU_* \circ \Av_\chi(\hDel_e). 
\]
\end{lem}

\begin{proof}
Since the averaging functors commute with $\pi_\dag$, in view of the characterization of $\hTil_{w_0}$ in Proposition~\ref{prop:classification-tiltings} it is sufficient to prove the isomorphisms $\mathscr{T}_{w_0} \cong \Av^\bU_! \circ \Av_\chi(\Delta_e) \cong \Av^\bU_* \circ \Av_\chi(\Delta_e)$. This follows from standard arguments, showing that $\Av^\bU_! \circ \Av_\chi(\Delta_e)$ is the projective cover of $\IC_e$ and that $\Av^\bU_* \circ \Av_\chi(\Delta_e)$ is the injective hull of $\IC_e$ and then using Lemma~\ref{lem:properties-Tw0}\eqref{it:Tw0-2}; see~\cite[Lemma~4.4.11]{by} or~\cite[Lemma~5.18]{modrap1} for details.
\end{proof}

\begin{rmk}
As explained above, Lemma~\ref{lem:Tw0-etale} provides a canonical representative for the object $\hTil_{w_0}$ (in the present \'etale setting). In view of
Remark~\ref{rmk:hTils-canonical}, the objects $\hTil_s$ with $s \in S$ are then also canonically defined.
\end{rmk}

\section{Soergel theory}
\label{sec:soergel-theory}

In this section we use Theorem~\ref{thm:main} and Corollary~\ref{cor:main} to obtain a description of tilting objects in $\scO$ and $\hscO$ in terms of some kinds of Soergel bimodules. For simplicity, we assume that $\bk$ is a finite field. (This assumption does not play any role in~\S\S\ref{ss:functor-V}--\ref{ss:image-Ts}.)

In~\S\S\ref{ss:functor-V}--\ref{ss:image-Ts} we 
work either in the ``classical'' setting of Sections~\ref{sec:flag-tilting}--\ref{sec:main-thm} or in the \'etale setting of Section~\ref{sec:etale}. (For simplicity we do not distinguish the two cases, and use the notation of Sections~\ref{sec:flag-tilting}--\ref{sec:main-thm}.)
Then in~\S\ref{ss:monoidal-structure} we consider a construction that is available only in the \'etale setting, and in~\S\ref{ss:monoidal-classical} we explain how to extend the results obtained using this construction to the classical setting. Finally, in~\S\ref{ss:soergel-theory} we derive an explicit description of the categories of tilting objects in $\mathscr{O}$ and $\widehat{\mathscr{O}}$.

\subsection{The functor \texorpdfstring{$\mathbb{V}$}{V}}
\label{ss:functor-V}

We fix a representative $\hTil_{w_0}$, and set $\mathscr{T}_{w_0} := \pi_\dag(\hTil_{w_0})$ (so that $\mathscr{T}_{w_0}$ is as above the indecomposable tilting object in $\scO$ associated with $w_0$, but now chosen in a slightly more specific way).

Thanks to Theorem~\ref{thm:main} and Corollary~\ref{cor:main} respectively, we have isomorphisms
\[
R_T^\wedge \otimes_{(R_T^\wedge)^W} R_T^\wedge \simto \End(\hTil_{w_0}), \quad R_T/(R_T)^W_+ \simto \End(\mathscr{T}_{w_0}),
\]
so that we can consider the functors
\begin{align*}
\widehat{\mathbb{V}} : \hscO &\to \Mof(R_T^\wedge \otimes_{(R_T^\wedge)^W} R_T^\wedge) \\
\mathbb{V} : \scO &\to \Mof(R_T/(R_T)^W_+)
\end{align*}
(where, for $A$ a Noetherian ring, we denote by $\Mof(A)$ the abelian category of finitely generated left $A$-modules) defined by
\[
\widehat{\mathbb{V}}(\widehat{\mathscr{F}}) = \Hom(\hTil_{w_0}, \widehat{\mathscr{F}}), \quad \mathbb{V}(\mathscr{F}) = \Hom(\mathscr{T}_{w_0},\mathscr{F}).
\]
Here, the fact that $\mathbb{V}$ takes values in $\Mof(R_T/(R_T)^W_+)$ is obvious, while for $\widehat{\mathbb{V}}$ the corresponding property follows from Corollary~\ref{cor:Hom-fg}\eqref{it:Hom-fg}.
If $\hTil$ is a tilting object in $\hscO$, then by Lemma~\ref{lem:tiltings}\eqref{it:tilting-Hom} we have a canonical isomorphism
\begin{equation}
\label{eqn:V-hV}
\bk \otimes_{R_T^\wedge} \widehat{\mathbb{V}}(\hTil) \cong \mathbb{V}(\pi_\dag \hTil),
\end{equation}
where the tensor product is taken with respect to the action of the right copy of $R_T^\wedge$.

\begin{rmk}
Lemma~\ref{lem:formal-neighb} shows that the category $\Mof(R_T^\wedge \otimes_{(R_T^\wedge)^W} R_T^\wedge)$ can be described more geometrically as the category of coherent sheaves on the formal neighborhood of the point $(1,1)$ in $T^\vee_\bk \times_{(T^\vee_\bk)^W} T^\vee_\bk$ (considered as a scheme). Similarly, the category $\Mof(R_T/(R_T)^W_+)$ is the category of coherent sheaves on the fiber of the quotient morphism $T^\vee_\bk \to (T^\vee_\bk)/W$ over the image of $1$. In these terms, the monoidal structure on $\Mof(R_T^\wedge \otimes_{(R_T^\wedge)^W} R_T^\wedge)$ considered in~\S\ref{ss:monoidal-structure} below can be described as a convolution product.
\end{rmk}

These functors are ``fully faithful on tilting objects'' in the following sense.

\begin{prop}
\label{prop:ff}
For any tilting perverse sheaves $\hTil$, $\hTil'$ in $\hscO$, the functor $\widehat{\mathbb{V}}$ induces an isomorphism
\[
\Hom_{\hscO}(\hTil,\hTil') \simto \Hom_{R_T^\wedge \otimes_{(R_T^\wedge)^W} R_T^\wedge} \bigl( \widehat{\mathbb{V}}(\hTil),\widehat{\mathbb{V}}(\hTil') \bigr).
\]
Similarly, for any tilting perverse sheaves $\mathscr{T}$, $\mathscr{T}'$ in $\scO$, the functor $\mathbb{V}$ induces an isomorphism
\[
\Hom_{\scO}(\mathscr{T},\mathscr{T}') \simto \Hom_{R_T/(R_T)^W_+} \bigl( \mathbb{V}(\mathscr{T}), \mathbb{V}(\mathscr{T}') \bigr).
\]
\end{prop}

\begin{proof}
The second case is treated in~\cite[\S 2.1]{bbm}. Here we prove both cases using a closely related argument explained in~\cite[\S 4.7]{by}.

We start with the case of the functor $\mathbb{V}$. We remark that this functor admits a left adjoint $\gamma : \Mof(R_T/(R_T)^W_+) \to \scO$ defined by $\gamma(M)=\mathscr{T}_{w_0} \otimes_{R_T/(R_T)^W_+} M$. More concretely, if $M$ is written as the cokernel of a map $f : (R_T/(R_T)^W_+)^{\oplus n} \to (R_T/(R_T)^W_+)^{\oplus m}$, then in view of the isomorphism $R_T/(R_T)^W_+ \simto \End(\mathscr{T}_{w_0})$ the map $f$ defines a morphism $(\mathscr{T}_{w_0})^{\oplus n} \to (\mathscr{T}_{w_0})^{\oplus m}$, whose cokernel is $\gamma(M)$. From this description and using the exactness of $\mathbb{V}$ (see Lemma~\ref{lem:properties-Tw0}\eqref{it:Tw0-2}), we see that the adjunction morphism $\id \to \mathbb{V} \circ \gamma$ is an isomorphism.

We now assume that $\mathscr{T}$ is a tilting perverse sheaf, and consider the adjunction morphism
\begin{equation}
\label{eqn:adjunction-V}
\gamma(\mathbb{V}(\mathscr{T})) \to \mathscr{T}.
\end{equation}
The image of this morphism under $\mathbb{V}$ is an isomorphism, since its composition with the (invertible) adjunction morphism $\id \to \mathbb{V} \circ \gamma$ applied to $\mathbb{V}(\mathscr{T})$ is $\id_{\mathbb{V}(\mathscr{T})}$. Hence its kernel and cokernel are killed by $\mathbb{V}$; in other words, they do not admit $\IC_e$ as a composition factor.
In view of Corollary~\ref{cor:sub-quo-filtered}, this shows that the cokernel of this morphism vanishes, i.e.~that~\eqref{eqn:adjunction-V} is surjective. Moreover, if $\mathscr{T}'$ is another tilting object in $\scO$, then the kernel of this morphism does not admit any nonzero morphism to $\mathscr{T}'$, again by Corollary~\ref{cor:sub-quo-filtered}. Hence the induced morphism
\[
\Hom(\mathscr{T},\mathscr{T}') \to \Hom(\gamma(\mathbb{V}(\mathscr{T})),\mathscr{T}')
\]
is an isomorphism, which finishes the proof in this case.

Now we consider the case of $\widehat{\mathbb{V}}$. As for $\mathbb{V}$, this functor admits a left adjoint
\[
\widehat{\gamma} : \Mof(R_T^\wedge \otimes_{(R_T^\wedge)^W} R_T^\wedge) \to \hscO
\]
defined by
$\widehat{\gamma}(M) = \hTil_{w_0} \otimes_{R_T^\wedge \otimes_{(R_T^\wedge)^W} R_T^\wedge} M$; in more concrete terms if $M$ is the cokernel of a map $(R_T^\wedge \otimes_{(R_T^\wedge)^W} R_T^\wedge)^{\oplus n} \to (R_T^\wedge \otimes_{(R_T^\wedge)^W} R_T^\wedge)^{\oplus m}$ then $\widehat{\gamma}(M)$ is the cokernel of the corresponding map $(\hTil_{w_0})^{\oplus n} \to (\hTil_{w_0})^{\oplus m}$. From this description and the fact that the functor $\pH^0 \circ \pi_\dag$ is right exact (see Corollary~\ref{cor:pidag-exact}) we see that for any $M$ in $\Mof(R_T^\wedge \otimes_{(R_T^\wedge)^W} R_T^\wedge)$ we have
\[
\pH^0(\pi_\dag(\widehat{\gamma}(M))) \cong \gamma \bigl( \bk \otimes_{R_T^\wedge} M \bigr).
\]
Moreover, if $\hTil$ is a tilting object in $\hscO$, under this identification and that in~\eqref{eqn:V-hV}, applying $\pH^0 \circ \pi_\dag$ to the adjunction morphism
\begin{equation}
\label{eqn:adjunction-hV}
\widehat{\gamma}(\widehat{\mathbb{V}}(\hTil)) \to \hTil
\end{equation}
we recover the adjunction morphism~\eqref{eqn:adjunction-V} for $\mathscr{T}=\pi_\dag(\hTil)$. Since the latter map is known to be surjective, this shows that the cokernel of~\eqref{eqn:adjunction-hV} is killed by $\pH^0 \circ \pi_\dag$ hence, in view of Lemma~\ref{lem:perv-pi}\eqref{it:pidag-conservative-perv}, that this morphism is surjective.

Let now $\widehat{\mathscr{K}}$ be the kernel of~\eqref{eqn:adjunction-hV}. To conclude the proof, it now suffices to prove that
$\Hom_{\hscO}(\widehat{\mathscr{K}}, \hTil')=0$
for any tilting object $\hTil'$ in $\hscO$. For this it suffices to prove that $\Hom_{\hscO}(\widehat{\mathscr{K}},\hDel_w)=0$ for any $w \in W$. And finally, by the description of morphisms as in~\eqref{eqn:Hom-hD} and since each local system $\mathscr{L}_{A,n}$ is an extension of copies of the trivial local system, for this it suffices to prove that
\[
\Hom_{\hscO}(\widehat{\mathscr{K}},\pi^\dag \Delta_w)=0
\]
for any $w \in W$.

By adjunction and right-exactness of $\pi_\dag$ (see Corollary~\ref{cor:pidag-exact}), we have
\[
\Hom_{\hscO}(\widehat{\mathscr{K}},\pi^\dag \Delta_w) \cong \Hom_{\Db_{U}(Y,\bk)}(\pi_\dag \widehat{\mathscr{K}}, \Delta_w) \cong \Hom_{\scO}(\pH^0(\pi_\dag \widehat{\mathscr{K}}), \Delta_w).
\]
Now the remarks above (and the observation that $\pH^{-1} (\pi_\dag \hTil)=0$) show that $\pH^0(\pi_\dag \widehat{\mathscr{K}})$ is the kernel of the morphism~\eqref{eqn:adjunction-V} for $\mathscr{T}=\pi_\dag(\hTil)$. In particular this object does not admit $\IC_e$ as a composition factor; by Lemma~\ref{lem:soc-top-Delta-nabla} this implies that $\Hom_{\scO}(\pH^0(\pi_\dag \widehat{\mathscr{K}}), \Delta_w)=0$, and finishes the proof.
\end{proof}

We also observe the following consequence of Proposition~\ref{prop:ff}, following~\cite{bbm}.

\begin{cor}
\label{cor:ff}
For any projective perverse sheaves $\mathscr{P}$, $\mathscr{P}'$ in $\scO$, the functor $\mathbb{V}$ induces an isomorphism
\[
\Hom_{\scO}(\mathscr{P},\mathscr{P}') \simto \Hom_{R_T/(R_T)^W_+} \bigl( \mathbb{V}(\mathscr{P}), \mathbb{V}(\mathscr{P}') \bigr).
\]
\end{cor}

\begin{proof}
It is well known that the functor
\[
(-) \star^B \Delta_{w_0} : \Db_{U}(Y,\bk) \to \Db_{U}(Y,\bk)
\]
is an equivalence of triangulated categories which restricts to an equivalence between tilting and projective objects in $\scO$; see~\cite{bbm} or~\cite{modrap1}. The inverse equivalence is the functor
\[
(-) \star^B \nabla_{w_0} : \Db_{U}(Y,\bk) \to \Db_{U}(Y,\bk).
\]
Therefore we have
\[
\mathbb{V}(\mathscr{P}) = \Hom(\mathscr{T}_{w_0},\mathscr{P}) \cong \Hom(\mathscr{T}_{w_0} \star^B \nabla_{w_0}, \mathscr{P} \star^B \nabla_{w_0}) \cong \mathbb{V}(\mathscr{P} \star^B \nabla_{w_0})
\]
since $\mathscr{T}_{w_0} \star^B \nabla_{w_0} \cong \mathscr{T}_{w_0}$; see~\eqref{eqn:Tw0-convolution-2}. In other words, we have constructed an isomorphism between the restriction of $\mathbb{V}$ to the subcategory $\mathsf{Proj}(\scO)$ of projective objects in $\scO$ and the composition
\[
\mathsf{Proj}(\scO) \xrightarrow[\sim]{(-) \star^B \nabla_{w_0}} \mathsf{Tilt}(\scO) \xrightarrow{\mathbb{V}} \Mof(R_T/(R_T)^W_+),
\]
where $\mathsf{Tilt}(\scO)$ is the category of tilting objects in $\scO$.
Hence the desired claim follows from Proposition~\ref{prop:ff}.
\end{proof}

\subsection{Image of \texorpdfstring{$\hTil_s$}{Ts}}
\label{ss:image-Ts}

Let us fix $s \in S$. Recall (see Remark~\ref{rmk:hTils-canonical}) that since we have chosen a representative for $\hTil_{w_0}$ we have a canonical representative for $\hTil_s$.
In the following lemma, we denote by $(R_T^\wedge)^s$ the $s$-invariants in $R_T^\wedge$.

\begin{lem}
\label{lem:V-Ts}
There exists a canonical isomorphism
\[
R_T^\wedge \otimes_{(R_T^\wedge)^s} R_T^\wedge \simto \widehat{\mathbb{V}}(\hTil_s).
\]
\end{lem}

\begin{proof}
Recall that $\hTil_s = (\overline{\jmath}_s)_* \overline{\jmath}_s^* \hTil_{w_0}$; hence by adjunction we have
\[
\widehat{\mathbb{V}}(\hTil_s) = \Hom(\hTil_{w_0}, \hTil_s) \cong \End(\hTil_s).
\]
By Proposition~\ref{prop:mon-tilting} (applied to the Levi subgroup of $G$ containing $T$ associated with $s$) the morphism
\[
R_T^\wedge \otimes_\bk R_T^\wedge \to \End(\hTil_s)
\]
induced by monodromy factors through a morphism $R_T^\wedge \otimes_{(R_T^\wedge)^s} R_T^\wedge \to \End(\hTil_s)$, and by Corollary~\ref{cor:mon-Ts} this morphism is surjective. Now under our assumptions $R_T^\wedge$ is free of rank $2$ over $(R_T^\wedge)^s$. (In fact, if $\delta^\vee \in X_*(T)$ is a cocharacter such that $\langle \delta^\vee, \alpha_s \rangle = 1$, then $\{1,\delta^\vee\}$ is a basis of this module.) Hence $R_T^\wedge \otimes_{(R_T^\wedge)^s} R_T^\wedge$ is free of rank $2$ as an $R_T^\wedge$-module. Since $\End(\hTil_s)$ also has this property (see Lemma~\ref{lem:tiltings}\eqref{it:tilting-Hom}), this morphism must be an isomorphism.
\end{proof}

\subsection{Monoidal structure -- \'etale setting}
\label{ss:monoidal-structure}

In this subsection we consider the setting of Section~\ref{sec:etale}. In this case, in view of Lemma~\ref{lem:Tw0-etale} we have a canonical choice for the object $\hTil_{w_0}$; this is the choice we consider. 

We will denote by
\[
\hTet_{\mathbf{U}}(\mathbf{X} \quot \bT,\bk)
\]
the category of tilting perverse sheaves in $\hDet_{\mathbf{U}}(\mathbf{X} \quot \bT,\bk)$. By Remark~\ref{rmk:convolution-tilting} this subcategory is stable under the convolution product $\hatstar$; moreover, it contains the unit object $\hDel_e$ (see Lemma~\ref{lem:unit-convolution}); hence it has a natural structure of monoidal category.

In the following proposition, we consider the monoidal structure on the category $\Mof(R_\bT^\wedge \otimes_{(R_\bT^\wedge)^W} R_\bT^\wedge)$ given by $(M,N) \mapsto M \otimes_{R_\bT^\wedge} N$, where the tensor product is defined with respect to the action of the second copy of $R_\bT^\wedge$ on $M$ and the first copy on $N$, and the action of $R_\bT^\wedge \otimes_{(R_\bT^\wedge)^W} R_\bT^\wedge$ on $M \otimes_{R_\bT^\wedge} N$ is induced by the action of the first copy of $R_\bT^\wedge$ on $M$ and the second copy of $R_\bT^\wedge$ on $N$.

\begin{prop}
\label{prop:monoidal}
The functor $\widehat{\mathbb{V}} : \hTet_{\mathbf{U}}(\mathbf{X} \quot \bT,\bk) \to \Mof(R_{\bT}^\wedge \otimes_{(R_{\bT}^\wedge)^W} R_{\bT}^\wedge)$ has a canonical monoidal structure.
\end{prop}

\begin{proof}
Recall from~\S\ref{ss:Whittaker} the category $\hDet_\Whit(\mathbf{X} \quot \bT,\bk)$, the object $\hDel_\chi$, the equivalence
\[
\Upsilon : \Db \Mod^{\mathrm{fg}}(R_{\bT}^\wedge) \simto \hDet_\Whit(\mathbf{X} \quot \bT,\bk)
\]
 from~\eqref{eqn:equiv-hDWhit}, and the functor
$\Av_\chi : \hDet_{\mathbf{U}}(\mathbf{X} \quot \bT,\bk) \to \hDet_{\Whit}(\mathbf{X} \quot \bT,\bk)$. 
We also have a right action of the monoidal category $\hDet_{\mathbf{U}}(\mathbf{X} \quot \bT,\bk)$ on $\hDet_{\Whit}(\mathbf{X} \quot \bT,\bk)$, denoted again $\hatstar$.

Let us denote by $\hTet_\Whit(\mathbf{X} \quot \bT,\bk)$ the full subcategory of $\hDet_\Whit(\mathbf{X} \quot \bT,\bk)$ whose objects are the direct sums of copies of $\hDel_\chi$, or equivalently the image under $\Upsilon$ of the category of free $R_{\bT}^\wedge$-modules. We claim that, for $\hTil$ in $\hTet_{\mathbf{U}}(\mathbf{X} \quot \bT,\bk)$, the functor
\begin{equation}
\label{eqn:conv-hTil-Whit}
(-) \hatstar \hTil : \hDet_{\Whit}(\mathbf{X} \quot \bT,\bk) \to \hDet_{\Whit}(\mathbf{X} \quot \bT,\bk)
\end{equation}
stabilizes the subcategory $\hTet_{\Whit}(\mathbf{X} \quot \bT,\bk)$. In fact, to prove this it suffices to prove that $\hDel_\chi \hatstar \hTil$ belongs to $\hTet_{\Whit}(\mathbf{X} \quot \bT,\bk)$. But we have $\hDel_\chi \hatstar \hTil \cong \Av_\chi(\hTil)$, and
\begin{multline*}
\mathsf{H}^\bullet(\Upsilon^{-1}(\Av_\chi(\hTil))) \cong \Hom^\bullet_{\hDet_\Whit(\mathbf{X} \quott \bT,\bk)}(\hDel_\chi, \Av_\chi(\hTil)) \\
\cong \Hom^\bullet_{\hDet_{\mathbf{U}}(\mathbf{X} \quott \bT,\bk)}(\Av^{\bU}_!(\hDel_\chi), \hTil) \cong \Hom^\bullet_{\hDet_{\mathbf{U}}(\mathbf{X} \quott \bT,\bk)}(\hTil_{w_0}, \hTil)
\end{multline*}
where the second isomorphism uses adjunction, and the third one uses Lemma~\ref{lem:Tw0-etale}. Now the right-hand side is concentrated in degree $0$, and free over $R_{\bT}^\wedge$ by Lem\-ma~\ref{lem:tiltings}\eqref{it:tilting-Hom}. Hence $\Av_\chi(\hTil)$ is indeed a direct sum of copies of $\hDel_\chi$.

The claim we have just proved shows in particular that the functor~\eqref{eqn:conv-hTil-Whit} is right exact for the perverse t-structure. Hence the functor
\[
\Mof(R_{\bT}^\wedge) \xrightarrow[\sim]{\Upsilon} \hP^{\mathrm{et}}_\Whit(\mathbf{X} \quot \bT,\bk) \xrightarrow{\pH^0( - \hatstar \hTil)} \hP^{\mathrm{et}}_\Whit(\mathbf{X} \quot \bT,\bk) \xrightarrow[\sim]{\Upsilon^{-1}} \Mof(R_{\bT}^\wedge)
\]
is right exact, and therefore representable by the $R^\wedge_\bT$-bimodule
\[
\Upsilon^{-1} \Bigl( \pH^0 \bigl( \Upsilon(R_{\bT}^\wedge) \hatstar \hTil \bigr) \Bigr) = \widehat{\mathbb{V}}(\hTil).
\]

In the case $\hTil = \hDel_e$, since the functor $(-) \hatstar \hDel_e$ is canonically isomorphic to the identity functor, we must have a canonical isomorphism $\widehat{\mathbb{V}}(\hDel_e) \cong R_{\bT}^\wedge$ (which can of course also been seen directly). And, if $\hTil,\hTil'$ belong to $\hTet_{\mathbf{U}}(\mathbf{X} \quot \bT,\bk)$, since the functor constructed as above from $\hTil \hatstar \hTil'$ is canonically isomorphic to the composition of the functors associated with $\hTil$ and with $\hTil'$ respectively, we obtain a canonical isomorphism
\[
\widehat{\mathbb{V}}(\hTil \hatstar \hTil') \cong \widehat{\mathbb{V}}(\hTil) \otimes_{R_{\bT}^\wedge} \widehat{\mathbb{V}}(\hTil').
\]
It is easy to check that these isomorphisms are compatible with the associativity and unit constraints, hence define a monoidal structure on $\widehat{\mathbb{V}}$.
\end{proof}

\subsection{Monoidal structure -- classical setting}
\label{ss:monoidal-classical}

In this subsection we consider the ``classical'' setting of Sections~\ref{sec:flag-tilting}--\ref{sec:main-thm}. Here we do not have (at present) a counterpart of the Whittaker category; but an analogue of Proposition~\ref{prop:monoidal} can be obtained from general principles. For this we have to assume that $\bk$ contains a primitive $p$-th root of unity for some prime number $p \neq \ell$; we fix a choice of $p$ and of a primitive root.

\begin{prop}
\label{prop:monoidal-classical}
There exists a choice of object $\hTil_{w_0}$ such that
the functor $\widehat{\mathbb{V}} : \hT_{U}(X \quot T,\bk) \to \Mof(R_T^\wedge \otimes_{(R_T^\wedge)^W} R_T^\wedge)$ admits a monoidal structure.
\end{prop}

\begin{proof}
We follow the procedure of~\cite[\S 6.1]{bbd} to deduce the result in the classical topology (over $\C$) from its \'etale counterpart (over an algebraically closed field of characteristic $p$).

Let $\bG_\Z$ be split connected reductive group over $\Z$ such that $\mathrm{Spec}(\C) \times_{\mathrm{Spec}(\Z)} \bG_\Z$ is isomorphic to $G$, and let $\bB_\Z$ be a Borel subgroup of $\bG_\Z$ and $\bT_\Z \subset \bB_\Z$ be a (split) maximal torus; then we can assume that $B=\mathrm{Spec}(\C) \times_{\mathrm{Spec}(\Z)} \bB_\Z$ and $T=\mathrm{Spec}(\C) \times_{\mathrm{Spec}(\Z)} \bT_\Z$. Let also $\bU_\Z$ be the unipotent radical of $\bB_\Z$, so that $U=\mathrm{Spec}(\C) \times_{\mathrm{Spec}(\Z)} \bU_\Z$; then we can set $\bX_\Z := \bG_\Z / \bU_\Z$, $\bY_\Z := \bG_\Z / \bB_\Z$, which provides versions of $X$ and $Y$ over $\Z$. We set $\bX_\C := \mathrm{Spec}(\C) \times_{\mathrm{Spec}(\Z)} \bX_\Z$, $\bY_\C := \mathrm{Spec}(\C) \times_{\mathrm{Spec}(\Z)} \bY_\Z$; of course these varieties coincide with $X$ and $Y$, but we change notation to emphasize the fact that we now consider them as schemes (with the Zariski topology) rather than topological spaces (with the classical topology). If $\bU_\C = \mathrm{Spec}(\C) \times_{\mathrm{Spec}(\Z)} \bU_\Z$ and $\bT_\C = \mathrm{Spec}(\C) \times_{\mathrm{Spec}(\Z)} \bT_\Z$, we can consider the categories $\Dbet_{\bU_\C}(\bY_\C, \bk)$ and $\Dbet_{\bU_\C}(\bX_\C \quot \bT_\C, \bk)$ defined using \'etale sheaves (but now over a complex scheme) as in Section~\ref{sec:etale}. The general results recalled in~\cite[\S 6.1.2]{bbd} provide canonical equivalences of categories
\[
\Dbet_{\bU_\C}(\bY_\C, \bk) \cong \Db_{U}(Y, \bk), \qquad \Dbet_{\bU_\C}(\bX_\C \quot \bT_\C, \bk) \cong \Db_{U}(X \quot T, \bk)
\]
which commute (in the obvious sense) with pullback and pushforward functors.

Now, choose an algebraically closed field $\F$ whose characteristic is $p$, and a strictly henselian discrete valuation ring $\mathfrak{R}\subset \C$ whose residue field is $\F$. Then we can consider the base changes of $\bG_\Z$, $\bB_\Z$, etc. to $\mathfrak{R}$ or $\F$, which we will denote by the same letter with a subscript $\mathfrak{R}$ or $\F$. We can consider the versions of the categories considered above for $\bX_{\mathfrak{R}}$ and $\bY_{\mathfrak{R}}$ instead of $\bX_\C$ and $\bY_\C$; the results explained in~\cite[\S\S 6.1.8--6.1.9]{bbd} (see also~\cite[Corollary~VI.4.20 and Remark~VI.4.21]{milne}) guarantee that pullback along the natural morphisms
\[
\xymatrix{
\bY_\C \ar[r] & \bY_{\mathfrak{R}} & \bY_\F \ar[l]
}
\quad \text{and} \quad
\xymatrix{
\bX_\C \ar[r] & \bX_{\mathfrak{R}} & \bX_\F \ar[l]
}
\]
induce equivalences of triangulated categories
\[
\xymatrix{
\Dbet_{\bU_\C}(\bY_\C, \bk) & \Dbet_{\bU_{\mathfrak{R}}}(\bY_{\mathfrak{R}}, \bk) \ar[l]_-{\sim} \ar[r]^-{\sim} & \Dbet_{\bU_\F}(\bY_\F, \bk)
}
\]
and
\[
\xymatrix{
\Dbet_{\bU_\C}(\bX_\C \quot \bT_\C, \bk) & \Dbet_{\bU_{\mathfrak{R}}}(\bX_{\mathfrak{R}} \quot \bT_{\mathfrak{R}}, \bk) \ar[l]_-{\sim} \ar[r]^-{\sim} & \Dbet_{\bU_\F}(\bX_\F \quot \bT_\F, \bk).
}
\]

Combining these two constructions we obtain an equivalence of categories
\begin{equation}
\label{eqn:equiv-hT-etale-classical}
\hT_{U}(X \quot T,\bk) \simto \hTet_{\mathbf{U}_\F}(\mathbf{X}_\F \quot \bT_\F,\bk)
\end{equation}
which is easily seen to be monoidal. Let us denote by $\hTil_{w_0}^{\mathrm{et}}$ the object of the category $\hTet_{\mathbf{U}_\F}(\mathbf{X}_\F \quot \bT_\F,\bk)$ considered in~\S\ref{ss:monoidal-structure}; then Proposition~\ref{prop:monoidal} provides us with a coalgebra structure on $\hTil_{w_0}^{\mathrm{et}}$ (in the monoidal category $(\hTet_{\mathbf{U}_\F}(\mathbf{X}_\F \quot \bT_\F,\bk), \hatstar)$). If we
choose the object $\hTil_{w_0}$ as the inverse image of $\hTil_{w_0}^{\mathrm{et}}$ under~\eqref{eqn:equiv-hT-etale-classical}, then the coalgebra structure on $\hTil_{w_0}^{\mathrm{et}}$ induces a coalgebra structure on $\hTil_{w_0}$.
Given such a structure, it is not difficult (see e.g.~\cite[Proposition~4.6.4 and its proof]{by}) to construct a monoidal structure on the associated functor $\widehat{\mathbb{V}}$.
\end{proof}

\begin{rmk}
One can obtain a result weaker than Proposition~\ref{prop:monoidal-classical} without using the comparison with \'etale sheaves. Namely, choose an identification $(i_e)_* i_e^* \hTil_{w_0} \cong \hDel_e$. Then by adjunction we deduce a morphism $\xi : \hTil_{w_0} \to \hDel_e$, which itself induces a morphism
\[
\xi \hatstar \xi : \hTil_{w_0} \hatstar \hTil_{w_0} \to \hDel_e \hatstar \hDel_e = \hDel_e.
\]
One can show (following e.g.~the ideas in~\cite[Proof of Proposition~4.6.4]{by}) that there exists a morphism $\eta : \hTil_{w_0} \to \hTil_{w_0} \hatstar \hTil_{w_0}$ which makes the diagram
\[
\xymatrix@R=0.5cm{
\hTil_{w_0} \ar[rr]^-{\eta} \ar[rd]_-{\xi} && \hTil_{w_0} \hatstar \hTil_{w_0} \ar[ld]^-{\xi \hatstar \xi} \\
& \hDel_e &
}
\]
commutative, and that moreover for any such $\eta$ the morphism of bifunctors
\[
\widehat{\mathbb{V}}(-) \otimes_{R_T^\wedge} \widehat{\mathbb{V}}(-) \to
\widehat{\mathbb{V}}(- \hatstar -)
\]
sending $f \otimes g$ to $(f \hatstar g) \circ \eta$ is an isomorphism of functors. However, to make sure that this isomorphism induces a monoidal structure, we would have to choose $\eta$ such that $(\eta \hatstar \id) \circ \eta = (\id \hatstar \eta) \circ \eta$. We do not know how to make such a choice.
\end{rmk}

\subsection{Soergel theory}
\label{ss:soergel-theory}

In this subsection we work either in the classical or in the \'etale setting (but use the notation from Sections~\ref{sec:flag-tilting}--\ref{sec:main-thm}).

With Proposition~\ref{prop:ff}, Lemma~\ref{lem:V-Ts} and Proposition~\ref{prop:monoidal} (or Proposition~\ref{prop:monoidal-classical}) at hand, one can obtain a very explicit description of the categories $\hT_{U}(X \quot T,\bk)$ and $\Tilt(\scO)$, as follows.

\begin{thm}\phantomsection
\label{thm:soergel-theory}
\begin{enumerate}
\item
\label{it:soergel-theory-hT}
The functor $\widehat{\mathbb{V}}$ induces an equivalence of monoidal categories between $\hT_{U}(X \quot T,\bk)$ and the full subcategory $\mathsf{SMof}(R_T^\wedge \otimes_{(R_T^\wedge)^W} R_T^\wedge)$ of $\Mof(R_T^\wedge \otimes_{(R_T^\wedge)^W} R_T^\wedge)$ generated under direct sums, direct summands, and tensor products, by the objects $R_T^\wedge$ and $R_T^\wedge \otimes_{(R_T^\wedge)^s} R_T^\wedge$ with $s \in S$.
\item
The functor $\mathbb{V}$ induces an equivalence of categories between $\Tilt(\scO)$ and the full subcategory $\mathsf{SMof}(R_T^\wedge)$ of $\Mof(R_T^\wedge)$ generated under direct sums, direct summands, and application of functors $R_T^\wedge \otimes_{(R_T^\wedge)^s} -$ (with $s \in S$) by the trivial module $\bk$.
\item
These equivalences are compatible in the sense that the diagram
\[
\xymatrix@C=1.5cm{
\hT_{U}(X \quot T,\bk) \ar[r]^-{\widehat{\mathbb{V}}}_-{\sim} \ar[d]_-{\pi_\dag} & \mathsf{SMof}(R_T^\wedge \otimes_{(R_T^\wedge)^W} R_T^\wedge) \ar[d]^-{ - \otimes_{R_T^\wedge} \bk} \\
\Tilt(\scO) \ar[r]^-{\mathbb{V}}_-{\sim} & \mathsf{SMof}(R_T^\wedge)
}
\]
commutes (up to canonical isomorphism) and that the convolution action of $\hT_{U}(X \quot T,\bk)$ on $\Tilt(\scO)$ identifies with the action induced by the action of $\Mof(R_T^\wedge \otimes_{(R_T^\wedge)^W} R_T^\wedge)$ on $\Mof(R_T^\wedge)$ by tensor product over $R_T^\wedge$.
\end{enumerate}
\end{thm}

\begin{proof}
The theorem follows from the results quoted above and Remark~\ref{rmk:convolution-tilting}.
\end{proof}

One can also state similar results for triangulated categories.

\begin{thm}
\label{thm:soergel-triang}
There exist canonical equivalences of monoidal triangulated categories
\begin{align*}
\Kb \mathsf{SMof}(R_T^\wedge \otimes_{(R_T^\wedge)^W} R_T^\wedge) &\simto \hD_U(X \quot T, \bk), \\
\Kb \mathsf{SMof}(R_T^\wedge) &\simto \Db_U(Y,\bk).
\end{align*}
These equivalences are compatible in a sense similar to that in Theorem~{\rm \ref{thm:soergel-theory}}.
\end{thm}

\begin{proof}
The first equivalence follows from Proposition~\ref{prop:realization-equiv} and Theorem~\ref{thm:soergel-theory}\eqref{it:soergel-theory-hT}. (The fact that the equivalence of Proposition~\ref{prop:realization-equiv} is monoidal in our setting follows from standard arguments, see~\cite[Lemma~A.7.1]{beilinson} or~\cite[Proposition~2.3]{amrw}.) The second equivalence, and their compatibilities, follow from similar arguments.
\end{proof}

\begin{rmk}
\begin{enumerate}
\item
Using Remark~\ref{rmk:DbS-hDS}, from the description of the category $\hD_U(X \quot T, \bk)$ given in Theorem~\ref{thm:soergel-triang} one can deduce a description of the category $\Db_U(X \quot T, \bk)$ in algebraic terms.
\item
Using Theorem~\ref{thm:soergel-theory} and the known structure of the additive categories $\hT_{U}(X \quot T,\bk)$ and~$\Tilt(\scO)$ one obtains some sort of ``multiplicative variant'' of the theory of Soergel modules and bimodules (see~\cite{soergel-bim}) in our present setting. It might be interesting to understand if such a theory can be developed algebraically, and in bigger generality.
\end{enumerate}
\end{rmk}

Finally, following~\cite{bbm}, from our results we deduce the following description of the category $\scO$.
Here, for $\underline{w}=(s_1, \cdots, s_r)$ a sequence of elements of $S$, we set
\[
\mathsf{B}(\underline{w}) = R_T^\wedge \otimes_{(R_T^\wedge)^{s_1}} \cdots \otimes_{(R_T^\wedge)^{s_{r-1}}} R_T^\wedge \otimes_{(R_T^\wedge)^{s_r}} \bk.
\]

\begin{thm}
Choose, for any $w \in W$, a reduced expression $\underline{w}$ for $w$. Then there exists an equivalence of categories between $\scO$ and the category $\Mof(A)$, where
\[
A = \left( \End_{R_T^\wedge}\Bigl( \bigoplus_{w \in W} \mathsf{B}(\underline{w}) \Bigr) \right)^{\mathrm{op}}.
\]
\end{thm}

\begin{proof}
For $\underline{v}=(s_1, \cdots, s_r)$ a sequence of elements of $S$, we set
\[
\mathscr{T}(\underline{v}) = \hTil_{s_1} \hatstar \cdots \hatstar \hTil_{s_r} \hatstar \Delta_e.
\]
Then by Corollary~\ref{cor:ff} and its proof, the object
\[
\mathscr{P} := \bigoplus_{w \in W} \mathscr{T}(\underline{w}) \star^B \Delta_{w_0}
\]
is a projective generator of $\scO$, and we have $\End(\mathscr{P}) \cong A^{\mathrm{op}}$. Then the claim follows from general and well-known result, see e.g.~\cite[Exercise on p.~55]{bass}.
\end{proof}

\section{Erratum to~\texorpdfstring{\cite{ab}}{AB}}
\label{sec:erratum}

In this section we use the above results to correct an error found in the proof of~\cite[Lemma~5]{ab}\footnote{Namely, it is claimed in this proof that the complex denoted ``$C$'' is concentrated in positive perverse degrees. But the arguments given there only imply that its \emph{negative} perverse cohomology objects vanish.} and generalize that statement to arbitrary coefficients. The new proof below follows the strategy suggested in~\cite[Remark~3]{ab}.
The statement of~\cite[Lemma~5]{ab} involves an \emph{affine} flag variety but it readily reduces to Lemma~\ref{lem:ab} below restricted to the special case
of characteristic zero coefficients.
 
As in Section~\ref{sec:etale} we consider a connected reductive algebraic group $\bG$ over an algebraically closed field $\F$ of characteristic $p \neq \ell$, and choose a Borel subgroup $\bB \subset \bG$ and a maximal torus $\bT \subset \bB$. Fixing the same data as in~\S\ref{ss:Whittaker} we can consider the standard perverse sheaf $\Delta_\chi:=\Av_\chi(\Delta_e)$. (Note that the natural morphism $\Delta_\chi \to \nabla_\chi:=\Av_\chi(\nabla_e)$ is an isomorphism.) In~\S\ref{ss:Whittaker} we have considered the averaging functors $\Av^{\bU}_!$ and $\Av^{\bU}_*$. We can similarly define the functors
\[
\Av^{\bB}_! := (a_{\bB})_! \bigl(\underline{\bk}_{\bB} \boxtimes - \bigr) [\dim \bB], \quad \Av^{\bB}_* := (a_{\bB})_* \bigl(\underline{\bk}_{\bB} \boxtimes - \bigr) [\dim \bB],
\]
from $\Dbet_\Whit(\bY,\bk)$ to the $\bB$-equivariant derived category $\Dbet_\bB(\bY,\bk)$, where $a_\bB : \bB \times \bY \to \bY$ is the action morphism.

In the following lemma, we denote by $\Phi \subset X^*(\bT)$ the root system of $(\bG,\bT)$, and by $\Z\Phi$ the lattice generated by $\Phi$.

\begin{lem}
\label{lem:ab}
The $\bB$-equivariant complex $\Av^\bB_*(\Delta_\chi)$ is concentrated in perverse degrees $\geq -\dim(\bT)$. Moreover, if
$X^*(\bT) / \Z \Phi$ has no torsion then we have
\[
\pH^{-\dim \bT} \bigl( \Av^\bB_*(\Delta_\chi) \bigr) \cong \Delta_{w_0}.
\]
\end{lem}

\begin{proof}
Using Verdier duality, this statement is equivalent to the fact that $\Av^\bB_!(\Delta_\chi)$ is concentrated in perverse degrees $\leq \dim(\bT)$, and that if
$X^*(\bT) / \Z \Phi$ has no torsion then we have
$\pH^{\dim \bT} \bigl( \Av^\bB_!(\Delta_\chi) \bigr) \cong \nabla_{w_0}$. This is the statement we will actually prove.

We have
\[
\Av^\bB_! \cong {}^! \hspace{-0.5pt} \mathrm{Ind}_{\bU}^{\bB} \circ \Av^\bU_!,
\]
where ${}^! \hspace{-0.5pt} \mathrm{Ind}_{\bU}^{\bB} : \Dbet_\bU(\bY,\bk) \to \Dbet_\bB(\bY,\bk)$ is the functor sending $\mathscr{F}$ to 
\[
(a'_\bB)_! (\underline{\bk}_{\bB/\bU} \, \widetilde{\boxtimes} \, \mathscr{F}) [\dim(\bB/\bU)].
\]
(Here, $a'_\bB : \bB \times^\bU \bY \to \bY$ is the natural map, and $\widetilde{\boxtimes}$ is the twisted external product.) Using Lemma~\ref{lem:Tw0-etale}, we deduce that
\[
\Av^\bB_!(\Delta_\chi) \cong {}^! \hspace{-0.5pt} \mathrm{Ind}_{\bU}^{\bB}(\mathscr{T}_{w_0}).
\]
It is clear that for any $\bB$-equivariant perverse sheaf $\mathscr{F}$ on $\bY$, the complex ${}^! \hspace{-1pt} \mathrm{Ind}_{\bU}^{\bB}(\mathscr{F})$ is concentrated in perverse degrees between $0$ and $\dim(\bT)$. Hence the same claim holds for any extension of such perverse sheaves, i.e.~for any $\bU$-equivariant perverse sheaf; thus the first claim of the lemma is proved. Now the functor ${}^! \hspace{-0.5pt} \mathrm{Ind}_{\bU}^{\bB}$ is left adjoint to $\mathrm{For}^{\bB}_{\bU}[\dim(\bB/\bU)]$, where $\mathrm{For}^{\bB}_{\bU} : \Dbet_\bB(\bY,\bk) \to \Dbet_\bU(\bY,\bk)$ is the forgetful functor. Using this fact, it is not difficult to check that for any $\bU$-equivariant perverse sheaf $\mathscr{F}$ on $\bY$, the perverse sheaf
\[
\pH^{\dim \bT} \bigl( {}^! \hspace{-0.5pt} \mathrm{Ind}_{\bU}^{\bB}(\mathscr{F}) \bigr)
\]
is characterized as the largest $\bB$-equivariant quotient of $\mathscr{F}$.

To conclude the proof, it remains to prove that if
$X^*(\bT) / \Z \Phi$ has no torsion then $\nabla_{w_0}$ is the largest $\bB$-equivariant quotient of $\mathscr{T}_{w_0}$. Now $\mathscr{T}_{w_0}$ has a costandard filtration, whose last term is $\nabla_{w_0}$; therefore there exists a surjection $\mathscr{T}_{w_0} \twoheadrightarrow \nabla_{w_0}$ (which is unique up to scalar). Since $\nabla_{w_0}$ is $\bB$-equivariant, we deduce that this map factors as a composition
\[
\mathscr{T}_{w_0} \twoheadrightarrow \pH^{\dim \bT} \bigl( {}^! \hspace{-0.5pt} \mathrm{Ind}_{\bU}^{\bB}(\mathscr{T}_{w_0}) \bigr) \twoheadrightarrow \nabla_{w_0}.
\]
The kernel of the second map here is the image of the kernel of our surjection $\mathscr{T}_{w_0} \twoheadrightarrow \nabla_{w_0}$. Since the latter admits a costandard filtration, in view of Lemma~\ref{cor:sub-quo-filtered}, if the former is nonzero then it admits $\IC_e$ as a composition factor; in other words the vector space 
\[
\Hom\bigl( \mathscr{T}_{w_0}, \pH^{\dim \bT} ( {}^! \hspace{-0.5pt} \mathrm{Ind}_{\bU}^{\bB}(\mathscr{T}_{w_0}) ) \bigr)
\]
has dimension at least $2$. 

On the other hand, we have a surjection
\[
\Hom(\mathscr{T}_{w_0},\mathscr{T}_{w_0}) \twoheadrightarrow \Hom\bigl( \mathscr{T}_{w_0}, \pH^{\dim \bT} ( {}^! \hspace{-0.5pt} \mathrm{Ind}_{\bU}^{\bB}(\mathscr{T}_{w_0}) ) \bigr).
\]
Our assumption on $\bG$ means that the quotient morphism $\bG \twoheadrightarrow \bG/Z(\bG)$ (where $Z(\bG)$ is the center of $\bG$) induces a surjection $X_*(\bT) \twoheadrightarrow X_*(\bT/Z(\bG))$. Applying Corollary~\ref{cor:main} to $\bG/Z(\bG)$ we obtain that monodromy induces a surjection
\[
R_\bT \twoheadrightarrow \Hom(\mathscr{T}_{w_0},\mathscr{T}_{w_0}).
\]
Since $\pH^{\dim \bT} ( {}^! \hspace{-0.5pt} \mathrm{Ind}_{\bU}^{\bB}(\mathscr{T}_{w_0}) )$ is $\bB$-equivariant, the composition
\[
R_\bT \twoheadrightarrow \Hom(\mathscr{T}_{w_0},\mathscr{T}_{w_0}) \twoheadrightarrow \Hom\bigl( \mathscr{T}_{w_0}, \pH^{\dim \bT} ( {}^! \hspace{-0.5pt} \mathrm{Ind}_{\bU}^{\bB}(\mathscr{T}_{w_0}) ) \bigr)
\]
factors through $\varepsilon_\bT$, proving that the rightmost term has dimension at most $1$. 
This condition prevents the kernel of the surjection $\pH^{\dim \bT} \bigl( {}^! \hspace{-0.5pt} \mathrm{Ind}_{\bU}^{\bB}(\mathscr{T}_{w_0}) \bigr) \twoheadrightarrow \nabla_{w_0}$ to be nonzero, which concludes the proof.
\end{proof}

\begin{rmk}
\begin{enumerate}
\item
Using the remarks in~\S\ref{ss:intro-remarks}, one can show that another setting in which the second claim in Lemma~\ref{lem:ab} holds is when $\ell$ is very good for $\bG$ hence, in particular, when $\ell=0$. (Note that under this assumption $X^*(\bT) / \Z \Phi$ has no $\ell$-torsion, see~\cite[\S 2.10]{herpel}.)
\item
Replacing the proof of~\cite[Lemma~5]{ab} by the proof given above, one can check that all the results of~\cite[\S 2]{ab} (hence, in particular,~\cite[Proposition~2]{ab}) extend in a straightforward way to positive-characteristic coefficients.
\end{enumerate}
\end{rmk}


\end{document}